\documentclass[reqno,a4paper,12pt]{amsart}
\usepackage[utf8]{inputenc}
\usepackage{color}
\usepackage{amsmath,amssymb,amsfonts,amsthm,enumerate}
\usepackage{hyperref,mathrsfs, accents}
\usepackage{lmodern}
\usepackage[T1]{fontenc}
\usepackage[active]{srcltx}
\usepackage{epsfig}
\usepackage{graphicx}
\usepackage[noadjust]{cite}
\usepackage{verbatim} 

\newcommand{\Tr}[3]{\ban{\mathcal{#1}_{#2},\nu_{#3}}}
\newcommand{\intE}[1]{\widetilde{#1}^1}
\newcommand{\extE}[1]{\widetilde{#1}^0} 

\newcommand{\intee}[2]{{#1}^{1,#2}}
\newcommand{\extee}[2]{{#1}^{0,#2}}

\newcommand{\redB}[1]{\widetilde{\mathscr{F}_{H}#1}}
\newcommand{\divFs}{\div^{\rm s}\!F} 
\newcommand{\divFa}{\div^{\rm a}\!F} 
\newcommand{\nb}[1]{\nabla_H#1}

\renewcommand{\H}{\mathbb{H}}

\newcommand{\G}{\mathbb{G}}

\newcommand{\N}{\mathbb{N}}

\newcommand{\R}{\mathbb{R}}

\newcommand{\cS}{\mathcal{S}}

\newcommand{\cM}{\mathcal{M}}
\newcommand{\cR}{\mathcal{R}}

\newcommand{\cV}{\mathcal{V}}

\newcommand{\ep}{\varepsilon}
\newcommand{\ph}{\varphi}

\newcommand{\sm}{\setminus}

\newcommand{\lb}{{\big\lbrace}}
\newcommand{\rb}{{\big\rbrace}}

\newcommand{\lls}{\big(}
\newcommand{\rrs}{\big)}

\newcommand{\diam}{\mbox{\rm diam}}

\renewcommand{\exp}{\mbox{\rm exp}\;\!}
\newcommand{\supp}{\mbox{\rm supp}}

\renewcommand{\div}{\mbox{\rm div}}

\newcommand{\dist}{\mbox{dist}}

\newcommand{\Lie}{\mathrm{Lie}}

\newcommand{\q}{\mathrm q}
\newcommand{\m}{\mathrm m}

\newcommand{\QQ}{\mathrm Q}

\newcommand{\bcup}{\bigcup}

\newcommand{\lan}{\langle}
\newcommand{\ran}{\rangle}

\newcommand{\lra}{\longrightarrow}
\newcommand{\der}{\partial}
\newcommand{\eR}{{\overline{\mathbb R}}}

\newcommand{\epty}{\emptyset}

\newcommand{\ds}{\displaystyle}

\newcommand{\beqas}{\begin{eqnarray*}}
\newcommand{\eeqas}{\end{eqnarray*}}
\newcommand{\beqa}{\begin{eqnarray}}
\newcommand{\eeqa}{\end{eqnarray}}
\newcommand{\beq}{\begin{equation}}
\newcommand{\eeq}{\end{equation}}
\newcommand{\bce}{\begin{center}}
\newcommand{\ece}{\end{center}}

\newcommand{\pa}[1]{\left( #1 \right)}               
\newcommand{\set}[1]{\left\{ #1 \right\}}            
\newcommand{\pal}[1]{\left| #1 \right|}            
\newcommand{\ban}[1]{\left\langle  #1 \right\rangle}  

\newcommand{\qs}[1]{\quad\mbox{ #1} \quad}               

\newcommand{\qandq}{\quad\mbox{and}\quad}

\newcommand{\tildef}[1]{\widetilde{ #1}}

\newcommand{\eps}{\varepsilon}
\newcommand{\dime}{\mathrm{dim}} 
\newcommand{\DM}{\mathcal{DM}} 
\newcommand{\Lip}{\mathrm{Lip}} 

\newcommand{\Haus}[1]{{\mathcal H}^{#1}} 
\newcommand{\SHaus}[1]{{\mathcal S}^{#1}} 
\newcommand{\Leb}[1]{{\mathscr L}^{#1}} 
\newcommand{\Per}{{\sf P}} 

\newcommand{\mean}[1]{\,-\hskip-1.08em\int_{#1}} 

\newcommand{\redb}{\mathscr{F}_{H}} 
\newcommand{\redbeu}{\mathscr{F}} 
\newcommand{\mtb}{\partial^{*}_{H}} 
\newcommand{\mtbR}{\partial^{*, \cR}_{H}} 

\newcommand{\weakto}{\rightharpoonup}

\newcommand{\weakstarto}{\overset*\rightharpoonup}

\newcommand{\res}{\mathop{\hbox{\vrule height 7pt width .5pt depth 0pt \vrule height .5pt width 6pt depth 0pt}}\nolimits}

\theoremstyle{plain}
\newtheorem{theorem}{Theorem}[section]
\newtheorem{lemma}[theorem]{Lemma}
\newtheorem{proposition}[theorem]{Proposition}
\newtheorem{corollary}[theorem]{Corollary}

\theoremstyle{definition}
\newtheorem{remark}[theorem]{Remark}
\newtheorem{example}[theorem]{Example}

\newtheorem{definition}[theorem]{Definition}

\catcode`@=11 \@addtoreset{equation}{section} \catcode`@=12

\begin{document}

\title[The Gauss--Green theorem in stratified groups]
{The Gauss--Green theorem in stratified groups}
\author{Giovanni E. Comi}
\address{Giovanni E. Comi, Scuola Normale Superiore, Piazza dei Cavalieri 7, 56126 Pisa, Italy}
\email{giovanni.comi@sns.it}
\author{Valentino Magnani}
\address{Valentino Magnani, Dipartimento di Matematica, Universit\`a di Pisa,
Largo Pontecorvo~5,  I-56127, Pisa, Italy}
\email{magnani@dm.unipi.it}
\date{}

\thanks{The first author has been partially supported by 
the PRIN2015 MIUR Project ``Calcolo delle Variazioni".
The second author acknowledges the support of GNAMPA project 2017 ``Campi vettoriali, superfici e perimetri in geometrie singolari'' and of the University of Pisa, Project PRA 2018 49.}

\begin{abstract}
We lay the foundations for a theory of divergence-measure fields 
in noncommutative stratified nilpotent Lie groups.
Such vector fields form a new family of function spaces, which generalize in a sense the $BV$ fields. They provide the most general setting to 
establish Gauss--Green formulas 
for vector fields of low regularity on sets of finite perimeter. 
We show several properties of divergence-measure fields in
stratified groups, ultimately achieving the related Gauss--Green theorem.
\end{abstract}
\maketitle

\bce
Dedicated to the memory of William P. Ziemer (1934–2017)
\ece

\tableofcontents

\newpage

\section{Introduction}

The Gauss--Green formula is of significant importance in pure and applied Mathematics, as in PDEs, Geometric Measure Theory and Mathematical Physics. In the last two decades, there have been many efforts in extending this formula to very general assumptions, considering `nonsmooth domains' of integration and `weakly regular vector fields'.

The classical Gauss--Green theorem, or the divergence theorem, asserts that, for an open set $\Omega \subset \R^{n}$, a vector field $F \in C^{1}(\Omega; \R^{n})$ and an open subset $E \Subset \Omega$ such that $\bar E$ is a $C^1$ smooth manifold with boundary, there holds
\begin{equation} \label{classical_G_G_eq} \int_{E} \mathrm{div} F \, dx = \int_{\partial E} F \cdot \nu_{E} \, d \mathcal{H}^{n - 1}, \end{equation}
where $\nu_{E}$ is the exterior unit normal to $\partial E$ and $\Haus{n - 1}$ is the $(n - 1)$-dimensional Hausdorff measure.
The class of open sets considered above is too restrictive and this motivated the search for a wider class of integration domains for which the Gauss--Green theorem holds true in a suitable weaker form. Such a research was one of the aims which historically led 
to the theory of functions of bounded variation ($BV$) and sets of finite perimeter, or Caccioppoli sets.
Indeed, an equivalent definition of set of finite perimeter 
requires the validity of a measure theoretic Gauss-Green formula 
restricted to compactly supported smooth vector fields. The subsequent problem
is then finding the geometric structure of its support.

While it is well known that a set of finite perimeter $E$ may have very irregular topological boundary, even with positive Lebesgue measure, it is possible to consider a particular subset of $\partial E$, the reduced boundary $\redbeu E$, on which one can define a unit vector $\nu_{E}$, called measure theoretic exterior unit normal.
In view of De Giorgi's theorem, which shows the rectifiability of the reduced boundary, we know that $|D \chi_{E}| = \Haus{n - 1} \res \redbeu E$ and a first important relaxation of \eqref{classical_G_G_eq} is as follows
\begin{equation} \label{DeGiorgi_Federer_G_G_eq} \int_{E} \mathrm{div} F \, dx = \int_{\redbeu E} F \cdot \nu_{E} \, d \Haus{n - 1},
\end{equation}
where $E \Subset \Omega$ has finite perimeter and $F \in \Lip(\Omega; \R^{n})$.
This result, although important because of the large family of integration domains, is however restricted to a class of integrands with a still relatively strong regularity.

The subsequent generalization of \eqref{DeGiorgi_Federer_G_G_eq} is due to Vol'pert \cite{vol1967spaces} (we also refer to the classical monograph \cite{volpert1975analysis}). 
Thanks to further developments in the $BV$ theory, he was able to consider vector fields in $F \in BV(\Omega; \R^{n}) \cap L^{\infty}(\Omega; \R^{n})$ and sets $E \Subset \Omega$ of finite perimeter, getting the following formulas
\begin{align} \label{Volpert_G_G_1_eq} \mathrm{div}F(E^{1}) & = \int_{\redbeu E} F_{- \nu_{E}} \cdot \nu_{E}\, d\Haus{n -1}, \\
\label{Volpert_G_G_2_eq} \mathrm{div}F(E^{1} \cup \redbeu E) & = \int_{\redbeu E} F_{\nu_{E}} \cdot \nu_{E}\, d\Haus{n -1}, \end{align} 
where $E^{1}$ is the measure theoretic interior of $E$ and $F_{\pm \nu_{E}}$ are the exterior and interior traces of $F$ on $\redbeu E$; that is, the approximate limits of $F$ at $x \in \redbeu E$ restricted to the half spaces $\{ y \in \R^{n} : (y - x) \cdot (\pm \nu_{E}(x)) \ge 0 \}$.
The existence of such traces follows from the fact that any $BV$ function admits a representative which is well defined $\Haus{n - 1}$-a.e.

Not all distributional partial derivatives of a vector field are required to be Radon measures in \eqref{Volpert_G_G_1_eq} and \eqref{Volpert_G_G_2_eq}, since only the divergence appears. Moreover, Gauss--Green formulas with vector fields of lower regularity have proved to be very important in applications, as for instance in hyperbolic conservations laws or in the study of contact interactions in Continuum Physics, \cite{CF1}, \cite{Schuricht}.
All of these facts finally led to the study of $p$-summable divergence-measure fields,
namely $L^{p}$ vector fields whose divergence is a Radon measure. 

Divergence-measure fields provide a natural way to extend the Gauss--Green formula.
The family of $L^p$ summable divergence-measure fields, denoted by $\DM^{p}$, clearly generalizes the vector fields of bounded variation.
It was first introduced by Anzellotti for $p = \infty$ in \cite{Anzellotti_1983},
where he studied different pairings between vector fields and gradients of weakly differentiable functions. Thus, he considered $F \in \DM^{\infty}(\Omega)$ in order to define pairings between bounded vector fields and vector valued measures given by weak gradients of $BV$ functions. One of the main results is \cite[Theorem 1.2]{Anzellotti_1983}, which shows the existence of $L^{\infty}(\partial \Omega)$ traces of the normal component of $\DM^{\infty}(\Omega)$ fields on the boundary of open bounded sets $\Omega$ with Lipschitz boundary. These traces are referred to as {\em normal traces} in the literature. 

After the works of Anzellotti \cite{Anzellotti_1983}, \cite{Anzellotti1983Preprint}, the notion of divergence-measure fields was rediscovered in the early 2000s by many authors, with different purposes.
Chen and Frid proved generalized Gauss--Green formulas for divergence-measure fields on open bounded sets with Lipschitz deformable boundary (see \cite[Theorem 2.2]{CF1} and \cite[Theorem 3.1]{chen2003extended}), motivated by applications to the theory of systems of conservation laws with the Lax entropy condition. The idea of their proof rests on an approximation argument, which allows to obtain a Gauss--Green formula on a family of Lipschitz open bounded sets which approximate the given integration domain. Later, Chen, Torres and Ziemer generalized this method to the case of sets of finite perimeter in order to extend the result in the case $p = \infty$, achieving Gauss--Green formulas for essentially bounded divergence-measure fields and sets of finite perimeter (\cite[Theorem 5.2]{Chen_2009}). Then, Chen and Torres \cite{chen2011structure} applied this theorem to the study of the trace properties of solutions of nonlinear hyperbolic systems of conservation laws. Further studies and simplifications of \cite{Chen_2009} have subsequently appeared
in \cite{comi2017one} and \cite{comi2017locally}.
The alternative approach in \cite{comi2017locally} follows the original idea employed by Vol'pert to prove \eqref{Volpert_G_G_1_eq} and \eqref{Volpert_G_G_2_eq}. A key step is the derivation of a Leibniz rule between a vector field and a characteristic function of a set of finite perimeter.
As a byproduct, this ensures the definition of a suitable generalized notion of normal trace, which is coherent with the definition given by Anzellotti \cite{Anzellotti_1983} and Chen, Torres and Ziemer \cite{Chen_2009}.

The Gauss--Green formula for essentially bounded divergence-measure fields and sets of finite perimeter (\cite[Theorem 3.2]{comi2017locally}) states that, if $F \in \DM^{\infty}(\Omega)$ and $E \Subset \Omega$ is a set of finite perimeter in $\Omega$, then there exist {\em interior and exterior normal traces of $F$ on $\redbeu E$}; that is, $(\mathcal{F}_{i} \cdot \nu_{E}), (\mathcal{F}_{e} \cdot \nu_{E})  \in L^{\infty}(\redbeu E; \Haus{n - 1})$ such that we have\footnote{The absence of the minus sign at the right hand side of these formulas is due to our convention that $\nu_{E}$ denotes the measure theoretic unit {\rm exterior} normal. This differs from the convention in \cite{comi2017locally}.}
\begin{align}\nonumber
\mathrm{div}F(E^{1}) & = \int_{\redbeu E} \mathcal{F}_{i} \cdot \nu_{E} \, d \Haus{n - 1}\;\qandq \; \mathrm{div}F(E^{1} \cup \redbeu E)  = \int_{\redbeu E} \mathcal{F}_{e} \cdot \nu_{E} \, d \Haus{n - 1}. 
\end{align}
In addition, the following trace estimates hold:
\begin{equation} \label{traces_norm_Euclidean_eq}
\|\mathcal{F}_{i} \cdot \nu_{E}\|_{L^{\infty}(\redbeu E; \Haus{n - 1})} \le \|F\|_{\infty, E} \qandq
\|\mathcal{F}_{e} \cdot \nu_{E}\|_{L^{\infty}(\redbeu E; \Haus{n - 1})} \le \|F\|_{\infty, \Omega \setminus E}.
\end{equation}
It is of interest to mention also other methods to prove the Gauss--Green formula, and different applications. Degiovanni, Marzocchi and Musesti in \cite{degiovanni1999cauchy} and Schuricht in \cite{Schuricht} were interested in the existence of a normal trace under weak regularity hypotheses in order to achieve a representation formula for Cauchy fluxes, contact interactions and forces in the context of the foundations of Continuum Mechanics. As is well explained in \cite{Schuricht}, the search for a rigorous proof of Cauchy's stress theorem under weak regularity assumptions is a common theme in much of the literature on divergence-measure fields. The Gauss--Green formulas obtained in \cite{degiovanni1999cauchy} and \cite{Schuricht} are valid for $F \in \DM^{p}(\Omega)$ and $p \ge 1$, even though the domains of integration $E \subset \Omega$ must be taken from a suitable subalgebra of sets of finite perimeter, which are related to the vector field $F$.

\v{S}ilhav\'{y} in \cite{Silhavy} also studied the problem of finding a representation of Cauchy fluxes through traces of suitable divergence-measure fields.
 He gave a detailed description of generalized Gauss--Green formulas for $\DM^{p}(\Omega)$-fields with respect to $p \in [1, \infty]$ and suitable
hypotheses on concentration of $\mathrm{div}F$. In particular, he provided sufficient conditions under which the interior normal traces (and so also the exterior) can be seen as integrable functions with respect to the measure $\Haus{n - 1}$ on the reduced boundary of a set of finite perimeter. We should also note that \v{S}ilhav\'{y} studied the so-called extended divergence-measure fields, already introduced by Chen-Frid in \cite{chen2003extended}, which are vector valued Radon measures whose divergence is still a Radon measure. He proved absolute continuity results and Gauss--Green formulas in \cite{vsilhavy2008divergence} and \cite{vsilhavy2009divergence}.
It is also worth to mention the paper by Ambrosio, Crippa and Maniglia \cite{ACM}, where the authors employed techniques similar to the original ones of Anzellotti and studied a class of essentially bounded divergence-measure fields induced by functions of bounded deformation. Their results were motivated by the aim of extending DiPerna-Lions theory of the transport equation to special vector fields with bounded deformation.
 
In the last decades and more recently,
Anzellotti's pairings and Gauss--Green formulas 
have appeared in several applied and theoretical questions,
as the $1$-Laplace equation, minimal surface equation,  
the obstacle problem for the area functional and theories of
integration to extend the Gauss--Green theorem.
We refer for instance to the works 
\cite{Pfeffer1991GG}, \cite{DePauwPfeffer2004GG}, 
\cite{kawohl2007dirichlet}, \cite{scheven2016bv}, \cite{scheven2017dual}, \cite{scheven2017anzellotti}, \cite{leonardi2017rigidity} and \cite{leonardi2018prescribed}. 
Recently Anzellotti's pairing theory has been extended in \cite{crasta2017anzellotti}, see also \cite{crasta2017pairing} and \cite{crasta2019pairings}, where the authors have also established integration by parts formulas for essentially bounded divergence-measure fields, sets of finite perimeter and essentially bounded scalar functions of bounded variation.
In the context of unbounded divergence-measure fields, we mention the work \cite{chencomitorres}, where new integration by parts formulas are presented and the normal trace functional is studied in relation with the Leibniz rules between the fields and the characteristic functions of sets.

The Gauss--Green formula has been deeply studied also in 
a number of different non-Euclidean contexts, \cite{HarNorGaussGreen1992}, \cite{LyonsYamGG2006}, \cite{Guseynov2016FractGG}.
Related to these results is also the recent study by Z\"ust, on functions of bounded
fractional variation, \cite{ZustFractBV-pr2018}.
Other extensions of the Gauss--Green formula 
appears in the framework of doubling metric spaces satisfying a Poincar\'e inequality, \cite{MarMirShan2015BoundaryMeasGG}.
Through special trace theorems for $BV$ functions 
in Carnot--Carath\'eodory spaces, an integration by parts formula has been 
established also in \cite{Vittone2012},
assuming an intrinsic Lipschitz regularity on the
boundary of the domain of integration.

The main objective of this paper is to establish a Gauss--Green theorem for sets of finite perimeter and divergence-measure vector fields in a family of noncommutative nilpotent Lie groups, called {\em stratified groups} or {\em Carnot groups}.
Such Lie groups equipped with a suitable {\em homogeneous distance} represent infinitely many different types of non-Euclidean geometries, with Hausdorff dimension strictly greater than their topological dimension. Notice that commutative stratified Lie groups
coincide with normed vector spaces, where our results agree with the classical ones.
Stratified groups arise from Harmonic Analysis and PDE, \cite{SteinICM1977}, \cite{Fol75}, and represent an important class of connected and simply connected real nilpotent Lie groups.
They are characterized by a family of 
dilations, along with a left invariant distance that properly scales with dilations,
giving a large class of metric spaces that are not bi-Lipschitz equivalent to each other.

The theory of sets of {\em finite h-perimeter} in stratified groups has known a wide development in the last two decades, especially in relation to topics like De Giorgi's rectifiability, minimal surfaces and differentiation theorems.
We mention for instance some relevant works  
\cite{Ambrosio2001}, \cite{MSC2001}, \cite{franchi2003regular}, \cite{FSSC5}, \cite{CHMY2005MinSurfPseudoherm}, \cite{magnani2006characteristic}, \cite{BASCVMinSurf2007}, \cite{AKLD09TangSpGro}, \cite{MSSC10},
\cite{DGNP2010MinSurfEmbedded}, \cite{MonVitHeightEst2015}, \cite{Mag31}, \cite{MonSteLipApp2017}, \cite{LeDR2017Besicovitch}, only to give a small glimpse of the wide and always expanding literature.
Some basic facts on the theory of sets of finite perimeter and 
$BV$ functions hold in this setting, once these notions are properly
defined. Indeed, other related notions such as reduced boundary and essential boundary, intrinsic rectifiability and differentiability 
can be naturally introduced in this setting, see for instance
\cite{cassano2016some} for a recent overview on these topics
and further references. 

The stratified group $\G$, also called {\em Carnot group}, is always equipped with left invariant {\em horizontal vector fields} $X_1,\ldots,X_\m$, that determine
the directions along which it is possible to differentiate. 
The corresponding distributional derivatives define functions
of bounded h-variation (De\-finition~\ref{def:BV}) and sets of finite h-perimeter (Definition~\ref{def:finitePer}).
We consider {\em divergence-measure horizontal fields},
that are $L^p$-summable sections of the horizontal subbundle $H\Omega$, where $\Omega$ is an open set of $\G$ (Definition~\ref{DMdef}).
Notice that the space of these fields, $\DM^p(H\Omega)$, with $1\le p\le \infty$, 
contains divergence-measure horizontal fields that are not $BV$ even with respect to the group structure (Example~\ref{examples_fiels}).
Nevertheless, the fields in $\DM^{\infty}(H\Omega)$ satisfy a Leibniz rule when multiplied by a function of bounded h-variation, that might be much less regular than a $BV$ function on Euclidean space, see Theorem~\ref{productrule} below.
The loss of Euclidean regularity can be already seen with sets of finite h-perimeter, that are not necessarily of finite perimeter in Euclidean sense, \cite[Example~1]{franchi2001rectifiability}.
Sets of finite h-perimeter are in some sense the largest class of measurable sets for which one can expect existence of normal traces and Gauss--Green formulas
for divergence-measure horizontal fields.

A special aspect of our techniques is a smooth approximation result, obtained by
the noncommutative group convolution (Definition~\ref{d:mollification}).
This is a well known tool in Harmonic Analysis and PDE on homogeneous Lie groups, 
\cite{folland1982hardy}, \cite{Stein93HarmAnalysis}, that has been already used
to study perimeters and BV functions on Heisenberg groups, \cite{SCVit2014GraphsBV}, \cite{MontiVitt2018SetsHperim}.
On the other hand, a number of smooth approximations can be obtained 
in Carnot-Carath\'eodory spaces or sub-Riemannian manifolds 
using the Euclidean convolution, also in relation to Meyers--Serrin theorem and Anzellotti--Giaquinta approximations for Sobolev and BV functions, \cite{franchi1996meyer}, \cite{FSSC2}, \cite{garofalo1996isoperimetric}, \cite{GN1998LipApp}, \cite{Vittone2012}, \cite{AmbMag28Ghe}. The next 
result provides a number of natural properties that are satisfied by the ``correct'' mollified function.

\begin{theorem} \label{commutation convolution derivative} 
Let $f\in BV_{H,\rm loc}(\Omega)$ be such that $|D_Hf|(\Omega)<+\infty$ and
let $\rho \in C_c(B(0,1))$ with $\rho\ge 0$, $\int_{B(0,1)} \rho \, dx=1$ and 
$\rho(x) = \rho(x^{-1})$. Then $\rho_\ep\ast f\in C^1_H(\Omega^\cR_{2\ep})$ and we have 
\begin{align} \label{weak_conv_D_f} 
\nabla_{H} (\rho_{\eps} \ast f) & \weakto D_{H} f \qandq |\rho_{\eps} \ast D_{H} f|  \weakto |D_{H} f|, \\
\label{pointwise_upper_control}
|\nabla_{H}(\rho_{\eps} \ast f)| \,\mu& \le (\rho_{\eps} \ast |D_{H} f|) \,\mu\quad \text{on}\quad \Omega^\cR_{2\ep}
\end{align}
for every $\ep>0$ such that $\Omega^\cR_{2\ep}\neq\epty$.
Finally, the following estimate holds
\begin{equation} \label{total_var_convol_convergence} 
 |\nabla_{H} (\rho_{\eps} \ast f)|(\Omega^\cR_{2\ep}) \le |D_{H} f|(\Omega). \end{equation}
\end{theorem}
We believe this smooth approximation has an independent interest, with possibly different applications. Other basic smoothing results have
been provided in Sections~\ref{sect:Difflocalsmooth} and \ref{sect:BVfunction}.
Notice that the inequality \eqref{pointwise_upper_control} is between measures, where $\mu$ denotes the Haar measure of the group. Here the minimal regularity of the mollifier $\rho$ is necessary in order to have Proposition~\ref{pointwise limit moll} and its consequences.
Indeed, the mollifier $\rho_{\eps}$ can be also built using a homogeneous
distance, that in general may not be smooth even outside the origin.
Theorem~\ref{commutation convolution derivative} plays an important role also in the proof of the Leibniz rule of Theorem \ref{productrule}. 
The noncommutativity of the group convolution makes necessary a right invariant distance $d^\cR$ canonically associated to $d$ \eqref{d:d^R} and the `right inner parts' of an open set $\Omega^\cR_{2\ep}$ \eqref{eq:Omega-R_ep},
that appear in the statement of Theorem~\ref{commutation convolution derivative}.

\begin{theorem}[Approximation and Leibniz rule]\label{productrule}
If $F \in \DM^{\infty}(H \Omega)$ and $g \in L^{\infty}(\Omega)$ 
with $|D_Hg|(\Omega)<+\infty$, then $g F \in \DM^{\infty}(H \Omega)$.
If $\rho \in C_{c}(B(0, 1))$ is nonnegative, 
$\rho(x) = \rho(x^{-1})$ and $\displaystyle \int_{B(0, 1)} \rho \, dx = 1$, then
for any infinitesimal sequence $\tilde\ep_k>0$, setting $g_{\eps} := \rho_{\eps} \ast g$, there exists a subsequence $\ep_k$ such that
$g_{\eps_{k}} \weakstarto \tilde{g}$ in $L^{\infty}(\Omega; |\div F|)$
and $\ban{F, \nabla_{H} g_{\eps_{k}}} \mu \weakto (F, D_{H} g)$ 
in $\mathcal{M}(\Omega)$.
Moreover, the following formula holds
\begin{equation} \label{product rule eq} 
\div (g F) = \tilde{g}\, \div F + (F, D_{H} g),
\end{equation}
where the measure $(F, D_{H} g)$ satisfies
\begin{equation} \label{abs_cont_pairing_infty} |(F, D_{H} g)| \le \|F\|_{\infty, \Omega} |D_{H} g|.
\end{equation} 
Finally, we have the decompositions 
\begin{equation} \label{pairing_decomposition_parts} 
(F, D_{H} g)^{\rm a} \mu  = \ban{F, \nb{g}} \mu \qandq
(F, D_{H} g)^{\rm s}  = (F, D_H^{\rm s} g),
\end{equation}
where $\nabla_Hg$ denotes the approximate differential of $g$.
\end{theorem}
In the Euclidean setting, this Leibniz rule has been established in \cite[Theorem~3.1]{CF1} and \cite[Theorem~2.1]{Frid}. The product rule \eqref{product rule eq} is the starting point of many of our results. For instance, applying this formula to $F \in \DM^{\infty}(H \Omega)$ and $g = \chi_{E}$, for a set of finite h-perimeter $E \Subset \Omega$, and using Lemma~\ref{DMcomptsupp},
one is led to a first embryonic Gauss--Green formula. Here the pairing $(F,D_H\chi_E)$ has still to be related
to suitable notions of normal trace. Indeed, the interior and exterior normal traces $ \Tr{F}{i}{E}$ and $ \Tr{F}{e}{E}$, respectively, are defined in Section~\ref{sect:normaltraces} through the notion of pairing measure as follows:  
\begin{align*}
 (\chi_{E} F, D_{H} \chi_{E}) & = - \frac12 \Tr{F}{i}{E} |D_{H} \chi_{E}|, \\ \label{exterior_normal_trace_def} 
 (\chi_{\Omega \setminus E} F, D_{H} \chi_{E}) & = - \frac12\Tr{F}{e}{E} |D_{H} \chi_{E}|.
\end{align*}
We notice that this definition is well posed, since $(\chi_{E} F, D_{H} \chi_{E})$ and $(\chi_{\Omega \setminus E} F, D_{H} \chi_{E})$ are absolutely continuous with respect to the perimeter measure thanks to \eqref{abs_cont_pairing_infty}.
It is important to stress that the weak assumptions of 
Theorem~\ref{productrule} a priori do not ensure the uniqueness of 
$\tilde g$ and of the pairing $(F,D_Hg)$. They may both depend on the approximating sequence.
A first remark is that at those points where the averaged limit of $g$ exists with respect to $d^\cR$, the function $\tilde g$ can be characterized explicitly (Proposition~\ref{tilde_g_characterization}). However, the appearance of the
right invariant distance $d^\cR$ prevents the use of any
intrinsic regularity of the reduced boundary (Definition~\ref{d:RedBdry}) for sets of finite h-perimeter.

Despite these difficulties, a rather unexpected fact occurs, since in the 
case $g=\chi_E$ and $E$ has finite h-perimeter, it is possible to prove that the limit
$\widetilde{\chi_E}$ is uniquely determined, along with the normal traces, regardless of the choice of the mollifying sequence $\rho_{\ep_k}*\chi_E$. The surprising aspect is that we have no rectifiability result for the reduced boundary in arbitrary stratified groups.
We mainly use functional analytic arguments, the absolute continuity $\div F\ll\cS^{Q-1}$ and
the important Proposition~\ref{overline_chi_E}, that is further discussed below.
We summarize these relevant facts by restating here the main results of Theorem~\ref{Uniqueness_traces}.

\begin{theorem}[Uniqueness of traces] \label{IntroUniqueness_traces}
If $F \in \DM^{\infty}(H \Omega)$ and $E \subset \Omega$ is a set of finite h-perimeter, then there exists a unique $|\div F|$-measurable subset 
$$\intee{E}{F} \subset \Omega \setminus \redb E,$$
up to $|\div F|$-negligible sets, such that 
\begin{equation}\label{IntrochiEtildeUniq}
\tildef{\chi_{E}}(x) = \chi_{\intee{E}{F}}(x) + \frac{1}{2} \chi_{\redb E}(x)\quad \text{ for} \ \ |\div F|\text{-a.e.} \ x \in \Omega.
\end{equation}
In addition, there exist unique normal traces 
$$\Tr{F}{i}{E}, \Tr{F}{e}{E}\in L^{\infty}(\redb E; |D_{H} \chi_{E}|)$$
satisfying
\begin{align} \label{IntroLeibniz_infty_E_1_trace_ref} 
\div(\chi_{E} F) & = \chi_{\intee{E}{F}} \div F - \Tr{F}{i}{E} |D_{H} \chi_{E}|,  \\
\label{IntroLeibniz_infty_E_2_trace_ref} 
\div(\chi_{E} F) & = 
\chi_{\intee{E}{F} \cup \redb E} \div F - \Tr{F}{e}{E} |D_{H} \chi_{E}|.
\end{align}
\end{theorem} 

Equalities \eqref{IntroLeibniz_infty_E_1_trace_ref} and \eqref{IntroLeibniz_infty_E_2_trace_ref}
 immediately lead to general Gauss--Green formulas. Indeed, taking $F \in \DM^{\infty}(H \Omega)$ and a set of finite h-perimeter $E \Subset \Omega$, 
it is enough to evaluate \eqref{IntroLeibniz_infty_E_1_trace_ref} and \eqref{IntroLeibniz_infty_E_2_trace_ref} on $\Omega$, and then to exploit the fact that $\chi_{E} F \in \DM^{\infty}(H \Omega)$, thanks to Theorem \ref{productrule}, and Lemma~\ref{DMcomptsupp}. In this way, we obtain the following general versions of the Gauss--Green formulas in stratified groups:
\begin{align}
\label{Intro_G-G groups 1} \div F(\intee{E}{F}) & = \int_{\redb E} \Tr{F}{i}{E} \, d |D_{H} \chi_{E}|, \\
\label{Intro_G-G groups 2} \div F(\intee{E}{F}\cup \redb E) & = \int_{\redb E} \Tr{F}{e}{E} \, d |D_{H} \chi_{E}|.
\end{align}
We notice that, as a simple consequence of \eqref{IntroLeibniz_infty_E_1_trace_ref}, we can define $\intee{E}{F}$, up to $|\div F|$-negligible sets, as that Borel set in $\Omega \setminus \redb E$ satisfying
\begin{equation} \label{Introeq:def_E_1_F}
\div(\chi_{E} F) \res \Omega \setminus \redb E = \div F \res \intee{E}{F}.
\end{equation}
However, it is still an appealing open question to characterize $\intee{E}{F}$ explicitly, namely in geometric terms, or even to prove that the set $E^{1,F}$ does not depend on the vector field $F$, as it happens in the Euclidean context.

Nevertheless, we are able to find different sets of assumptions, involving either the regularity of $E$ or of the field $F$, for which $\intee{E}{F}$ can be properly detected. This immediately yields a number of Gauss--Green and integration by parts formulas in the spirit of the well known Euclidean results.

Before discussing Gauss--Green formulas, it is natural to ask whether normal traces have the locality property.
Rather unexpectedly, also locality of normal traces is obtained without any blow-up technique related to rectifiability of the reduced boundary.
Indeed, the classical proofs in the literature heavily employ the existence of an approximate tangent space at almost every point on the reduced boundary of a set of finite perimeter (\cite[Proposition 3.2]{ACM} and \cite[Proposition 4.10]{comi2017locally}). 

In Theorem~\ref{normal_trace_locality_theorem} we show that the normal traces of a divergence-measure horizontal section $F$ only depend on the orientation of the reduced boundary. It can be seen using the Leibniz rule established in Proposition~\ref{boundary_term_divergence}, the locality of perimeter in stratified groups proved by Ambrosio-Scienza \cite{ambrosio2010locality} and general arguments of measure theory. Another important tool that somehow allows us to overcome
the absence of regularity of the reduced boundary is Proposition~\ref{overline_chi_E},
where we prove that the weak$^{*}$ limit of $\rho_\ep\ast \chi_E$ 
in $L^\infty(\Omega;|D_H\chi_E|)$ is precisely $1/2$, for any
set $E\subset\G$ of finite h-perimeter and any symmetric mollifier $\rho$. 
This proposition can be proved by a soft argument borrowed from  \cite[Proposition~4.3]{ambrosio2010surface}.
It seems quite interesting that this weak$^*$ convergence comes from an analogous study in the infinite dimensional setting of Wiener spaces and it does not require any existence of blow-ups.

Proposition~\ref{overline_chi_E}, together with Remark~\ref{weak_conv_absolutely_continuous_perimeter_measures} and 
Lemma~\ref{overline_D_chi_E_1_2_lemma}, is fundamental in proving the refinements \eqref{IntroLeibniz_infty_E_1_trace_ref} and \eqref{IntroLeibniz_infty_E_2_trace_ref} of the Leibniz rule, along with the uniqueness results of Theorem~\ref{IntroUniqueness_traces}.
Furthermore, Proposition~\ref{overline_chi_E} immediately leads to the 
`intrinsic blow-up property' (Lemma~\ref{overline_D_chi_E_1_2_lemma}), that is fundamental to prove the estimates of Proposition~\ref{boundary_term_divergence1} for the normal traces of $F \in \DM^{\infty}(\Omega)$. 
We point out that the names of interior and exterior normal traces can be also justified by the same estimates \eqref{interior_normal_trace_norm}
and \eqref{exterior_normal_trace_norm}. We stress that the proofs of this result in the Euclidean literature rely on De Giorgi's blow-up theorem, see \cite[Theorem 3.2]{comi2017locally}, while Proposition~\ref{boundary_term_divergence1}, when the group is commutative, provides an alternative proof.

Returning to Gauss--Green formulas,  we observe first that when $\redb E$ is negligible with respect to $|\div F|$ (Theorem~\ref{G-G groups general II}), then the interior and exterior normal traces coincide. In particular, we can define the {\em average normal trace} $\Tr{F}{}{E}$ as the density of the pairing $(F, D_{H} \chi_{E})$ with respect to the h-perimeter measure $|D_{H} \chi_{E}|$, according to Definition~\ref{def:average_normal_trace}. Thanks to \eqref{F_D_chi_E_pairing_eq}, it is immediate to observe that
\begin{equation*}
\Tr{F}{}{E} = \frac{\Tr{F}{i}{E} + \Tr{F}{e}{E}}{2}.
\end{equation*}
As a result, when $|\div F|(\redb E) = 0$, we have $\Tr{F}{i}{E} = \Tr{F}{e}{E} = \Tr{F}{}{E}$, so that there exists a unique normal trace and the Gauss--Green formula \eqref{G-G_no_red_boundary} holds.

In case the divergence-measure field $F$ is continuous (Theorem~\ref{G-G groups general III}), then the Gauss--Green formula \eqref{G-G groups continuous} holds
and the normal trace has an explicit representation
by the scalar product between the field $F$ and the measure 
theoretic outer normal $\nu_E$. 

A first important consequence of the previous theorems is a Gauss--Green formula for horizontal fields with divergence-measure absolutely continuous with respect to the Haar measure of the group. Such a result could be also achieved from a modified product rule with additional assumptions on $\div F$, but we have preferred to start from a more general Leibniz rule and then derive some special cases from it.

\begin{theorem} \label{divF_abs_cont_G_G} 
Let $F \in \DM^{\infty}(H \Omega)$ such that $|\div F| \ll \mu$ and let $E \Subset \Omega$ be a set of finite h-perimeter. Then there exists a unique normal trace $\Tr{F}{}{E} \in L^{\infty}(\Omega; |D_{H} \chi_{E}|)$ such that there holds
\begin{equation} \label{G-G_abs_cont}
\div F(E) = \int_{\redb E} \Tr{F}{}{E} \, d |D_{H} \chi_{E}|. 
\end{equation}
\end{theorem}

The key point is to show that $E^{1, F}$ can be replaced by $E$, namely, to prove that their symmetric difference is $|\div F|$-negligible.
The Gauss--Green formula \eqref{G-G_abs_cont} naturally leads to the following integration by parts formula.

\begin{theorem} \label{IBP_divF_abs_cont}
Let $F \in \DM^{\infty}_{\rm loc}(H \Omega)$ be such that $|\div F| \ll \mu$, and let $E$ be a set of locally finite h-perimeter in $\Omega$. Let $\varphi \in C(\Omega)$ with $\nabla_{H} \varphi \in L^{1}_{\rm loc}(H \Omega)$ such that $$\mathrm{supp}(\varphi \chi_{E}) \Subset \Omega.$$ Then there exists a unique normal trace $\Tr{F}{}{E} \in L^{\infty}_{\rm loc}(\Omega; |D_{H} \chi_{E}|)$ of $F$, such that the following formula holds
\begin{equation} \label{IBP_abs_cont} \int_{E} \varphi \, d \div F + \int_{E} \ban{F, \nabla_{H} \varphi} \, dx = \int_{\redb E} \varphi \Tr{F}{}{E} \, d |D_{H} \chi_{E}|. \end{equation}
\end{theorem}

We notice that the assumption $|\div F| \ll \mu$ is very general in the sense that it is satisfied by $F \in W^{1, p}_{H, \rm loc}(H \Omega)$, for any $1 \le p \le \infty$. Moreover, it clearly implies $|\div F|(\redb E) = 0$, which means that the divergence-measure is not concentrated on the reduced boundary of $E$, and thus there is no jump component in the divergence. It is also worth to point out that both \eqref{G-G_abs_cont} and \eqref{IBP_abs_cont} hold also for sets whose boundary is not rectifiable in the Euclidean sense (Example~\ref{ex:notrectif}).

If we slightly weaken the absolute continuity assumption on $\div F$, requiring instead $|\div F|(\mtbR E) =0$, where $\mtbR E$ is the measure theoretic boundary of $E$ with respect to the right invariant distance $d^{\cR}$ \eqref{eq:mtbR E_def}, we are able to prove that $\intee{E}{F}$ is equivalent to $E^{1, \cR}$; that is, the measure theoretic interior with respect to $d^{\cR}$. As a consequence, we can derive a modified Gauss--Green formula (Theorem \ref{thm:GGIV}) and related statements.

Finally, other versions of the Gauss--Green theorem and integration by parts formulas can be obtained in the case the set $E\subset\G$ has finite perimeter in the Euclidean sense. Here it is important to investigate the behavior of the Euclidean pairing of a field $F \in \DM^{\infty}(H \Omega)$ and a function $g \in BV(\Omega) \cap L^{\infty}(\Omega)$. 
Let us remark that, even if the family $\DM^p(H\Omega)$ with $1\le p\le \infty$ is
strictly contained in the known space of divergence-measure fields (Section~\ref{sect:leibniz}), the known Euclidean results could only prove that 
the Euclidean pairing measure $(F, Dg)$ is absolutely continuous 
with respect to the total variation $|Dg|$.
This result does not imply the absolute continuity of the pairing with respect
to $|D_Hg|$, since this measure is absolutely continuous
with respect to $|Dg|$ while the opposite may not hold in general.

In Theorem \ref{abs_continuity_pairing_infty_Euclidean} we refine the classical results on $(F, D g)$, proving that, up to a restriction to bounded open sets,
\[
|(F, Dg)| \le \|F\|_{\infty, \Omega} |D_{H} g|.
\]
For this purpose, we have used the Euclidean convolution to compare the Euclidean pairing with the intrinsic pairing in the stratified group. While no exact commutation rule between the horizontal gradient and the Euclidean convolution holds, it is however possible to use an asymptotic commutator estimate similar
to the classical one by Friedrichs \cite{Friedrichs1944}, see also \cite{garofalo1996isoperimetric}.
Thanks to the above absolute continuity, we can actually prove that, given a set of Euclidean finite perimeter $E$, the group pairing $(F, D_{H} \chi_{E})$ defined in Theorem \ref{productrule} is actually equal to the Euclidean pairing $(F, D \chi_{E})$,
according to Theorem~\ref{equivalence_Leibniz_rule_Eucl}.
An important tool used in the proof of this result is Theorem~\ref{absolute continuity}, which states that $|\div F| \ll \SHaus{Q - 1}$, if $F \in \DM^{\infty}_{\rm loc}(H \Omega)$. This property of $\div F$ allows us to show in Theorem~\ref{equivalence_Leibniz_rule_Eucl} that we have $\intee{E}{F} = E^{1}_{|\cdot|}$, up to a $|\div F|$-negligible set, where we denote by $E^1_{|\cdot|}$ the Euclidean measure theoretic interior of $E$; that is, the set of points with density one with respect to the balls defined using the Euclidean distance in the group. 
These results allow us to prove Leibniz rules and integration by parts formulas for sets of Euclidean finite perimeter in stratified groups.

\begin{theorem} \label{Gauss_Green_Euclidean}
Let $F \in \DM^{\infty}_{\rm loc}(H \Omega)$ and $E \subset \Omega$ be a set of Euclidean locally finite perimeter in $\Omega$, then there exist interior and exterior normal traces $\Tr{F}{i}{E}, \Tr{F}{e}{E} \in L^{\infty}_{\rm loc}(\Omega; |D_{H} \chi_{E}|)$ such that, for any open set $U \Subset \Omega$, we have
\begin{align} \label{Leibniz_Eu_infty_1} \div(\chi_{E} F) & = \chi_{E^{1}_{|\cdot|}} \div F - \Tr{F}{i}{E} |D_{H} \chi_{E}|, \\ 
\label{Leibniz_Eu_infty_2} \div(\chi_{E} F) & = \chi_{E^{1}_{|\cdot|} \cup \redb E}\, \div F - \Tr{F}{e}{E} |D_{H} \chi_{E}|, \\ 
\label{Leibniz_Eu_infty_3} \chi_{\redb E}\, \div F & = 
(\Tr{F}{e}{E} - \Tr{F}{i}{E})\, |D_{H} \chi_{E}|
\end{align}
in $\mathcal{M}(U)$. Moreover, we get the trace estimates 
\begin{align*}
\|\Tr{F}{i}{E}\|_{L^{\infty}(\redb E \cap U; |D_{H} \chi_{E}|)}&\le \| F\|_{\infty, E \cap U},\\
\|\Tr{F}{e}{E}\|_{L^{\infty}(\redb E \cap U; |D_{H} \chi_{E}|)}&\le \| F\|_{\infty, U \setminus E}. 
\end{align*}
For any $\varphi \in C(\Omega)$ with $\nabla_{H} \varphi \in L^{1}_{\rm loc}(H \Omega)$ such that $\mathrm{supp}(\varphi \chi_{E}) \Subset \Omega$, we have 
\begin{align} \label{IBP_1_Eu} \int_{E^{1}_{|\cdot|}} \varphi \, d \div F + \int_{E} \ban{F, \nabla_{H} \varphi} \, dx & = \int_{\redb E} \varphi \Tr{F}{i}{E} \, d |D_{H} \chi_{E}|, \\
\label{IBP_2_Eu} \int_{E^{1}_{|\cdot|} \cup \redb E} \varphi \, d \div F + \int_{E} \ban{F, \nabla_{H} \varphi} \, dx & = \int_{\redb E} \varphi \Tr{F}{e}{E} \, d |D_{H} \chi_{E}|. \end{align}
\end{theorem}

Formulas \eqref{IBP_abs_cont} and \eqref{IBP_1_Eu} extend Anzellotti's pairings to stratified groups in the case the $BV$ function of the pairing is the characteristic function of a finite h-perimeter set. Indeed, if we take $E$ to be an open bounded set with Euclidean Lipschitz boundary, as in the assumptions of \cite[Theorem 1.1]{Anzellotti_1983}, then it is well known that $E^{1}_{|\cdot|} = E$. Thus, for this choice of $E$, it is clear that \eqref{IBP_abs_cont} and \eqref{IBP_1_Eu} are equivalent to definition of (interior) normal trace of Anzellotti; that is, the pairing between $F$ and $D \chi_{E}$ (see \cite[Definition 1.4]{Anzellotti_1983}).

Let us point out that Theorem~\ref{Gauss_Green_Euclidean} is new
even if it is seen in Euclidean coordinates, since the
measures appearing in the Leibniz rules are in fact absolutely continuous with respect to the h-perimeter.

In the assumptions of Theorem~\ref{Gauss_Green_Euclidean}, if $E \Subset \Omega$, taking 
the test function $\ph\equiv1$ in both \eqref{IBP_1_Eu} and \eqref{IBP_2_Eu}, 
we get the following general Gauss--Green formulas
\begin{align} \label{IBP_1_EuGG} 
\div F (E^{1}_{|\cdot|}) & = \int_{\redb E}  \Tr{F}{i}{E} \, d |D_{H} \chi_{E}|, \\
\label{IBP_2_EuGG} 
\div F (E^{1}_{|\cdot|} \cup \redb E) & = \int_{\redb E} \Tr{F}{e}{E} \, d |D_{H} \chi_{E}|.
\end{align}
Analogously, the estimates on the $L^{\infty}$-norm of the normal traces are similar to those in \eqref{traces_norm_Euclidean_eq}.
When the vector field $F$ is $C^1$ smooth up to the boundary of a bounded
set $E\Subset\Omega$ of Euclidean finite perimeter,
then all \eqref{G-G_abs_cont}, \eqref{IBP_1_EuGG} and \eqref{IBP_2_EuGG} boil down to the following one
\beq\label{eq:Eucl}
\int_E \div F\, dx = \int_{\mathscr{F}E} \ban{F,\nu_E} d |D_H \chi_E|=
\int_{\mathscr{F}E} \ban{F,N^H_E} d |D\chi_E|
\eeq
where $N^H_E=\sum_{j=1}^m \ban{N_E,X_j}_{\R^\q} X_j$ is the non-normalized horizontal
normal, $\G$ is linearly
identified with $\R^\q$ (Section~\ref{sect:metric}), $\ban{\cdot,\cdot}_{\R^\q}$ denotes the Euclidean scalar product, $N_E$ is the Euclidean measure theoretic exterior normal,
$|D\chi_E|$ is the Euclidean perimeter and $\mathscr{F}E$ is the Euclidean reduced boundary.
In the special case of \eqref{eq:Eucl} the proof is a simple application of the Euclidean theory of sets of finite perimeter, see for instance \cite[Remark~2.1]{ChoMagTys2015}.

Equalities of \eqref{eq:Eucl} can be also written using Hausdorff measures, getting
\beq\label{eq:EuclHaus}
\int_E \div F\, dx=\int_{\mathscr{F}E} \ban{F,N^H_E} d\Haus{\q-1}_{|\cdot|}
=\int_{\mathscr{F}E} \ban{F,\nu_E} d\SHaus{Q-1}.
\eeq
The first equality is a consequence of the rectifiability of Euclidean finite perimeter sets
\cite{DeGiorgi1955} and the second one follows from \cite{Mag31},
when the homogeneous distance $d$ constructing $\SHaus{Q-1}$
is suitably symmetric. For instance, when $E$ is bounded, $\der E$ is piecewise smooth
and $F$ is a $C^1$ smooth vector field on a neighborhood of $\overline E$, 
then \eqref{eq:Eucl} and \eqref{eq:EuclHaus} hold and the reduced boundary
$\mathscr F E$ can be replaced by the topological boundary $\der E$, coherently with the classical result \eqref{classical_G_G_eq}.

For smooth functions and sufficiently smooth domains, Green's formulas,
that are simple consequences of the Gauss-Green theorem, 
have proved to have a wide range of applications in classical PDE's.
In the context of sub-Laplacians in stratified groups these formulas
play an important role, \cite{BLU2007stratified}, \cite{RuzSur2017}.

As a consequence of our results, we obtain a very general version of Green's formulas in stratified groups. Precisely in the next theorem, \eqref{first_Green_Euclidean} and \eqref{second_Green_Euclidean} 
represent the first and the second Green's formulas, where the domain of integration is only assumed to be a set with Euclidean finite perimeter and the sub-Laplacians are measures.
\begin{theorem}\label{t:GreenIV}
Let $u \in C^1_{H}(\Omega)$ satisfy $\Delta_{H} u \in \mathcal{M}_{\rm loc}(\Omega)$ and let $E \subset \Omega$ be a set of Euclidean locally finite perimeter in $\Omega$. Then for each $v \in C_{c}(\Omega)$ with $\nabla_{H} v \in L^{1}(H \Omega)$ one has
\begin{equation} \label{first_Green_Euclidean}
\int_{E^1_{|\cdot|}} v \, d\Delta_{H} u  = \int_{\redb E} v \, \ban{\nabla_{H} u, \nu_{E}} \, d |D_{H} \chi_{E}| - \int_{E} \ban{\nabla_{H} v, \nabla_{H} u} \, dx. \end{equation}
If $u, v \in C^1_{H, c}(\Omega)$ also satisfy $\Delta_{H} u, \Delta_{H} v \in \mathcal{M}(\Omega)$, one has
\begin{equation} \label{second_Green_Euclidean}
\int_{E^{1}_{|\cdot|}} v \, d\Delta_{H} u  - u \, d\Delta_{H} v  = \int_{\redb E} \ban{v \nabla_{H} u - u \nabla_{H} v, \nu_{E}} \, d |D_{H} \chi_{E}|. \end{equation}
If $E \Subset \Omega$, one can drop the assumption that $u$ and $v$ have compact support in $\Omega$.
\end{theorem}
These Green's formulas are extended in Theorem~\ref{t:GreenII}  
to sets of h-finite perimeter, assuming that the sub-Laplacian 
is absolutely continuous with respect to the Haar measure of the group.

\section{Preliminaries}

In this section we present some preliminary notions, hence setting the notation, and extend some well known facts from the Euclidean context to the stratified groups' one.

In what follows, $\Omega$ is an open set in a stratified group $\G$.
Unless otherwise stated, $\subset$ and $\subseteq$ are equivalent. We denote by $E \Subset \Omega$ a set $E$ whose closure, $\bar{E}$, is a compact inside $\Omega$, by $E^{\circ}$ the interior of $E$ and by $\partial E$ its topological boundary.

\subsection{Basic facts on stratified groups}

We recall now the main features of the stratified group, also well known as Carnot group. More information on these groups can be found for instance in
\cite{Fol75}, \cite{folland1982hardy}, \cite{Heinonen1995}.
 In particular, we are including at the end of the section an approximation results for intrinsic Lipschitz functions. 

A {\em stratified group} can be seen as a linear space $\G$ equipped with an analytic group operation
such that its Lie algebra $\Lie(\G)$ is {\em stratified}.
This assumption on $\Lie(\G)$ corresponds to the following conditions
\[
\Lie(\G)=\cV_1\oplus\cdots\oplus\cV_\iota, \qquad [\cV_1,\cV_j]=\cV_{j+1}
\]
for all integers $j\ge0$ and $\cV_{j}=\{0\}$ for all $j>\iota$ with $\cV_\iota\neq\{0\}$. 
The integer $\iota$ is the step of nilpotence of $\G$. 
The tangent space $T_0\G$ can be canonically identified with $\Lie(\G)$ by associating to each $v\in T_0\G$ the unique left invariant vector field $X\in\Lie(\G)$ such that $X(0)=v$. This allows for transferring the Lie algebra structure from $\Lie(\G)$ to $T_0\G$. We can further simplify the structure of $\G$ by identifying it with $T_0\G$, hence having a Lie product on $\G$, that yields the group operation
by the Baker-Campbell-Hausdorff formula. This identification also gives a graded structure to $\G$, obtaining the subspaces $H^j$ of $\G$ from the subspaces
\[
 \set{v\in T_0\G: v=X(0),\  X\in\cV_j}, 
\]
therefore getting $\G=H^1\oplus\cdots \oplus H^\iota$. By these assumptions the exponential mapping  
\[
\exp:\Lie(\G)\to\G 
\]
is somehow the ``identity mapping'' $\exp X=X(0)$. It is clearly a bianalytic diffeomorphism. We will denote by $\q$ the dimension of $\G$, seen as a vector space. Those dilations that are compatible with the algebraic structure of $\G$  are defined as linear mappings $\delta_r:\G\to\G$ such that $\delta_r(p)=r^ip$ for each $p\in H^i$, $r>0$ and $i=1,\ldots,\iota$.

\subsection{Metric structure, distances and graded coordinates} \label{sect:metric}
We may use a graded basis to introduce a natural scalar product on
a stratified group $\G$. We then define the unique scalar product on $\G$ such that the graded basis is orthonormal.

 We will denote by $|\cdot|$ the associated 
Euclidean norm, that exactly becomes the Euclidean norm with respect to the corresponding graded coordinates.

On the other hand, the previous identification of $\G$ with $T_0\G$ yields
a scalar product on $T_0\G$, that defines by left translations 
a left invariant Riemannian metric on $\G$.  
By a slight abuse of notation, we use the symbols $|\cdot|$ and 
$\ban{\cdot,\cdot}$ to denote the norm arising from this left invariant Riemannian metric
and its corresponding scalar product.
By $\ban{\cdot,\cdot}_{\R^\q}$ we will denote the Euclidean scalar product,
that makes the fixed graded basis $(e_1,\ldots,e_\q)$ orthonormal. 
 
Notice that the basis $(X_1,\ldots,X_\q)$ of $\Lie(\G)$ associated

to our graded basis is clearly orthonormal with respect to the 
same left invariant Riemannian metric.

A {\em homogeneous distance} 
\[
d:\G\times\G\to[0,+\infty)
\]
on a stratified group $\G$ is a continuous and left invariant distance with 
\[
d(\delta_r (p),\delta_r (q))=r\,d(p,q)
\]
for all $p,q\in\G$ and $r>0$. 
We define the open balls as
\[
B(p,r)=\lb q\in\G: d(q,p)<r\rb. 
\]
The corresponding {\em homogeneous norm} will be denoted by 
$\|x\|=d(x,0)$ for all $x\in\G$.
It is worth to compare $d$ with our fixed Euclidean norm on $\G$, getting
\beq\label{eq:locEstDist}
C^{-1} |x-y|\le d(x, y) \le C |x - y|^{1/\iota}
\eeq
on compact sets of $\G$.
A homogeneous distance also defines a Hausdorff measure $\Haus{\alpha}$ and a spherical measure $\SHaus{\alpha}$. As it is customary, we set for $\delta > 0$ and $A \subset \G$:
\begin{align*} \Haus{\alpha}_{\delta}(A) & := \inf \left \{ \sum_{j \in J} \left ( \frac{\diam{A_{j}}}{2} \right )^{\alpha} \, : \, \diam{A_{j}} < \delta, \, A \subset \bigcup_{j \in J} A_{j} \right \}, \\ 
\SHaus{\alpha}_{\delta}(A) & := \inf \left \{ \sum_{j \in J} r_{j}^{\alpha} \, : \, 2 r_{j} < \delta, \, A \subset \bigcup_{j \in J} B(x_{j}, r_{j}) \right \} \end{align*}
and we take the following suprema 
\[
\Haus{\alpha}(A) := \sup_{\delta > 0} \Haus{\alpha}_{\delta}(A)
\qandq \SHaus{\alpha}(A) := \sup_{\delta > 0} \SHaus{\alpha}_{\delta}(A).
\]
It will be useful to introduce the right invariant distance
$d^\cR$ associated to $d$ as follows
\beq\label{d:d^R}
d^\cR(x,y):=\|xy^{-1}\|=d(xy^{-1},0)=d(x^{-1},y^{-1}).
\eeq
It is not difficult to check that $d^\cR$ is a continuous and right invariant
distance, that is also homogeneous, namely
\[
d^\cR(\delta_rx,\delta_ry)=rd^\cR(x,y)
\]
for $r>0$ and $x,y\in\G$. The local estimates \eqref{eq:locEstDist}
also show that $d^\cR$ defines the same topology of both $d$ and the
Euclidean norm $|\cdot|$.
The metric balls associated to $d^\cR$ are
\beq
B^\cR(p,r)=\lb q\in\G: d^\cR(q,p)<r\rb.
\eeq
We notice that 
\beq
B^\cR(0,1)=B(0,1),
\eeq
being $d^\cR(x,0)=d(x^{-1},0)=d(0,x)$ for all $x\in\G$.
%
%
%
%

A basis $(e_1,\ldots,e_\q)$ of $\G$ that respects the grading of $\G$ has the property that
\[
(e_{\m_{j-1}+1},e_{\m_{j-1}+2},\ldots,e_{\m_j})
\]
is a basis of $H^j$ for each $j=1,\ldots,\iota$, where $\m_j=\sum_{i=1}^j\dim H^i$
for every $j=1,\ldots,\iota$, $\m_0=0$ and $\m=\m_1$. The basis $(e_1,\ldots,e_\q)$
is then called {\em graded basis} of $\G$. 
Such basis provides the corresponding {\em graded coordinates} $x=(x_1,\ldots,x_\q)\in\R^\q$, that give the unique element of $\G$ that satisfies 
\[
p=\sum_{j=1}^\q x_je_j\in\G.
\]
We define a graded basis $(X_1,\ldots,X_\q)$ of $\Lie(\G)$ defining $X_j\in\Lie(\G)$ as the unique left invariant vector field with $X_j(0)=e_j$ and $j=1,\ldots,\q$.

We assign {\em degree $i$} to each left invariant vector field of $\cV_i$. In different terms, for each $j\in\set{1,\ldots,\q}$ we define the integer function $d_j$ on $\set{1,\dots,\iota}$ such that 
\[
\m_{d_j-1}< j\leq \m_{d_j}.
\]
The previous definitions allow to represent any
left invariant vector field $X_j$ as follows
\beq\label{eq:X_j}
X_j=\der_{x_j}+\sum_{i: d_i>d_j}^\q a_j^i\der_{x_i},
\eeq
where $j=1,\ldots,\q$ and $a^i_j$ are suitable polynomials.
The vector fields $X_1,X_2,\ldots,X_\m$ of degree one,
are the so-called {\em horizontal left invariant vector fields} and constitute the
horizontal left invariant frame of $\G$.

Using graded coordinates, the dilation of $x \in \R^{\q}$ is given by
\[
\delta_{r}(x) = \sum_{j = 1}^{\q} r^{d_{j}} x_{j} e_{j}.
\] 
Through the identification of $\G$ with $T_{0} \G$, it is also possible 
to write explicitly the group product in the graded coordinates.
%
\begin{comment}
\begin{proposition} \label{group product}
For any $p, q \in \G$, having graded coordinates $x = (x_{1}, \dots, x_{\q}), y = (y_{1}, \dots, y_{\q}) \in \R^{\q}$, we have that
\begin{equation*} p q = x + y + \QQ(x, y), \end{equation*}
where $\QQ : \R^{\q} \times \R^{\q} \to \R^{\q}$, and each $\QQ_{j}$ is a homogeneous polynomial of degree $d_{j}$ with respect to the intrinsic dilations of $\G$; that is,
\begin{equation*} \QQ_{j}(\delta_{\lambda}(x), \delta_{\lambda}(y)) = \lambda^{d_{j}} \QQ_{j}(x, y), \ \ \forall \ x, y \in \G. \end{equation*}
In addition, for any $x, y \in \G$, we have $\QQ_{j}(x, y) \equiv 0$ for any $1 \le j \le \m_{1}$, $\QQ_{j}(x, 0) = \QQ_{j}(0, y) = 0$ and $\QQ_{j}(x, x) = \QQ_{j}(x, - x) = 0$ for $\m_{1} \le j \le \q$. Finally, if $i \ge 2$ and $m_{i - 1} < j \le m_{i}$, then
\begin{equation*} \QQ_{j}(x, y) = \QQ_{j}(x_{1}, \dots, x_{m_{i - 1}}, y_{1}, \dots, y_{m_{i - 1}}). \end{equation*}
\end{proposition}
\end{comment}
%
%
%
In the sequel, an auxiliary scalar product on $\G$ is fixed such that 
our fixed graded basis is orthonormal.
The restriction of this scalar product to $V_1$ can be translated to the so-called {\em horizontal fibers}
\[
H_p\G=\{X(p)\in T_p\G: X\in\cV_1\}
\]
as $p$ varies in $\G$, hence defining a left invariant sub-Riemannian metric $g$ on $\G$.
We denote by $H\G$ the {\em horizontal subbundle} of $\G$, whose fibers are $H_x\G$.

The Hausdorff dimension $Q$ of the stratified group $\G$ with respect to any homogeneous distance is given by the formula 
\begin{equation*} Q= \sum_{i = 1}^{\iota} i \,\dime(H^{i}). \end{equation*}
We fix a Haar measure $\mu$ on $\G$, that with respect to our
graded coordinates becomes the standard $\q$-dimensional Lebesgue measure $\Leb{\q}$. Because of this identification, we shall write $d x$ instead of $d \mu(x)$ in the integrals.
This measure defines the corresponding Lebesgue spaces
$L^p(A)$ and $L^p_{\rm loc}(A)$ for any measurable set $A\subset\G$.
The $L^p$-norm will be denoted using the same symbols 
we will use for horizontal vector fields in Definition~\ref{Lp_horizontal_field}.

For any measurable set $E \subset \G$, we have 
$\mu(xE) = \mu(E)$ for any $x \in \G$ and 
\[
\mu(\delta_\lambda E) = \lambda^{Q} \mu(E) \quad \text{for any $\lambda>0$.}
\]
Since $B(p, r) = p\, \delta_rB(0, 1)$ and $B^\cR(p, r) = \delta_{r}(B(0, 1)) p$, we get
\beq\label{eq:muB(p,r)}
\mu(B(p, r)) = r^Q \mu(B(0, 1)) \qandq  \mu(B^\cR(p, r)) = r^{Q} \mu(B(0, 1))
\eeq
due to the left and right invariance of the Haar measure $\mu$.
The previous formulas show the existence of
constants $c_1,c_2>0$ such that
\beq\label{eq:LebHausdQ}
\Haus{Q}=c_1\, \Haus{Q}_\cR=c_2\,\mu,
\eeq
where $\Haus{Q}$ and $\Haus{Q}_\cR$ are the 
Hausdorff measures with respect to $d$ and $d^\cR$, respectively.
In particular, \eqref{eq:muB(p,r)} shows that $\mu$ is doubling 
with respect to both $d$ and $d^\cR$, hence
the Lebesgue differentiation theorem holds with respect to $\mu$ 
and both distances $d$ and $d^\cR$.

\begin{theorem} \label{Lebesgue_diff_theorem} Given $f \in L^{1}_{\rm loc}(\G)$, we have
\[
\lim_{r \to 0} \mean{B(x, r)}|f(y) - f(x)| \, dy  = 0 \qandq
\lim_{r \to 0} \mean{B^\cR(x, r)}|f(y) - f(x)| \, dy  = 0,
\]
for $\mu$-a.e. $x \in \G$.
\end{theorem}
For a general proof of the previous theorem in metric measure spaces equipped with a doubling measure, we refer for instance to \cite[Theorem~5.2.3]{AmbrosioTilli}.

\subsection{Differentiability, local convolution and smoothing}\label{sect:Difflocalsmooth}

The group structure and the intrinsic dilations naturally give a notion of
``differential'' and of ``differentiability'' made by the corresponding operations.
A map $L : \G \to \R$ is a homogeneous homomorphism, in short, a h-homomorphism
if it is a Lie group homomorphism such that $L\circ \delta_r=r\,L$.
It can be proved that $L : \G \to \R$ is a h-homomorphism if and only if 
there exists $(a_{1}, \dots, a_{m_{1}}) \in \R^{m_{1}}$ such that 
$L(x) = \sum_{j = 1}^{m_{1}} a_{j} x_{j}$ with respect to our fixed graded coordinates.
If not otherwise stated, in the following we denote by $\Omega$ an open set in $\G$. 

\begin{definition}[Differentiability] \label{d:differentiability}\rm We say that $f : \Omega \to \R$ is differentiable at $x_{0} \in \Omega$ if there is an h-homomorphism $L:\G\to\R$ such that 
\begin{equation*} \lim_{x \to x_{0}} \frac{f(x) - f(x_{0}) - L(x_{0}^{-1} x)}{d(x, x_{0})} = 0. \end{equation*}
If $f$ is differentiable, then $L$ is unique and we denote it simply by $df(x_{0})$.
\end{definition}
A weaker notion of differentiability, that holds for Sobolev and $BV$ functions
on groups is the following.
\begin{definition}[Approximate differentiability] \label{d:appdifferentiability}\rm We say that $f : \Omega \to \R$ is approximately differentiable at $x_{0} \in \Omega$ if there is an h-homomorphism $L:\G\to\R$ such that 
\begin{equation*} 
\lim_{r\to0^+} \mean{B(x_0,r)} \frac{|f(x) - f(x_{0}) - L(x_{0}^{-1} x)|}{r} \, dx = 0. \end{equation*}
The function $L$ is uniquely defined and it is called the approximate
differential of $f$ at $x_0$. The unique vector defining $L$ with
respect to the scalar product is denoted by $\nabla_{H} f (x_0)$.
\end{definition}
\begin{remark}
When $\G$ is the Euclidean space, the simplest stratified group,
Definition~\ref{d:differentiability} yields the standard 
notion of differentiability in Euclidean spaces.
\end{remark}
We denote by $C^1_H(\Omega)$ the linear space of real-valued functions
$f:\Omega\to\R$ such that the pointwise partial derivatives $X_1f,\ldots,X_{\m}f$
are continuous in $\Omega$. For any $f\in C^1_H(\Omega)$ we introduce
the {\em horizontal gradient}
\beq\label{eq:horiz_grad}
 \nabla_H f := \sum_{j = 1}^{\m} (X_{j} f) X_{j},
\eeq
whose components $X_jf$ are continuous functions in $\Omega$.
Taylor's inequality \cite[Theorem~1.41]{folland1982hardy} simply leads 
us to the everywhere differentiability of $f$ and to the formula
\[
df(x)(v)=\ban{\nabla_Hf(x),v}=\sum_{j=1}^\m v_jX_jf(x)
\]
for any $x\in\Omega$ and $\ds v=\sum_{j=1}^\q v_j e_j\in\G$.

We denote by $\Lip(\Omega)$, $\Lip_{\rm loc}(\Omega)$ and $\Lip_{c}(\Omega)$ the spaces of Euclidean Lipschitz, locally Lipschitz and Lipschitz functions with compact support in $\Omega$, respectively. Analogously, we can define the space of Lipschitz functions with respect to any homogeneous distance of the stratified group, $\Lip_{H}(\Omega)$.
It is well known that $\Lip_{\rm loc}(\Omega) \subset \Lip_{H, {\rm loc}}(\Omega)$, due to the local estimate \eqref{eq:locEstDist}.

\begin{theorem}\label{MSC-Rademacher} If $f \in \Lip_{H, {\rm loc}}(\Omega)$, then $f$ is differentiable $\mu$ almost everywhere.
\end{theorem}
This result follows from the Rademacher's type theorem 
by Monti and Serra~Cassano, proved in more general Carath\'eodory spaces,
\cite{MSC2001}.

\begin{remark}\label{product rule Lip H} \rm  From the standard
Leibniz rule, if $f, g \in \Lip_{H, {\rm loc}}(\Omega)$, the definition of differentiability
joined with Theorem~\ref{MSC-Rademacher} gives
\begin{equation*} \nabla_{H}(f g)(x) = f(x) \nabla_{H}g(x) + g(x) \nabla_{H} f(x) \ \ \text{for} \  \mu\text{-a.e.} \ x. \end{equation*}
\end{remark}

The Haar measure on stratified groups allows for defining the convolution
with respect to the group operation. 

\begin{definition}[Convolution]\label{definition:convol}\rm For $f, g \in L^{1}(\G)$, we define the 
{\em convolution} of $f$ with $g$ by the integral
\begin{equation*} (f \ast g)(x) := \int_{\G} f(y) g(y^{-1} x) \, dy = \int_{\G} f(x y^{-1}) g(y) \, dy, 
\end{equation*}
that is well defined for $\mu$-a.e.\ $x\in\G$, see for instance 
\cite[Proposition~1.18]{folland1982hardy}.
\end{definition}

Due to the noncommutativity of the group operation, one may clearly
expect that $g \ast f$ differs from $f \ast g$, in general.
This difference appears especially when we wish to localize the
convolution. In the sequel, $\Omega$ denotes an open set, if not otherwise stated. 
For every $\ep>0$, two possibly empty open subsets of $\Omega$ are defined as follows
\beq\label{eq:Omega-R_ep}
\Omega_\ep^\cR=\set{x\in\G: \dist^\cR(x,\Omega^c)>\ep}\qandq
\Omega_\ep=\set{x\in\G: \dist(x,\Omega^c)>\ep},
\eeq
where we have defined the distance functions
\[
\dist^\cR(x,A)=\inf\set{d^\cR(x,y): y\in A }\qandq
\dist(x,A)=\inf\set{d(x,y): y\in A }
\]
for an arbitrary subset $A\subset\G$. We finally define the open set
\[
A^{\cR,\ep}=\set{x\in\G: \dist^\cR(x,A)<\ep}.
\]
\begin{definition}[Mollification]\label{d:mollification}\rm
Given a function $\rho \in C_c(B(0, 1))$,  we set 
\[
\rho_{\eps}(x) := \eps^{-Q} \rho(\delta_{1/\eps}(x))
\]
for $\eps > 0$. If $f \in L^{1}(\Omega)$ and $x\in\G$, we define
\beq \label{d_eq_mollification_1}
\rho_\ep\ast f(x)= \int_\Omega \rho_\eps(x y^{-1}) f(y) \, dy.
\eeq
If we restrict the domain of this convolution considering $x\in\Omega^\cR_{\ep}$, then we can allow for $f \in L^{1}_{\rm loc}(\Omega)$ and we have 
\beq
\rho_\ep\ast f(x)=\int_{B^\cR(x, \eps)} \rho_\eps(x y^{-1}) f(y) \, dy,
\eeq
which is well posed since the map $y \to \rho_{\eps}(x y^{-1})$ has compact support inside $B^\cR(x,\ep)\subset\Omega$. In addition, under these assumptions, a simple change of variables also yields
\beq\label{eq:rho_ep_fyx}
\rho_\ep\ast f(x)=\int_{B(0,\ep)} \rho_\ep(y)\,f(y^{-1}x) \, dy= \int_{B(0, 1)} \rho(y)\,f\pa{(\delta_\eps y^{-1}) x} \, dy.
\eeq
Due to the noncommutativity, a different convolution may also be introduced
\beq
f\ast \rho_\ep(x)= \int_\Omega f(y) \rho_\eps(y^{-1}x)\, dy=
   \int_{B(x, \eps)} \rho_\eps(y^{-1}x) f(y) \, dy,
\eeq
where the first integral makes sense for all $x\in\G$ and the second one
only for $x\in\Omega_\ep$. 
\end{definition}
It is not difficult to show that the mollified functions $\rho_\ep\ast f$ and $f\ast\rho_\ep$ enjoy many standard properties. For instance, $\rho_\ep\ast f$ converges to $f$ in $L^1_{\rm loc}(\Omega)$, whenever $f\in L^1_{\rm loc}(\Omega)$.

We may also define the convolution between a (signed) Radon measure and a continuous function. As it is customary, we denote by $\mathcal{M}_{\rm loc}(\Omega)$ the space of signed Radon measures on $\Omega$, and by $\mathcal{M}(\Omega)$ the space of finite signed Radon measures on $\Omega$.

\begin{definition}[Local convolution of measures] \label{d:loc_conv_meas}\rm 
Let us consider two open sets $\Omega,U\subset\G$ and 
define the new open set $O=U(\Omega^{-1})\subset\G$. Let $f \in C(O)$ and $\nu \in \mathcal{M}(\Omega)$. Then the {\em convolution between $f$ and $\nu$} is given by
\begin{equation} 
(f \ast \nu)(x) := \int_{\Omega} f(xy^{-1}) \, d \nu(y),
\end{equation}
with the additional assumption that $\Omega\ni y\mapsto f(xy^{-1})$ is
$|\nu|$-integrable for every $x\in U$.
Thus, $f\ast\nu$ is well defined in $U$.
If $\rho \in C_c(B(0, 1))$, for any $x\in\Omega^\cR_\ep$ 
we may represent the convolution as follows 
\begin{equation} \label{moll_meas_def}
(\rho_{\eps} \ast \nu)(x) = \int_{\Omega} \rho_{\eps}(xy^{-1}) \, d \nu(y) = \int_{B^\cR(x, \eps)} \rho_\eps(x y^{-1}) \, d \nu(y).
\end{equation}
The first integral makes sense for 
all $x\in\G$, being $\rho$ continuously extendable by zero outside $B(0,1)$. In addition, \eqref{moll_meas_def} is well posed also for $\nu \in \cM_{\rm loc}(\Omega)$, if $x \in \Omega_{2 \eps}^{\cR}$. The function $\rho_\ep\ast\nu$ is the {\em mollification} of $\nu$.
\end{definition}
\begin{definition}[Local weak$^*$ convergence]
We say that a family of Radon measures $\nu_\ep\in\cM(\Omega)$
{\em locally weakly$^*$ converges to $\nu\in\cM(\Omega)$}, if
for every $\phi\in C_c(\Omega)$ we have
\beq\label{eq:nu_epconv}
\int_\Omega \phi\, d\nu_\ep\to\int_\Omega \phi\,  d\nu \qs{as} \ep\to0^+
\eeq
and in this case we will use the symbols
$\nu_\ep\weakto\nu$ as $\ep\to0^+$.
\end{definition}

\begin{remark}\label{remark:weakconvOmegaep}
In the sequel, the local weak$^*$ convergence above will also refer to measures
$\nu_\ep\in\cM(\Omega^\ep)$ defined on a family of increasing open sets $\Omega^\ep\subset\Omega$ 
as $\ep$ decreases, such that $\bcup_{\ep>0}\Omega^\ep=\Omega$ and for every compact set $K\subset\Omega$ there exists $\ep'>0$ such that $K\subset\Omega_{\ep'}$.
This type of local weak$^*$ convergence does not make a substantial difference
compared to the standard one, so we will not use a different symbol.

For instance, the local weak$^*$ convergence of \eqref{weak_conv_D_f}
refers to a family of measures that are not defined on all of $\Omega$
for every fixed $\ep>0$. We stress that this distinction is important, since
our mollifier $\rho$ is assumed to be only continuous. 
\end{remark}

\begin{remark} \label{weak_conv_moll_measure} \rm 
For any measure $\nu \in \mathcal{M}(\Omega)$ and any mollifier $\rho \in C_{c}(B(0, 1))$ satisfying $\rho(x) = \rho(x^{-1})$ and $\ds\int_{B(0,1)} \rho \, dx = 1$, 
we observe that $\rho_{\eps} \ast \nu \in C(\G)$  and we have the local weak$^*$ convergence of measures 
\begin{equation} \label{weak_conv_moll_measure_eq}
(\rho_{\eps} \ast \nu) \mu \weakto \nu
\end{equation}
in $\Omega$, as $\ep\to0^+$.
Indeed, let $\phi\in C_c(\Omega)$ and let $\ep>0$ small enough, such that $\supp\, \phi\subset U$ and $U\subset\Omega^\cR_\ep$ is an open set. Then we have 
\begin{align*} 
\int_\Omega \phi(x) (\rho_{\eps} \ast \nu)(x) \, dx & = \int_U
\phi(x)\pa{ \int_{B^\cR(x,\ep)} \rho_{\eps}(xy^{-1}) \, d \nu(y) } dx \\
 &= \int_{U^{\cR,\ep}} \pa{ \int_U \rho_{\eps}(yx^{-1}) \phi(x)  \, dx } d \nu(y) \\
& = \int_{\Omega} (\rho_{\eps} \ast \phi)(y) \, d \nu(y) \to \int_\Omega \phi(y) \, d \nu(y), \end{align*}
since $U^{\cR,\ep}\subset (\Omega^\cR_\ep)^{\cR,\ep}\subset\Omega$ and $\rho_{\eps} \ast \phi \to \phi$ uniformly on compact subsets of $\Omega$. 
The previous equalities also show that
\begin{equation} \label{exchanging_convolution} 
\int_{\Omega} \phi(x) (\rho_{\eps} \ast \nu)(x) \, dx = \int_{\Omega} (\rho_{\eps} \ast \phi)(y) \, d \nu(y), \end{equation}
whenever $\rho \in C_{c}(B(0, 1))$, $\nu \in \mathcal{M}(\Omega)$ and
$\phi \in C_{c}(\Omega)$ such that $\mathrm{supp}(\phi) \subset \Omega^\cR_{\eps}$.
\end{remark}

\begin{comment} We notice that 
\[ (\rho_{\eps} \ast \nu) (x) = \int_{\Omega} \rho_{\eps}(x y^{-1}) \, d \nu(y) \]
is defined on the whole group $\G$ and $\rho_{\eps} \ast \nu \in C(\G) \cap L^{1}(\Omega)$, for any $\rho \in C_{c}(B(0, 1))$. Indeed,
\[ \int_{\Omega} |\rho_{\eps} \ast \nu| \, dx \le \int_{\Omega} \int_{\Omega} \rho_{\eps}(x y^{-1}) \, d |\nu|(y) \, dx \le \int_{\Omega} \int_{\G} \rho_{\eps}(x y^{-1}) \, dx \, d |\nu|(y) = |\nu|(\Omega). \]
\end{comment}

\begin{remark}\label{r:contin_vectconvmeas}
The previous arguments also show that $\rho_\ep\ast f$, for $f\in L^1(\Omega)$, enjoys all properties of the convolution in Remark~\ref{weak_conv_moll_measure}. The same is true for $\rho_\ep\ast\nu$, if we consider
\beq\label{eq:nubeta}
\nu=(u_1,\ldots,u_\m) \beta,
\eeq
where $u_1,\ldots,u_\m:\Omega\to\eR$ belong to $L^{1}(\Omega; \beta)$ and $\beta\in\cM^+(\Omega)$.
\end{remark}

\begin{proposition}\label{proposition:convolC1}
Let $\Omega^\cR_{\eps}\neq\epty$ for some $\ep>0$. 
The following statements hold.
\begin{enumerate}
\item 
If $f \in C_H^{1}(\Omega)$, $\rho \in C_c(B(0, 1))$ and $\ds\int_{B(0,1)} \rho \, dx = 1$ , then $\rho_\ep\ast f\in C^{1}_{H}(\Omega^\cR_{\eps})$,
\beq\label{eq:convolC1}
X_{j}(\rho_{\eps} \ast f) = \rho_{\eps} \ast X_{j} f \quad \text{in $\Omega^\cR_{\eps}$,}
\eeq
and both $\rho_\ep\ast f$ and $\nabla_H(\rho_\ep\ast f)$ uniformly converge to $f$ and $\nabla_Hf$ on compact subsets of $\Omega$.
In addition, if $f \in L^{1}(\Omega)$ and $\rho\in C_c^{k}(B(0,1))$ for some $k\ge1$, then $\rho_\eps\ast f \in C^{k}(\G)$.
\item 
If $f \in L^\infty(\Omega)$ and $\rho \in \Lip_{c}(B(0,1))$, then $\rho_\ep\ast f\in \Lip_{\rm loc}(\G)$.
\end{enumerate}
\end{proposition}
\begin{proof}
Let $f \in C_H^{1}(\Omega)$ and $\rho \in C_c(B(0, 1))$. 
By the estimate of \cite[Theorem~1.41]{folland1982hardy} and Lebesgue's dominated convergence, we have
\begin{align*} 
X_{j}(\rho_{\eps} \ast f)(x) & =\lim_{t\to 0}\int_{B(0,\ep)}\rho_\ep(y)\, \frac{f(y^{-1}x(te_j))-f(y^{-1}x)}{t} dy \\
&=\int_{B(0,\ep)}\rho_\ep(y)\, \lim_{t\to 0}\frac{f(y^{-1}x(te_j))-f(y^{-1}x)}{t} dy \\
& = \int_{B(0,\ep)} \rho_{\eps}(y) (X_{j} f)(y^{-1}x) \, dy = (\rho_{\eps} \ast X_{j} f)(x) \end{align*}
for any $x \in \Omega^\cR_{\eps}$, due to the left invariance of 
$X_{j}$. By the condition $\ds \int_{B(0,1)} \rho \, dx = 1$, the uniform convergence follows from the continuity of both $f$ and $\nabla_{H} f$,
along with the standard properties of the convolution. The second point can be proved in a similar way, by differentiating the mollifier $\rho_\ep$.
Here we only add that this differentiation is possible at every
point of $\G$, being $\rho_\ep\ast f$ defined on the whole group.

If $f \in L^\infty(\Omega)$ and $\rho \in \Lip_{c}(B(0,1))$, then it is easy to notice that the mollification $\rho_{\eps} \ast f$ as in \eqref{d_eq_mollification_1} is well posed and belongs to $L^{\infty}(\G)$. Hence, for each compact set $K\subset\G$ and any $x, y \in K$, we have 
\begin{equation*} 
\begin{split}
|f^{\eps}(x) - f^{\eps}(y)| &\le \int_{\Omega} |f(z)|\; |\rho_{\eps}(x z^{-1}) - \rho_{\eps}(y z^{-1})| \, dz  \\
& = \ep^{-Q} \int_V |f(z)|\; \Big|\rho\pa{\delta_{1/\ep}(x z^{-1})}-\rho\pa{\delta_{1/\ep}(y z^{-1})}\Big| \, dz  \\
&\le \|f\|_{\infty, \Omega} \,2\, \mu(B(0,1))\, L\, C_\ep |x - y|, 
\end{split}
\end{equation*}
where $V=\overline{B^\cR(x, \eps) \cup B^\cR(y, \eps)} \subset\Omega$, $L>0$ is the
Lipschitz constant of $\rho$ and $C_\ep>0$ is the supremum of all Lipschitz constants $L_{\ep,z}$ of $K\ni x\mapsto \delta_{1/\ep}(x z^{-1})$ as $z$ varies in 
$V$. Due to this fact, we have $L_\ep<+\infty$.
\end{proof}

The next density theorem follows from the choice of suitable mollified functions.

\begin{theorem} \label{approx Lip function}
If $g \in \Lip_{H, {\rm loc}}(\Omega)$, then there exists a sequence $(g_{k})_{k}$ in $C^{\infty}(\Omega)$ with the following properties
\begin{enumerate} 
\item $g_{k} \to g$ uniformly on compact subsets of $\Omega$;
\item $\|\nabla_Hg_k\|_{\infty,U}$ is bounded for each $U\Subset\Omega$ and $k$ sufficiently large;
\item $\nabla_{H} g_{k} \to \nabla_{H} g$ $\mu$-a.e. in $\Omega$.
\end{enumerate}
If $g \in \Lip_{H, c}(\Omega)$, then we can choose all $g_{k}$
to have compact support in $\Omega$.
\end{theorem}
\begin{proof}
We consider $\rho \in C^{\infty}_{c}(B(0, 1))$ satisfying $\rho \ge 0$
and $\ds\int_{B(0,1)} \rho \, dy = 1$. Then we define
\[
g_{k}(x)=(\rho_{\ep_k}\ast g)(x)= \int_\Omega \rho_{\eps_k}(x y^{-1}) g(y) \, dy
\]
for a positive sequence $\ep_k$ converging to zero, where $\rho_{\eps}(x) = \eps^{-Q} \rho(x/\eps)$. In this way, we  have $g_{k} \in C^{\infty}(\Omega)$ and Proposition~\ref{proposition:convolC1} implies the uniform convergence on compact subsets of $\Omega$. 
For the subsequent claims, we may consider an open set $U\Subset\Omega$
and take $k$ sufficiently large such that $U\Subset\Omega^\cR_{\ep_k}$. For every fixed $x\in U$, formula \eqref{eq:rho_ep_fyx} yields 
\[ 
g_k(x) = \int_{B(0,\ep_k)}\rho_{\ep_k}(y) g(y^{-1}x) dy, 
\]
therefore the following equalities hold:
\begin{align*}
&\pa{g_k(xh)-g_k(x)-\ban{\int_{B(0,\ep_k)}\rho_{\ep_k}(y)\nabla_H g(y^{-1}x) dy,h }}\|h\|^{-1} \\
=&\pa{g_k(xh)-g_k(x)-\int_{B(0,\ep_k)}\rho_{\ep_k}(y)\ban{\nabla_H g(y^{-1}x),h} dy }\|h\|^{-1} \\
=&\int_{B(0,\ep_k)}\rho_{\ep_k}(y)\, \pa{\frac{g(y^{-1}xh) -g(y^{-1}x)-\ban{\nabla_H g(y^{-1}x),h}}{\|h\|}}\,dy   
\end{align*}
for $h$ sufficiently small. The difference quotient in the last integral
is uniformly bounded with respect to $y$ and $h$, due to the Lipschitz continuity
of $g$. The a.e.\ differentiability of $g$, by Theorem~\ref{MSC-Rademacher},
joined with Lebesgue's dominated convergence show that
\[
\nabla_Hg_k(x)=\int_{B(0,\ep_k)}\rho_{\ep_k}(y)\nabla_H g(y^{-1}x) dy.
\]
The local Lipschitz continuity of $g$ provides 
local boundedness for $\nabla_Hg$, hence the previous formula
immediately establishes the second property.
By a change of variables, we get
\[
\nabla_Hg_k(x)-\nabla_Hg(x)=
\int_{B^\cR(x,\ep_k)}\rho_{\ep_k}(xz^{-1}) \pa{\nabla_H g(z)-\nabla_Hg(x) }dy.
\]
From this, it follows that 
\begin{equation*} |\nabla_Hg_k(x)-\nabla_Hg(x)| \le
\|\rho\|_{\infty} \frac{1}{\eps_{k}^{Q}} \int_{B^\cR(x,\ep_k)} |\nabla_H g(z)-\nabla_Hg(x) | \, dy, \end{equation*} 
and now we can conclude by Theorem \ref{Lebesgue_diff_theorem}.
Finally, if $g$ has compact support, it is clear that also $\mathrm{supp}(\rho_{\eps} \ast g)$ is compact in $\Omega$, for $\eps$ small enough.
\end{proof}

\section{Some facts on {\em BV} functions in stratified groups}\label{sect:BVfunction}

In this section we present some basic notions on $BV$ functions and
sets of finite perimeter in stratified groups. In particular, 
some new smoothing arguments for $BV$ functions in stratified groups
are presented. Additional results and references 
on these topics can be found for instance in \cite{cassano2016some}. 

We will need the important concept of {\em horizontal vector field}, that incidentally will also appear in the sequel, in connection with Leibniz formulas and the Gauss--Green theorem in stratified groups.

Let $\Omega\subset\G$ be an open set and denote by $H\Omega$ the restriction of the {\em horizontal subbundle} $H\G$ to the open set $\Omega$, whose {\em horizontal fibers} $H_p\G$ are restricted to all points $p\in \Omega$. 
\begin{definition}[Horizontal vector fields]  \label{Lp_horizontal_field} \rm
Any measurable section $F:\Omega\to H\Omega$ of $H\Omega$ is called a {\em measurable horizontal vector field in $\Omega$}.
We denote by $|F|$ the measurable function $x\to |F(x)|$,
where $|\cdot|$ denotes the fixed graded invariant Riemannian norm.

The $L^p$-norm of a measurable horizontal vector field $F$ in $\Omega$
is defined as follows:
\beqa\label{Lp_norms} 
&&\|F\|_{p, \Omega}  := \left ( \int_{\Omega} |F(x)|^{p} \, dx \right )^{1/p}  \quad \text{if} \ 1 \le p < \infty, \\
&&\| F \|_{\infty, \Omega} := \mathrm{ess} \sup_{x \in \Omega} |F(x)| 
\quad\qquad \text{if} \ p=\infty \label{Linfty_norm}.
\eeqa
We say that a measurable horizontal vector field $F$ is 
a {\em $p$-summable horizontal field} if $|F|\in L^p(\Omega)$.
For $1\le p\le\infty$, we denote by $L^{p}(H \Omega)$ the space of $p$-summable horizontal fields endowed with the norms defined
either in \eqref{Lp_norms} or \eqref{Linfty_norm}.
A measurable horizontal vector field $F$ in $\Omega$ is 
{\em locally $p$-summable} if for any open subset $W \Subset \Omega$, we have $F \in L^p(HW)$. The space of all such vector fields is denoted by $L^{p}_{\rm loc}(H \Omega)$. 

For $k\in\N\sm\set{0}$, the linear space of all $C^k$ smooth sections of $\Omega$ 
is denoted by $C^k(H\Omega)$ and its elements will be called
{\em horizontal vector fields of class $C^k$}. Considering 
the subclass of all $C^k$ smooth horizontal vector fields
with compact support in $\Omega$ yields the space
$C^k_c(H\Omega)$. When $k=0$ the integer $k$ is omitted and
the corresponding space of vector fields will include those with
continuous coefficients.
\end{definition}

It is easy to observe that for any $f\in C^1_H(\Omega)$ its horizontal gradient, given by \eqref{eq:horiz_grad}, defines a continuous horizontal vector field in $\Omega$.
\begin{definition}\label{def:BV} \rm We say that a function $f : \Omega \to \R$ is a function of {\em bounded h-variation}, or simply a {\em $BV$ function}, and write $f \in BV_{H}(\Omega)$, if $f \in L^{1}(\Omega)$ and 
\begin{equation} \label{eq:BV-D_Hf}
| D_H f | (\Omega) := \sup \left \{ \int_{\Omega} f \div \phi \, dx : \phi \in C^{1}_{c}(H \Omega), |\phi| \le 1 \right \} < \infty.
\end{equation}
We denote by $BV_{H, \rm loc}(\Omega)$ the space of functions in $BV_{H}(U)$ for any open set $U \Subset \Omega$.
\end{definition}
\begin{remark}
In the case $\G$ is commutative and equipped with the Euclidean metric,
the previous notion of $BV$ function coincides with the classical one.
\end{remark}
Due to the standard Riesz representation theorem, it is possible to show that 
when $f \in BV_{H}(\Omega)$ the total variation of its distributional horizontal grandient $| D_Hf |$ is a nonnegative Radon measure on $\Omega$.
In addition, there exists a $|D_H f |$-measurable horizontal vector field $\sigma_{f} : \Omega \to H \Omega$ in $\Omega$ such that $|\sigma_{f}(x)| = 1$ for $| D_{H} f |$-a.e. $x \in \Omega$, and
\beq\label{eq:BV_byparts}
\int_{\Omega}f \div \phi \, dx = - \int_{\Omega} \ban{\phi, \sigma_{f}} \, d |D_Hf |, 
\eeq
for all $\phi \in C^{1}_{c}(H \Omega)$. In fact, these conditions are equivalent to the finiteness of \eqref{eq:BV-D_Hf}. 
\begin{remark} \label{Lip test1} \rm
Using Theorem~\ref{approx Lip function} one can actually see that in 
\eqref{eq:BV_byparts} the horizontal vector field $\phi$ can be
taken with coefficients in $\Lip_{H, c}(\Omega)$.
\end{remark}
The integration by parts formula \eqref{eq:BV_byparts} allows us
to think of $D_Hf$ as a kind of  ``measure with values in $H \Omega$'',
even if we have not a standard vector measure in $\R^n$.
In the former case, the different horizontal tangent spaces of $H\Omega$  do not allow us to sum the different values of the measures.

\begin{definition}[Measures in $H\Omega$]
Let $\gamma\in\cM(\Omega)$ be a measure and let $\alpha:\Omega\to H\Omega$ be a 
$\gamma$-measurable horizontal vector field. We define the {\em vector measure $\alpha\gamma$ in $H\Omega$} as the linear operator  
\[
C_c(H\Omega)\ni \phi\lra \int_\Omega \ban{\phi,\alpha}\, d\gamma=:
\int_\Omega \ban{\phi,d(\alpha \gamma)}, 
\]
bounded on $C_c(HU)$ for any open set $U\Subset\Omega$ with respect to the $L^\infty$-topology.
\end{definition}
According to the previous definition, $\sigma_f |D_Hf|$ is a vector measure on $H\Omega$,
that will be also denoted by $D_Hf$.
When a horizontal frame $(X_1,\ldots,X_\m)$ is fixed, we can represent $\sigma_f$ by 
the $|D_Hf|$-measurable functions
$\sigma_f^1,\ldots,\sigma_f^\m:\Omega\to\R$ such that 
\[
\sigma_f = \sum_{j = 1}^{\m} \sigma_f^j\, X_j.
\]
Thanks to this representation, $D_Hf$ can be naturally identified with the vector valued Radon measure
\beq\label{eq:sigma_f_j}
(\sigma_f^1,\ldots,\sigma_f^\m)|D_Hf|.
\eeq
For each $j=1,\ldots,\m$, we define the scalar measures 
\beq\label{eq:BVX_jf}
D_{X_j}f=\sigma_f^j |D_Hf|,
\eeq
that represent the distributional derivatives of a $BV$ function as Radon measures. 
In view of the Radon-Nikod\'ym theorem, we have the decomposition
\[
D_Hf=D^{\rm a}_Hf+D^{\rm s}_Hf
\]
where $D^{\rm a}_Hf$ denotes the absolutely continuous part of
$D_Hf$ with respect to the Haar measure of the group
and $D^{\rm s}_Hf$ the singular part.

Any $BV$ function is approximately differentiable a.e.\ and in addition the approximate differential coincides a.e.\ with the vector density of $D^{\rm a}_Hf$, see \cite[Theorem~2.2]{AmbMag2003}.
As a result, we are entitled to denote 
$X_jf\in L^1(\Omega)$ as the unique measurable function such that
\beq\label{eq:X_jfmu}
D_{X_j}f=X_jf\,\mu.
\eeq
Thus, to a $BV$ function $f$ we can assign a unique horizontal vector 
field $\nabla_Hf\in L^1(H\Omega)$ whose components are defined
in \eqref{eq:X_jfmu} and by definition we have
\[
D^{\rm a}_Hf=\nabla_Hf\, \mu.
\]
As a result, we have the decomposition of measures
\beq\label{eq:BVdec-as}
D_{H} f = \nb{f}\, \mu + D^{\rm s}_Hf.
\eeq
In the previous formula, $\nb{f}$ is uniquely defined, up to $\mu$-negligible sets, and it coincides a.e.\ with the approximate differential of $f$, see Definition~\ref{d:appdifferentiability}.

The vector measure $D_Hf$ in $H\Omega$ enjoys some standard properties of vector measures, as
those mentioned in Remark~\ref{r:contin_vectconvmeas}. The mollification of $D_Hf$ is the vector field
\beq\label{eq:conv_D_Hf}
\rho_\ep\ast D_Hf(x):=
\sum_{j=1}^\m\big(\rho_\ep\ast (\sigma_f^j\,|D_Hf|)\big)(x) X_j(x)
=\sum_{j=1}^\m\big(\rho_\ep\ast (D_{X_j}f)\big)(x) X_j(x).
\eeq
\begin{lemma} \label{conv_meas_bounded_function} Let $F \in L^{\infty}(H\Omega)$, $\gamma \in \mathcal{M}(\Omega)$ and $\alpha : \Omega \to H \Omega$ be a $\gamma$-measurable horizontal section such that $|\alpha(x)| = 1$ for $\gamma$-a.e.\ $x \in \Omega$. Let $\nu:= \alpha \gamma$ be 
the corresponding vector measure in $H\Omega$ and let $\rho \in C_{c}(B(0, 1))$ be a nonnegative mollifier satisfying $\rho(x) = \rho(x^{-1})$ and $\int_{B(0,1)} \rho \, dx = 1$. Then the measures $\ban{F, (\rho_{\eps} \ast \nu)} \mu$ satisfy the estimate
\beq\label{eq:conv_unif_bd}
\int_\Omega |\ban{F, (\rho_{\eps} \ast \nu)}| dx\le \|F\|_{\infty,\Omega}\,|\nu|(\Omega)
\eeq
for $\ep>0$ and any weak$^{*}$ limit point $(F, \nu) \in \mathcal{M}(\Omega)$ satisfies $|(F, \nu)| \le \|F\|_{\infty, \Omega}|\nu|$.
\end{lemma}
\begin{proof}
For any $\phi \in C_{c}(\Omega)$, denoting by $K\subset\Omega$ its support, we have
\begin{align*}
\int_{\Omega} \phi(x) \ban{F(x), (\rho_{\eps} \ast \nu)(x)} \, dx 
& = \int_K \phi(x) \ban{F(x), \int_{K^{\cR,\ep}} \rho_{\eps}(xy^{-1}) \alpha(y)} \, d \gamma(y) \, dx \\
& = \int_{K^{\cR,\ep}} \int_\Omega \phi(x) \ban{F(x), \alpha(y)} \rho_{\eps}(y x^{-1})\, dx \, d \gamma(y) \\
& = \int_\Omega \ban{(\rho_{\eps} \ast (\phi F))(y), \alpha(y)} \, d \gamma(y). \end{align*}
This implies that
\begin{equation*} \left | \int_{\Omega} \phi(x) \ban{F(x), (\rho_{\eps} \ast \nu)(x)} \, dx \right |\le \|\rho_\ep\ast(\phi F) \|_{\infty, \Omega} |\nu|(\Omega) \le \|\phi\|_{\infty, \Omega} \|F\|_{\infty, \Omega} |\nu|(\Omega), 
\end{equation*}
therefore the sequence $\ban{F, (\rho_{\eps} \ast \nu)} \mu$ 
satisfies \eqref{eq:conv_unif_bd}. Let now $\ban{F, (\rho_{\eps_{k}} \ast \nu)} \mu$ be a weakly converging subsequence, whose limit we denote by $(F, \nu)$. Then, by definition of weak$^{*}$ limit, for any $\phi \in C_{c}(\Omega)$ we obtain
\begin{align*} \left | \int_{\Omega} \phi \, d (F, \nu) \right | & = \lim_{\eps_{k} \to 0} \left | \int_{\Omega} \phi \ban{F, (\rho_{\eps} \ast \nu)} \, dx \right | \le \lim_{\eps_{k} \to 0} \|F\|_{\infty, \Omega} \int_{\Omega} |\phi| |\rho_{\eps_k} \ast \nu| \, dx \\
& \le \lim_{\eps_{k} \to 0} \|F\|_{\infty, \Omega} \int_{\Omega} |\phi| (\rho_{\eps} \ast |\nu|) \, dx = \|F\|_{\infty, \Omega} \int_{\Omega} |\phi| \, d |\nu|, \end{align*}
since $(\rho_{\eps} \ast |\nu|) \mu\weakto |\nu|$ by Remark \ref{weak_conv_moll_measure}.
This concludes our proof.
\end{proof}
\begin{remark}
We stress the fact that the pairing measure $(F, \nu)$ is not unique in general, unless $|\nu| \ll \mu$. Indeed, in the absolutely continuous case, we can write $\nu = G \mu$, for some $G \in L^{1}(H \Omega)$ and we have $\rho_{\eps} \ast G \to G$ in $L^{1}(H \Omega)$. Hence, it is clear that
\begin{equation*}
\ban{F, (\rho_{\eps} \ast \nu)} \mu \weakto \ban{F, G} \mu  \quad \text{in} \ \mathcal{M}(\Omega),
\end{equation*}
and so $(F, v) = \ban{F, G} \mu$.
\end{remark}

In the sequel, we need a weak notion of divergence
for nonsmooth vector fields, according to the following
definition.
\begin{definition}[Distributional divergence]\label{d:divergence_distrib}
\label{DistDivdef} The {\em divergence} of a measurable horizontal vector field 
$F\in L^1_{\rm loc}(H\Omega)$ is defined as the following distribution
\begin{equation} \label{distributional_divergence}
C_c^\infty(\Omega)\ni \phi\mapsto  - \int_{\Omega} \ban{F, \nabla_H \phi} \, dx.
\end{equation}
We denote this distribution by $\div F$. The same symbol will denote
the measurable function defining the distribution, whenever it exists.
\end{definition}
\begin{remark} \label{Lip test} 
Due to Theorem~\ref{approx Lip function}, we can extend \eqref{distributional_divergence}
to test functions $\phi$ in $\Lip_{H, c}(\Omega)$.
\end{remark}
The representation of left invariant vector fields \eqref{eq:X_j} gives 
\beq\label{eq:Ff_j}
F=\sum_{j=1}^\m F_j X_j=\sum_{j=1}^\m F_j \der_{x_j}+
\sum_{i=\m+1}^\q \Big(\sum_{j=1}^\m
F_j a_j^i\Big) \der_{x_i}=\sum_{i=1}^\q f_i \der_{x_i},
\eeq
therefore we have an analytic expression for the components $f_{i}$ of $F$
and we may observe that the distributional divergence \eqref{distributional_divergence} 
corresponds to the standard divergence 
\begin{equation} \label{distributional_div_equiv}
(\div F)(\phi)=-\int_{\Omega} \ban{F, \nabla_H \phi} \, dx=-\int_{\Omega} \ban{F, \nabla \phi}_{\R^\q} \, dx,
\end{equation}
where $\nabla$ denotes the Euclidean gradient in the fixed graded coordinates.

Let us consider a horizontal vector field $F=\sum_{j=1}^\m F_j\,X_j$
of class $C^1_H$, namely $F_j\in C^1_H(\Omega)$ for every $j=1,\ldots,\m$.
It is easy to notice that its distributional divergence coincides with
its pointwise divergence. Indeed, for $\phi\in C^1_c(\Omega)$ we have
\[
-\int_\Omega \ban{F,\nabla_H\phi} dx=-\int_\Omega \sum_{j=1}^\m X_j(F_j \phi)
+\int_\Omega\phi \sum_{j=1}^{\m} X_j F_j\,dx=\int_\Omega\phi \sum_{j=1}^{\m} X_j F_j\,dx.
\]
The last equality follows by approximation, using Theorem~\ref{approx Lip function},
the divergence theorem for $C^1$ smooth functions and the fact that $\div X_j=0$.
For this reason, in the sequel we will not use a different notation to distinguish
between the distributional divergence and the pointwise divergence.

\begin{comment}
\begin{remark}
One may observe that for a Sobolev horizontal vector field
$F=\sum_{j=1}^\m F_jX_j\in W^{1,p}(H\Omega)$ its 
distributional divergence is given by the following measurable function
\beq\label{eq:div_HF}
\div F=\sum_{j=1}^\m X_jF_j
\eeq
where $X_jF_j\in L^p(\Omega)$ is the distributional derivative $F_j$ along $X_j$. 
\end{remark}
%
%
\begin{remark}
aggiungere se utile che se $F=\sum F_j X_j$ con $F_j\in C^1_H(\Omega)$
la divergenza $\div $ non \`e definita puntualmente come operatore
differenziale, ma usando la \eqref{eq:convolC1} si pu\`o provare
che la divergenza distribuzionale soddisfa
\[
\div F=\sum_{j=1}^\m X_jF_j
\]
........
\end{remark}
\end{comment}

%
\begin{comment}
The following lemma shows that the variation measure of
a function in $BV$ is suitably invariant under left translations. 
\begin{lemma}  Let $f \in BV_{H, \rm loc}(\Omega)$, then for any $z \in \G$ we have $f \circ l_{z} \in BV_{H, \rm loc}(z^{-1} \Omega)$ and the distributional horizontal gradient satisfies $ D_{H}(f \circ l_{z}) = (l_{z^{-1}})_{\#}D_{H}f$.
\end{lemma}
\begin{proof}
It is clear that $f \circ l_{z} \in L^1_{\rm loc}(z^{-1} \Omega)$. Let $\phi \in C^{1}_{c}(z^{-1} \Omega, H \Omega)$, then we have
\begin{align*} 
\int_{z^{-1} \Omega} f \circ l_{z} \div \phi \, dx & = \int_{\Omega} f(x) (\div \phi)(z^{-1}x) \, dx = \int_{\Omega} f(x) \div (\phi(z^{-1} x)) \, dx, 
\end{align*}
since the divergence is left invariant.
This immediately leads us to our claim.
\end{proof}
\end{comment}

The following lemma will play an important role in the sequel. It tells us that
a mollifier that is only continuous turns a $BV$ function into a $C^1_H$ function.

\begin{lemma} \label{l:BVD_Hfconv}
If $f \in BV_{H, \rm loc}(\Omega)$, 
$\ep>0$ is such that $\Omega^\cR_{2\ep}\neq\epty$, $\rho \in C_c(B(0, 1))$ is nonnegative 
such that $\rho(x) = \rho(x^{-1})$ and $\int_{B(0,1)}\rho=1$, then 
$\rho_{\eps} \ast f \in C^1_H(\Omega_{2\eps}^{\cR}) \cap C(\G)$ and 
\begin{equation} \label{commutation conv derivative meas} 
\nabla_H(\rho_{\eps} \ast f) = (\rho_{\eps} \ast D_{H} f) \quad\text{on}
\quad \Omega^\cR_{2\ep}. \end{equation}
\end{lemma}
\begin{proof}
Let $\phi \in C^{1}_{c}(H \Omega_{2\eps}^{\cR})$. Arguing similarly as in the proof of \eqref{exchanging_convolution} and observing that $(\Omega_{2\ep}^\cR)^{\cR,\ep}\subset\Omega^\cR_\ep$, we get the following equalities,
where the second one is a consequence of \eqref{eq:convolC1}:
\begin{align*}
\int_{\Omega^\cR_{2\ep}} (\rho_{\eps} \ast f)(x)\, \div \phi(x) \, dx & = 
 \int_{\Omega_\ep^\cR} f(y)\, (\rho_{\eps} \ast \div \phi)(y) \, dy  = \int_{\Omega_\ep^\cR} f(y)\, \div ( \rho_{\eps} \ast \phi) (y) \, dy \\
& = - \int_{\Omega_\ep^\cR} \ban{(\rho_{\eps} \ast \phi)(y), \sigma_{f}(y)} \, d |D_{H} f|(y) \\
& = - \int_{\Omega_\ep^\cR} \int_{\Omega^\cR_{2\ep}} \rho_{\eps}(y x^{-1}) \ban{\phi(x), \sigma_{f}(y)} \, dx \, d |D_{H} f|(y) \\
& = - \int_{\Omega^\cR_{2\ep}} \int_{\Omega_\ep^\cR} \rho_{\eps}(x y^{-1}) \ban{\phi(x), \sigma_{f}(y)} \, d |D_{H} f|(y) \, dx \\
& = - \int_{\Omega^\cR_{2\ep}} \ban{\phi(x) , (\rho_{\eps} \ast D_{H}f)(x)} \, dx.
\end{align*}
\begin{comment}
We have also observed that 
\[ \int_{\Omega^\cR_{2\ep}} \int_{\Omega_\ep^\cR} \rho_{\eps}(x y^{-1}) \ban{\phi(x), \sigma_{f}(y)} \, d |D_{H} f|(y) \, dx 
 = \int_{\Omega^\cR_{2\ep}} \int_\Omega \rho_{\eps}(x y^{-1}) \ban{\phi(x), \sigma_{f}(y)} \, d |D_{H} f|(y) \, dx 
\]
\end{comment}
The standard density of $C^{1}_{c}(\Omega^\cR_{2\ep})$ in 
$C_{c}(\Omega^\cR_{2\ep})$, shows that $\rho_{\eps} \ast f \in BV_{H, {\rm loc}}(\Omega_{2\eps}^{\cR})$ and proves the following formula
\begin{equation} \label{commutation conv derivative meas1} 
D_{H}(\rho_{\eps} \ast f) = (\rho_{\eps} \ast D_{H} f) \mu\quad\text{on}
\quad \Omega^\cR_{2\ep}. 
\end{equation}
By Remark~\ref{weak_conv_moll_measure} and Remark~\ref{r:contin_vectconvmeas}, it follows that both $\rho_\ep\ast f$ and $\rho_{\eps}\ast D_{H} f$ are continuous,
therefore $\rho_\ep\ast f\in C^1_H(\Omega^\cR_{2\ep})$ and formula \eqref{commutation conv derivative meas} follows. 
\end{proof}

Taking into account \eqref{eq:BVX_jf}, 
formula \eqref{commutation conv derivative meas} can be written by components as follows
\begin{equation}\label{commutation conv derivative}
X_{j} (\rho_{\eps} \ast f)(x) = (\rho_{\eps} \ast D_{X_j} f)(x)\quad
\text{for every}\; x \in \Omega_{2\eps}^\cR.
\end{equation}

We prove now the smooth approximation results stated in
Theorem~\ref{commutation convolution derivative}.

\begin{proof}[Proof of Theorem~\ref{commutation convolution derivative}] 
The $C^1_H$ smoothness of $\rho_\ep\ast f$ and  
$\nabla_H(\rho_\ep\ast f)=\rho_\ep\ast D_Hf$ on $\Omega^\cR_{2\ep}$
follow from Lemma~\ref{l:BVD_Hfconv}. 
As a result, by \eqref{weak_conv_moll_measure_eq} 
and taking into account \eqref{eq:BVX_jf} we obtain 
the local weak$^*$ convergence
$X_{j} (\rho_{\eps} \ast f) \weakto D_{X_j} f$ for any $j = 1, \dots, \m$.
This proves the first convergence of \eqref{weak_conv_D_f}.

\begin{comment}
Considering $\ep>0$ sufficiently small, such that 
$\supp(\phi) \Subset \Omega^{\cR}_{2\eps}$, by \eqref{commutation conv derivative} we get the following equalities
\begin{align*} 
 \int_\Omega \phi \, X_{j} (\rho_{\eps} \ast f) \, dx 
 & = \int_{\Omega^\cR_{2\ep}} \phi\, (\rho_{\eps} \ast X_{j} f) \, dx  = \int_{\Omega^\cR_{2\ep}} \phi(x) \int_{\Omega^{\cR}_{\eps}} \rho_{\eps}(xy^{-1}) \, d X_{j}f(y) \, dx \\
& = \int_{\Omega^{\cR}_{\eps}} 
\pa{\int_{\Omega^\cR_{2\ep}} \phi(x) \rho_{\eps}(y x^{-1}) \, dx} \, d X_{j}f(y) 
= \int_{\Omega^{\cR}_{\eps}} (\rho_{\eps} \ast \phi)(y) \, d X_{j} f(y) \\
& = \int_{\Omega} (\rho_{\eps} \ast \phi)(y) \, d X_{j} f(y).
\end{align*}
The last equality is justified also by the following fact.
If $x\in\Omega^\cR_{2\ep}$ and $y\notin\Omega^\cR_\ep$, then $\dist^\cR(y,\Omega^c)\le\ep$ and for $\omega\in\Omega^c$ we have
\[
d^\cR(x,y)\ge d^\cR(x,\omega)-d^\cR(y,\omega)>2\ep-d^\cR(x,\omega)
\]
for the arbitrary choice of $\omega$ we get
\[
d^\cR(x,y)\ge 2\ep-\dist^\cR(x,\Omega)\ge \ep,
\]
therefore $\rho_\ep(yx^{-1})=0$. There for the
support of $y\to \rho_\ep(yx^{-1})$ contains $(\Omega^\cR_\ep)^c$
for every $x\in\Omega^\cR_{2\ep}$, therefore $\rho_\ep\ast\phi(y)=0$
if $y\notin\Omega^\cR_\ep$.
\end{comment}

To prove \eqref{pointwise_upper_control}, we 
consider $\phi \in C_{c}(H\Omega^\cR_{2\ep})$, therefore
\begin{align*}
\left | \int_{\Omega^\cR_{2\ep}} \ban{\phi(x), \nabla_{H} (\rho_{\eps} \ast f)(x)} \, dx \right | & = \left | \sum_{j = 1}^{\m} \int_{\Omega^\cR_{2\ep}} \phi_{j}(x) X_{j} (\rho_{\eps} \ast f)(x) \, dx \right | \\
& = \left |\int_\Omega \sum_{j = 1}^{\m} (\rho_{\eps} \ast \phi_{j})(y) \, d D_{X_j} f(y) \right | \\
& =  \left |\int_{\Omega} \ban{ (\rho_{\eps} \ast \phi)(y), \sigma_{f}} \, d |D_{H} f|(y) \right | .
\end{align*}
The second equality follows from \eqref{commutation conv derivative}
joined with \eqref{exchanging_convolution} and the last equality is a consequence of \eqref{eq:BVX_jf}. As a result, applying again \eqref{exchanging_convolution}, 
we get
\beq\label{eq:rhoepphi}
\begin{split}
\left | \int_{\Omega^\cR_{2\ep}} \ban{\phi(x), \nabla_{H} (\rho_{\eps} \ast f)(x)} \, dx \right | 
& \le \int_{\Omega} (\rho_{\eps} \ast |\phi|)(y) \, d |D_{H} f|(y)\\
 & = \int_{\Omega^\cR_{2\ep}} |\phi(x)|\, (\rho_{\eps} \ast |D_{H} f|)(x) \, dx. 
\end{split}
\eeq
By taking the supremum among all $\phi \in C_{c}(HU)$
with $\|\phi\|_{\infty} \le 1$ and $U\subset\Omega^\cR_{2\ep}$ open
set, we are immediately lead to \eqref{pointwise_upper_control}.
From the first inequality of \eqref{eq:rhoepphi}, we also get
\begin{equation*} \left | \int_{\Omega^\cR_{2\ep}} \ban{\phi(x), \nabla_{H} (\rho_{\eps} \ast f)(x)} \, dx \right | \le \|\phi \|_{\infty} |D_{H}f|(\Omega),
\end{equation*} 
whenever $\phi \in C_{c}(H\Omega^\cR_{2\ep})$.
This immediately proves \eqref{total_var_convol_convergence}.

Finally, we are left to show the second local weak$^*$ convergence of
\eqref{weak_conv_D_f}.  
We fix an open set $U\Subset\Omega$ and notice that, by 
\eqref{weak_conv_moll_measure_eq}, we have 
\begin{equation} \label{intermediate_weak_conv_tot_var_D_f} \rho_{\eps} \ast |D_{H} f| \weakto |D_{H} f|  \quad\text{in}\quad U.
 \end{equation} 
In addition, by \eqref{total_var_convol_convergence} and
\eqref{commutation conv derivative} we know that 
\[
\limsup_{\ep\to0}|\nabla_H(\rho_{\eps} \ast f)|(U)
\le\limsup_{\ep\to0}|\rho_{\eps} \ast D_{H} f|(\Omega^\cR_{2\ep})\le
|D_Hf|(\Omega),
\]
hence there exists a weakly$^*$ converging sequence 
$|\nabla_H(\rho_{\eps_{k}} \ast f)|\,\mu$ with limit $\nu$ in $U$.
By virtue of \cite[Proposition~1.62]{AFP} with \eqref{weak_conv_D_f}, we have $|D_{H} f|\le \nu$ in $U$. 
Therefore, taking nonnegative test functions $\ph\in C_c(U)$ 
and using \eqref{pointwise_upper_control}, we get
\[
\int_U \ph\, |\nabla_{H}(\rho_{\eps} \ast f)| \,dx\le \int_U \ph \, (\rho_\ep\ast|D_Hf|) \,dx
\]
for $\ep>0$ sufficiently small, depending on $U$.
Passing to the limit as $\ep\to0$, due to \eqref{intermediate_weak_conv_tot_var_D_f} 
we get the opposite inequality $\nu\le |D_H f|$ in $U$,
therefore establishing the second local weak$^*$ convergence of \eqref{weak_conv_D_f}.
\end{proof}

\begin{remark}
In the assumptions of Theorem~\ref{commutation convolution derivative},
the first local weak$^*$ convergence of \eqref{weak_conv_D_f}
joined with the lower semicontinuity of the total variation with respect
to the weak$^*$ convergence of measures imply that
\begin{equation*}
\liminf_{\eps \to 0} |\nabla_{H}(\rho_{\eps} \ast f)|(U) \ge |D_{H} f|(U)
\end{equation*}
for every open set $U\Subset\Omega$. If in addition $\rho\in C_c^1(B(0,1))$, and  then $\rho_\ep\ast f\in C^1(\G)$ by Proposition~\ref{proposition:convolC1},
the previous inequality immediately gives
\begin{equation*}
\liminf_{\eps \to 0} |\nabla_{H}(\rho_{\eps} \ast f)|(\Omega) \ge |D_{H} f|(\Omega).
\end{equation*}
\end{remark}

\subsection{Sets of finite perimeter in stratified groups}

Functions of bounded h-variation, introduced in the previous section,
naturally yield sets of finite h-perimeter
as soon as we consider their characteristic functions.
\begin{definition}[Sets of finite h-perimeter]\label{def:finitePer}\rm A measurable set $E \subset \G$ is of {\em locally finite h-perimeter} in $\Omega$ (or is a locally {\em h-Caccioppoli set}) if $\chi_{E} \in BV_{H, \rm loc}(\Omega)$. In this case, for any open set $U \Subset \Omega$, we denote the {\em h-perimeter of $E$} in $U$ by
\begin{equation*} \Per(E, U) =|\der_HE|(U):=|D_H\chi_{E} | (U). \end{equation*}
We say that $E$ is a set of {\em finite h-perimeter} if $|D_{H} \chi_{E}|$ is a finite Radon measure on $\Omega$.
The {\em measure theoretic exterior h-normal} of $E$ in $\Omega$ 
is the $|\der_HE|$-measurable horizontal section $\nu_{E} := - \sigma_{\chi_{E}}$.
\end{definition}

\begin{definition}[Measure theoretic interior]\label{d:measThInt}
If $E\subset \G$ is a measurable set, we define the
{\em measure theoretic interior of $E$}  as 
\[
E^1=\set{x\in\G: \lim_{r\searrow0}\frac{\mu(E\cap B(x,r))}{\mu(B(x,r))}=1}.
\] 
The {\em measure theoretic exterior of $E$} is the set
\[
E^0=\set{x\in\G: \lim_{r\searrow0}\frac{\mu(E\cap B(x,r))}{\mu(B(x,r))}=0}.
\]
\end{definition}

We can define two subsets of the topological boundary of a set of locally finite h-perimter $E$: the reduced boundary $\redb E$ and the measure theoretic boundary $\mtb E$.

\begin{definition}[Reduced boundary]\label{d:RedBdry}\rm If $E\subset\G$ is a set of locally finite h-perimeter, 
we say that $x$ belongs to the {\em reduced boundary} if
\begin{enumerate}
\item $|D_H \chi_{E}| (B(x, r)) > 0$ for any $r > 0$;
\item there exists $\displaystyle \lim_{r \to 0} \mean{B(x, r)} \nu_{E} \, d | D_{H} \chi_{E} |$;
\item $\displaystyle \pal{ \lim_{r \to 0} \mean{B(x, r)} \nu_{E} \, d | D_{H} \chi_{E} | } = 1$.
\end{enumerate}
The reduced boundary is denoted by $\redb E$. 
\end{definition}

\begin{definition}[Measure theoretic boundary]\rm Given a measurable set $E\subset\G$, we say that $x \in \mtb E$, if the following two conditions hold:
\begin{equation*} \limsup\limits_{r \to 0} \frac{\mu(B(x,r) \cap E)}{r^{Q}} > 0 \qandq   \limsup\limits_{r \to 0} \frac{\mu(B(x,r) \setminus E)}{r^{Q}} > 0. \end{equation*}
\end{definition}
The Lebesgue differentiation of Theorem \ref{Lebesgue_diff_theorem} immediately shows that 
\beq  \label{LebesgueNegEssBdry}
\mu(\mtb E) = 0.
\eeq
However, a deeper differentiability
result shows that indeed $\mtb E$ is $\sigma$-finite with 
respect to the h-perimeter measure.
Indeed, a general result on the integral representation of the perimeter measure holds in doubling metric measure spaces which admit a Poincar\'e inequality \cite{Ambrosio2002}. 

The following result restates \cite[Theorem~4.2]{Ambrosio2001} in the special case of stratified groups, that are special instances of Ahlfors regular metric spaces equipped with a Poincar\'e inequality.

\begin{theorem} \label{perimeter repr} Given a set of finite h-perimeter $E$ in $\G$, there exists $\gamma \in (0, 1)$ such that the measure $\Per(E, \cdot)$ is concentrated on the set $\Sigma_{\gamma} \subset \mtb E$ defined as
\begin{equation*} \Sigma_{\gamma} = \left \{ x : \limsup_{r \to 0} \min \left \{ \frac{\mu(E \cap B(x, r))}{\mu(B(x, r))}, \frac{\mu(B(x, r) \setminus E)}{\mu(B(x, r))} \right \} \ge \gamma \right \}. \end{equation*}
Moreover, $\SHaus{Q - 1}(\mtb E \setminus \Sigma_{\gamma}) = 0$, $\SHaus{Q - 1}(\mtb E) < \infty$ and there exists $\alpha > 0$, independent of $E$, and a Borel function $\theta_{E} : \G \to [\alpha, + \infty)$ such that
\begin{equation} \label{perimeter repr eq} \Per(E, B) = \int_{B \cap \mtb E} \theta_{E} \, d \SHaus{Q - 1} \end{equation}
for any Borel set $B\subset\G$. Finally, the perimeter measure is asymptotically doubling, i.e., for $\Per(E, \cdot)$-a.e. $x \in \G$ we have
$\ds \limsup_{r \to 0} \frac{\Per(E, B(x, 2r))}{\Per(E, B(x, r))} < \infty$.
\end{theorem}

\begin{lemma}\label{difference redb mtb} 
If $E\subset\G$ is a set of locally finite h-perimeter, then 
\beq\label{eq:difference redb mtb} 
\redb E\subset \mtb E \qandq \Haus{Q - 1}(\mtb E \setminus \redb E) = 0.
\eeq 
\end{lemma}
\begin{proof}
The lower estimates of \cite{FSSC5} joined with the invariance of reduced boundary
and perimeter measure when passing to the complement of $E$ immediately
give the inclusion of \eqref{eq:difference redb mtb}.
By Theorem~\ref{perimeter repr}, the perimeter measure $\Per(E, \cdot) = |D \chi_{E}|(\cdot)$ is a.e. asymptotically doubling, therefore the following differentiation property holds:
\begin{equation*} 
\lim_{r\to0}\mean{B(x, r)} \nu_{E} \, d |D_{H} \chi_{E}| = \nu_{E}(x) \ \text{for} \ |D_H \chi_E|\text{-a.e.} \ x, \end{equation*}
according to \cite[Sections 2.8.17 and 2.9.6]{Fe}.
This implies that $|D_H \chi_E|$-a.e. $x$ belongs to $\redb E$; that is, 
$|D_H \chi_E|(\G \setminus \redb E) = 0$. 
Moreover, \eqref{perimeter repr eq} yields 
$|D_H \chi_E|(B) \ge \alpha \SHaus{Q - 1}(B \cap \mtb E)$ on Borel
sets $B\subset\G$. This inequality also extends to $|D_H \chi_E|$-measurable sets,
hence taking $B = \G \setminus \redb E$, we obtain $\SHaus{Q - 1}(\mtb E \setminus \redb E) = 0$. Since $\Haus{Q-1}$ and $\SHaus{Q-1}$ have the same negligible sets, 
the equality of \eqref{eq:difference redb mtb} follows.
\end{proof}

\begin{remark} \label{negl_abs_cont_red_boundary} 
The previous lemma joined with \eqref{eq:LebHausdQ} and \eqref{LebesgueNegEssBdry} shows that
\beq  \label{LebesgueNegBdry}
\mu(\redb E) = 0.
\eeq
In addition, \eqref{perimeter repr eq} and \eqref{eq:difference redb mtb} imply that, for any Borel set $B$, $|D_{H} \chi_{E}|(B) = 0$ if and only if $\SHaus{Q - 1}(B \cap \redb E) = 0$; that is, the measures $|D_{H} \chi_{E}|$ and $\SHaus{Q - 1} \res \redb E$ have the same negligible sets. In particular, $|D_{H} \chi_{E}| \ge \alpha \SHaus{Q - 1}\res \redb E$.
\end{remark}

\begin{remark}\label{r:sign_chi_E} 
Let $\nu \in \mathcal{M}(\Omega)$ be any nonnegative measure and
denote by $\tildef{\chi_{E}}$ any weak$^*$ cluster point
of $\rho_\ep\ast\chi_E$ in $L^\infty(\Omega; \nu)$. Then the
lower semicontinuity of the $L^{\infty}$-norm gives
\[
 \|\tildef{\chi_{E}}\|_{L^{\infty}(\Omega; \nu)} \le \liminf_{\eps_{k} \to 0} \|(\rho_{\eps_{k}} \ast \chi_{E})\|_{L^{\infty}(\Omega, \nu)} \le 1 
\]
for some positive sequence of $\eps_k$ converging to zero.
Considering a nonnegative test function $\psi \in L^{1}(\Omega; \nu)$,
we also have
\begin{equation*}
0 \le \int_{\Omega} \psi \, (\rho_{\eps} \ast \chi_{E}) \, d \nu \to \int_{\Omega} \psi \tildef{\chi_{E}} \, d \nu,
\end{equation*}
hence proving that $0 \le \tildef{\chi_{E}}(x) \le 1$ for $\nu$-a.e.\ $x\in\Omega$. 
\end{remark}

In relation to the following proposition, we are grateful to Luigi Ambrosio for having pointed out to us his work with Alessio Figalli \cite{ambrosio2010surface}, where they study points of density 1/2 and relate them to the representation of perimeters in Wiener spaces.

\begin{proposition}\label{overline_chi_E} Let $E \subset \Omega$ be a set of locally finite h-perimeter, $\rho \in C_c(B(0, 1))$ be a mollifier satisfying $\rho \ge 0$, $\rho(x) = \rho(x^{-1})$ and $\ds\int_{B(0, 1)} \rho(y) \, dy = 1$. It follows that 
\begin{equation} \label{Leibniz_rule_Sobolev_reg_BV} D_{H}((\rho_{\eps} \ast \chi_{E}) \chi_{E}) = (\rho_{\eps} \ast \chi_{E}) D_{H} \chi_{E} + \chi_{E} (\rho_{\eps} \ast D_{H} \chi_{E}) \quad\text{in $\mathcal{M}(\Omega_{2 \eps}^{\cR})$}
\end{equation}
for any $\eps > 0$ such that $\Omega_{2 \eps}^{\cR} \neq \emptyset$ and
\beq\label{limit1/2}
\rho_\ep\ast \chi_E\weakstarto\frac{1}{2}\qs{as} \ep\to0^+\quad \text{in $L^\infty(\Omega;|D_H\chi_E|)$.}
\eeq
\end{proposition}
\begin{comment}
Suppose by contradiction that $\rho_\ep\ast\chi_E$ does not
weakly$^*$ converge to $1/2$. Then there exists $u_0\in L^1(\Omega;|D_H\chi_E|)$  such that 
\[
\pal{\int_\Omega  \rho_{\ep_k}\ast\chi_E\, u_0\,d|D_H\chi_E|-\frac12 
|D_H\chi_E|(\Omega)}>\kappa_0
\]
for $k$ large, then we may get a subsequence $\tilde\ep_k$ such that
\[
\int_\Omega  \rho_{\tilde\ep_k}\ast\chi_E\, u_0\,d|D_H\chi_E|\to
\frac12|D_H\chi_E|(\Omega)
\]
and this yields a contradiction.
\end{comment}
\begin{proof}
It suffices to show that for any cluster point $\overline{\chi_{E}} \in L^{\infty}(\Omega; |D_{H} \chi_{E}|)$ of $\rho_{\eps} \ast \chi_{E}$ as $\eps \to 0$,
then we have $\overline{\chi_{E}} =1/2$
a.e.\ with respect to $|D_H\chi_E|$.
We consider a positive vanishing sequence $\eps_{k}$ 
such that $\rho_{\eps_{k}} \ast \chi_{E} \weakstarto \overline{\chi_{E}}$ in $L^{\infty}(\Omega; |D_{H} \chi_{E}|)$. Indeed, $\rho_\ep\ast \chi_E$ is clearly uniformly bounded in $L^{\infty}(\Omega; |D_{H} \chi_{E}|)$, and therefore there exists at least a converging subsequence. 
We have first to prove \eqref{Leibniz_rule_Sobolev_reg_BV}.
We know that $\rho_{\eps} \ast \chi_{E} \in C^1_H(\Omega_{2 \eps}^{\cR}) \cap C(\G)$ by Lemma~\ref{l:BVD_Hfconv}. Choosing any $\phi \in C^{1}_{c}(H\Omega_{2 \eps}^{\cR})$
and taking into account \eqref{commutation conv derivative}, it follows that
\begin{equation} \label{weak_Leibniz_rule_Sobolev_reg_BV} \int_{\Omega_{2 \eps}^{\cR}} (\rho_{\eps} \ast \chi_{E}) \chi_{E} \div \phi \, dx = \int_{\Omega_{2 \eps}^{\cR}} \chi_{E} \div ( \phi  (\rho_{\eps} \ast \chi_{E}) ) \, dx - \int_{\Omega_{2 \eps}^{\cR}} \chi_{E} \ban{ \phi, \rho_{\eps} \ast D_{H} \chi_{E}} \, dx. \end{equation}
By Remark~\ref{Lip test1}, we get 
\begin{equation*} \int_{\Omega_{2 \eps}^{\cR}} (\rho_{\eps} \ast \chi_{E}) \chi_{E} \div \phi \, dx = - \int_{\Omega_{2 \eps}^{\cR}} (\rho_{\eps} \ast \chi_{E})  \, \ban{\phi, D_{H} \chi_{E}} - \int_{\Omega_{2 \eps}^{\cR}} \chi_{E} \ban{ \phi, \rho_{\eps} \ast D_{H} \chi_{E}} \, dx, \end{equation*}
which implies \eqref{Leibniz_rule_Sobolev_reg_BV}.
Thus, taking into account \eqref{commutation conv derivative meas} and \eqref{pointwise_upper_control}, for any open set $A \Subset \Omega$ such that $A \subset \Omega_{2 \eps}^{\cR}$, we obtain 
\begin{equation} \label{tot_var_inequality_loc} |D_{H}((\rho_{\eps} \ast \chi_{E}) \chi_{E})|(A) \le \int_{A} \rho_{\eps} \ast \chi_{E} \, d |D_{H} \chi_{E}| + \int_{E \cap A} \rho_{\eps} \ast |D_{H} \chi_{E}| \, dx. \end{equation}
Now we observe that
\begin{equation} \label{second_term_estimate}  
\begin{split}
\int_{E \cap A} \rho_{\eps} \ast |D_{H} \chi_{E}| \, dx &= \int_{\G} \int_\Omega \chi_{E \cap A}(x) \rho_{\eps}(y x^{-1}) \, d |D_{H} \chi_{E}|(y) \, dx \\
&= \int_\Omega (\rho_{\eps} \ast \chi_{E \cap A})(y) \, d |D_{H} \chi_{E}|(y), 
\end{split}
\end{equation}
since $\rho_{\eps}(xy^{-1}) = \rho_{\eps}(yx^{-1})$. We notice that $(\rho_{\eps} \ast \chi_{E \cap A}) \le (\rho_{\eps} \ast \chi_{E})$ and 
\[
(\rho_{\eps} \ast \chi_{E \cap A})(x)=0
\]
for any $x \notin A^{\cR,\ep}$. Taking into account this vanishing
property, along with \eqref{tot_var_inequality_loc}, \eqref{second_term_estimate}, 
\eqref{weak_conv_moll_measure_eq} and the lower semicontinuity of the total variation, we let $\eps = \eps_{k}$ and, for any open set $A \Subset \Omega$, we obtain
\begin{equation*} |D_{H} \chi_{E}|(A) \le 2 \int_{\overline{A}} \overline{\chi_{E}} \, d |D_{H} \chi_{E}|,
\end{equation*}
since $\overline{\chi_E} \ge 0$, as observed in Remark~\ref{r:sign_chi_E}, in the particular case $\nu = |D_{H} \chi_{E}|$.
This inequality can be refined by noticing that, given any open set $A \subset \Omega$, if we take an increasing sequence of open sets $A_{j}$ such that $A_{j} \Subset A_{j + 1}$ and $\bigcup_{j} A_{j} = A$, the regularity of the Radon measure $|D_{H} \chi_{E}|$ yields
\begin{equation} \label{key_estimate_one_tot_var}
\begin{split}
|D_{H} \chi_{E}|(A) &= \limsup_{j \to + \infty} |D_{H} \chi_{E}|(A_{j}) \\
&\le 2 \limsup_{j \to + \infty} \int_{\overline{A_{j}}} \overline{\chi_{E}} \, d |D_{H} \chi_{E}| \le 2 \int_{A} \overline{\chi_{E}} \, d |D_{H} \chi_{E}|.
\end{split}
\end{equation}
This means that $\displaystyle \overline{\chi_{E}}(x) \ge1/2$ for $|D_{H} \chi_{E}|$-a.e.\ $x\in\Omega$.
\begin{comment}
We observe that $\rho_{\ep_k}\ast\chi_\Omega$ is continuous 
in $\G$ and pointwise converges to one in $\Omega$, therefore
the equality $\rho_{\ep_k}\ast \chi_\Omega$ weakly$^*$ converges
to one in $\Omega$.
\end{comment}
Finally, we notice that also $\Omega \setminus E$ is a set of locally finite h-perimeter in $\Omega$ and the equality
\[
\rho_{\ep_k}\ast \chi_\Omega=\rho_{\ep_k}\ast \chi_E+
\rho_{\ep_k}\ast \chi_{\Omega\sm E}
\]
yields the weak$^*$ convergence of $\rho_{\ep_k}\ast \chi_{\Omega\sm E}$ to $1 - \overline{\chi_{E}}$ in $L^{\infty}(\Omega; |D_{H} \chi_{E}|)$. This implies that
$1-\overline{\chi_E} \ge 1/2$ at $|D_{H} \chi_{E}|$-a.e. point of $\Omega$,  
therefore our claim is achieved.
\end{proof}

\begin{lemma}\label{lemma:weak_conv_absolutely_continuous_perimeter_measures}
If $\gamma\in\cM(\Omega)$ is a nonnegative measure and $f_k\weakstarto f$ in $L^\infty(\Omega;\gamma)$ as $k\to\infty$, then for every $\theta\in L^1(\Omega;\gamma)$, setting $\nu=\theta\gamma$, we have 
\[
f_k \nu  \weakto f\nu
\]
in the sense of Radon measures on $\Omega$ and $f_{k} \weakstarto f$ in $L^{\infty}(\Omega; |\nu|)$.
\end{lemma}
\begin{proof}
For any $\phi \in C_{c}(\Omega)$, one clearly has $\phi \theta \in L^{1}(\Omega; \gamma)$ and so we get
\begin{equation*}
\int_{\Omega} \phi f_k \, d \nu =
\int_{\Omega} \phi \theta  f_k \, d \gamma \to  \int_{\Omega} \phi \theta  f\, d\gamma = \int_{\Omega} \phi  f\, d \nu.
\end{equation*}
We observe that $|\nu| = |\theta| \gamma$, and so, for any $\psi \in L^{1}(\Omega; |\nu|)$, we have $\psi |\theta| \in L^{1}(\Omega; \gamma)$. Thus, we obtain
\begin{equation*}
\int_{\Omega} \psi f_k \, d |\nu| =
\int_{\Omega} \psi |\theta|  f_k \, d \gamma \to  \int_{\Omega} \psi |\theta|  f\, d\gamma = \int_{\Omega} \psi  f\, d |\nu|,
\end{equation*}
concluding the proof.
\end{proof}

\begin{remark}\label{weak_conv_absolutely_continuous_perimeter_measures} 
By \eqref{limit1/2} and the previous lemma, we notice
that  
\[
(\rho_{\eps} \ast \chi_{E}) \nu \weakto (1/2) \nu,
\]
having $\nu = \theta |D_{H} \chi_{E}|$ and
$\theta \in L^1(\Omega; |D_{H} \chi_{E}|)$.
\end{remark}

\begin{lemma} \label{overline_D_chi_E_1_2_lemma}  Let $E \subset \Omega$ be a set of locally finite h-perimeter and $\rho \in C_c(B(0, 1))$ be a mollifier satisfying $\rho \ge 0$, $\rho(x) = \rho(x^{-1})$ and $\ds\int_{B(0, 1)} \rho(y) \, dy = 1$. Then, we have 
\begin{align} \label{overline_D_chi_E_1} \chi_{E} (\rho_{\eps} \ast D_{H} \chi_{E}) \mu & \weakto \frac{1}{2} D_{H} \chi_{E}, \\
\label{overline_D_chi_E_2} \chi_{\Omega \setminus E} (\rho_{\eps} \ast D_{H} \chi_{E}) \mu & \weakto \frac{1}{2} D_{H} \chi_{E}. \end{align}
\end{lemma}
\begin{proof} 
\begin{comment}
We choose $\phi \in C^{1}_{c}(\Omega, H \Omega)$ and we recall that $D_{H} \chi_{E} = - \nu_{E} |D_{H} \chi_{E}|$. By Proposition~\ref{overline_chi_E}, we have
\begin{align*} \lim_{\eps \to 0} \int_{\Omega} \ban{\phi, \chi_{E} \nabla_{H}(\rho_{\eps} \ast \chi_{E})} \, dx & = \lim_{\eps \to 0} \int_{\Omega} \chi_{E} \div(\phi (\rho_{\eps} \ast \chi_{E})) \, dx - \int_{\Omega} \chi_{E} (\rho_{\eps} \ast \chi_{E}) \div \phi \, dx \\
& = \lim_{\eps \to 0} \int_{\Omega} (\rho_{\eps} \ast \chi_{E}) \ban{\phi, \nu_{E}} \, d |D_{H} \chi_{E}| - \int_{\Omega} \chi_{E} (\rho_{\eps} \ast \chi_{E}) \div \phi \, dx \\
& = \int_{\Omega} \frac{1}{2} \ban{\phi, \nu_{E}} \, d |D_{H} \chi_{E}| - \int_{\Omega} \chi_{E} \div \phi \, dx \\
& = - \int_{\Omega} \frac{1}{2} \ban{\phi, \nu_{E}} \, d |D_{H} \chi_{E}|. \end{align*}
This shows \eqref{overline_D_chi_E_1}, by the density of $C^{1}_{c}(\Omega, H \Omega)$ in $C_{c}(\Omega, H \Omega)$ with respect to the sup norm.
\end{comment}
By \eqref{Leibniz_rule_Sobolev_reg_BV} and \eqref{commutation conv derivative meas}, we have
\begin{equation*} \chi_{E} (\rho_{\eps} \ast D_{H} \chi_{E}) \mu = \chi_{E} \nabla_{H} (\rho_{\eps} \ast \chi_{E}) \mu = D_{H}((\rho_{\eps} \ast \chi_{E}) \chi_{E}) - (\rho_{\eps} \ast \chi_{E}) D_{H} \chi_{E}  \quad\text{in $\mathcal{M}(\Omega_{2 \eps}^{\cR})$}.
\end{equation*}
Since for any $\phi \in C^{1}_{c}(H \Omega)$ we have $\mathrm{supp}(\phi) \subset \Omega_{2 \eps}^{\cR}$ for $\eps > 0$ sufficiently small, we get
\begin{equation*} \int_{\Omega} \phi \chi_{E} \nabla_{H} (\rho_{\eps} \ast \chi_{E}) \, dx = - \int_{\Omega} (\rho_{\eps} \ast \chi_{E}) \chi_{E} \, \div{\phi} \, dx - \int_{\Omega} (\rho_{\eps} \ast \chi_{E}) \ban{\phi, d D_{H} \chi_{E}}. \end{equation*}
We pass now to the limit on the right hand side, and, by Remark~\ref{weak_conv_absolutely_continuous_perimeter_measures}, we obtain that
\[- \int_{\Omega} (\rho_{\eps} \ast \chi_{E}) \chi_{E} \, \div{\phi} \, dx - \int_{\Omega} (\rho_{\eps} \ast \chi_{E}) \ban{\phi, d D_{H} \chi_{E}} 
\]
converges to 
\[
 - \int_{\Omega} \chi_{E} \, \div{\phi} \, dx - \int_{\Omega} \frac{1}{2}\ban{\phi, d D_{H} \chi_{E}} 
 \]
as $\ep\to0^+$. The last limit equals
\[
\int_{\Omega} \frac{1}{2} \ban{\phi, d D_{H} \chi_{E}}.
\]
Therefore, by the density of $C^{1}_{c}(H \Omega)$ in $C_{c}(H \Omega)$ with respect to the sup norm, we get \eqref{overline_D_chi_E_1}.
We observe that $\chi_{\Omega \setminus E} (\rho_{\eps} \ast D_{H} \chi_{E}) = (1 - \chi_{E})  \nabla_{H} (\rho_{\eps} \ast D_{H} \chi_{E})$, and so \eqref{overline_D_chi_E_2} follows from
the first local weak$^*$ convergence of \eqref{weak_conv_D_f} and from \eqref{overline_D_chi_E_1}. 
\end{proof}

\subsection{Precise representatives and mollifications}\label{sec:prec-representative}

We now introduce the notion of precise representative of a locally summable
function, which shall play an important role in the product rule for divergence-measure horizontal fields. However, due to our choice of mollifying functions by putting the mollifier on the left, we shall need to consider averages on right invariant balls.

\begin{definition}[Precise representative] \rm \label{precisedef}
Assume $u \in L^{1}_{\rm loc}(\G)$. Then
\begin{equation} \label{eq:representative_right} u^{*, \cR}(x) := \begin{cases} \displaystyle \lim_{r \to 0} \mean{B^\cR(x,r)} u(y) \, dy & \mbox{ if the limit exists } \\ 0 &  \mbox{ otherwise} \end{cases} 
\end{equation}
is the {\em precise representative} of $u$ on the balls with respect to the right invariant distance. We denote by $C_u^{\cR}$ the set of points such that the limit in \eqref{eq:representative_right} exists.
\end{definition}
It is clear that, by Theorem \ref{Lebesgue_diff_theorem}, all Lebesgue points of $u$ belong to $C_u^{\cR}$.
Given a measurable set $E \subset \Omega$, one can consider its points with density $\alpha \in [0, 1]$ with respect to the right invariant distance
\begin{equation*}
 E^{\alpha, \cR} := \set{x\in\G : \, \lim_{r \to 0} \frac{\mu(E \cap B^{\cR}(x, r))}{\mu(B^{\cR}(x, r))} = \alpha }, \end{equation*}
and hence define
\begin{equation} \label{eq:mtbR E_def} \mtbR E = \Omega \setminus (E^{1, \cR} \cup E^{0, \cR}). \end{equation}
Then, if we set $C_{\chi_{E}}^{\cR} = C_{E}^{\cR} $, we clearly have
\begin{equation*} C_{E}^{\cR} = \bigcup_{\alpha \in [0, 1]} E^{\alpha, \cR} \end{equation*}
and 
\begin{equation} \label{right_interior_chi_E_star} \chi_{E}^{*, \cR} = \chi_{E^{1, \cR}} \ \ \text{in} \ \ \Omega \setminus \mtbR E. \end{equation}

We state now a simple result which relates the pointwise limit of the mollification of a function $f \in L^{1}_{\rm loc}(\G)$ and the precise representative of $f$ on right invariant balls.

\begin{proposition} \label{pointwise limit moll} 
Let $\eta \in \Lip([0, 1])$ with $\eta\ge0$ and $\eta(1) = 0$, and $\rho(x) = \eta(d(x,0))$ for all $x\in\G$ such that $\int_{B(0, 1)} \rho(x) \, dx = 1$.
If $f \in L^1_{\rm loc}(\Omega)$ and $x \in C_{f}^{\cR}$, then we have
\[
(\rho_{\eps} \ast f)(x) \to f^{*, \cR}(x)\qs{ as }\eps \to 0.
\]
\end{proposition}
\begin{proof} 
Let $x \in C_{f}^{\cR}$ and $\eps > 0$ be sufficiently small, so that $B^{\cR}(x, \eps) \subset \Omega$. We assume first that $\eta$ is strictly decreasing. 
By Cavalieri's formula, we have
\begin{align*} (\rho_{\eps} \ast f)(x) & = \int_{B^{\cR}(x, \eps)} \eps^{-Q} \rho(\delta_{1/\eps}(xy^{-1})) f(y) \, d y \\
& = \int_{0}^{+ \infty} \int_{\{y\in B(x,\ep):\, \rho(\delta_{1/\eps}(xy^{-1})) > t \}} f(y) \eps^{-Q} \, d y \, dt \\
\left( t = \eta \left (\frac{r}{\eps} \right ) \right )\quad  & =  - \int_{0}^{\eps} \frac{1}{\eps}\, \eta' \left ( \frac{r}{\eps} \right ) \frac{1}{\eps^{Q}} \, \int_{B^{\cR}(x, r)} f(y) \, d y \, dr \\
\pa{ r = s \eps }\quad   &= - \int_{0}^{1} \eta'(s) \mu(B(0, 1)) s^{Q} \mean{B^{\cR}(x, s \eps)} f(y) \, d y \, ds. \end{align*}
The last equalities have been obtained from the standard area formula
for one-dimensional Lipschitz functions. 
Now, we use the existence of the limit of the averages of $f$ on the balls $B^{\cR}(x, s \eps)$.
This also implies that these averages are uniformly bounded with respect to $\ep$
sufficiently small. Thus, by Lebesgue's dominated convergence we obtain 
\begin{equation*} (\rho_{\eps} \ast f)(x) \to - \mu(B(0, 1)) f^{*, \cR}(x) \int_{0}^{1} \eta'(s) s^{Q} \, ds. \end{equation*}
We observe that the constant $C_{\eta, Q} := - \mu(B(0, 1)) \int_{0}^{1} \eta'(s) s^{Q} \, ds$ is independent from $f$. 
In addition, if we take $f \equiv 1$, we clearly have $(\rho_{\eps} \ast f) \equiv 1$ on $\Omega_{\eps}^{\cR}$. Thus, we can conlude that
\begin{equation*} - \mu(B(0, 1)) \int_{0}^{1} \eta'(s) s^{Q} \, ds = 1, \end{equation*}
and the statement follows.
We use now the well known fact that any Lipschitz continuous function in one variable can be written as the difference of two strictly decreasing functions to write $\eta = \eta_{1} - \eta_{2}$, with $\eta_{i} \in \Lip([0, 1])$, strictly decreasing and satisfying $\eta_{1}(1) = \eta_{2}(1)$. We can now repeat the above argument and so we obtain
\begin{align*} (\rho_{\eps} \ast f)(x) & \to - \mu(B(0, 1)) f^{*, \cR}(x) \int_{0}^{1} (\eta_{1}'(s) - \eta_{2}'(s)) s^{Q} \, ds \\
& = - f^{*, \cR}(x) \mu(B(0, 1)) \int_{0}^{1} \eta'(s) s^{Q} \, ds = f^{*, \cR}(x), \end{align*}
for any $x \in C_{f}^{\cR}$.
\end{proof}


\begin{remark}
We point out that the previous result also holds in the Euclidean case,
corresponding to a commutative group $\G$. It is then easy to see that 
the hypothesis that $\rho$ is radially symmetric cannot be removed. Indeed, 
we may consider $f = \chi_{E}$, where $E = (0, 1)^{2}$ and $\G=\R^{2}$, 
with $x = 0$. Clearly, $\chi_{E}^{*, \cR}(0) = 1/4$.  
If we choose 
\[
\rho \in C^{\infty}_{c}(B(0, 1) \cap (-1, 0)^{2}), \ \rho \ge 0, \quad \text{ with }\quad \int_{B(0, 1)} \rho(y) \, dy = 1,
\]
then we have
\begin{align*} (\rho_{\eps} \ast \chi_{E})(0) &= \int_{B(0, \eps) \cap E} \rho_{\eps}(- y) \, dy = \int_{B(0, \eps) \cap (-1, 0)^{2}} \rho_{\eps}(y) \, dy \\
&= \int_{B(0, 1) \cap (-1/\eps, 0)^{2}} \rho(y) \, dy = 1, \end{align*}
for any $\eps \in (0, 1]$.
\end{remark}

\section{Divergence-measure horizontal fields}

In this section we will introduce and study the function spaces of $p$-summable horizontal sections whose horizontal divergence is a Radon measure. In the sequel, $\Omega$ will denote a fixed open set of $\G$.

\subsection{General properties and Leibniz rules}\label{sect:leibniz}

By a little abuse of notation, for any $\mu$-measurable set $E$ we shall use the symbols $\|F\|_{p, E}$ and $\|F\|_{\infty, E}$ with the same meaning as in \eqref{Lp_norms} and \eqref{Linfty_norm}.

\begin{definition}[Divergence-measure horizontal field] \label{DMdef} A $p$-summable divergence-measure horizontal field is a field $F \in L^{p}(H \Omega)$ whose
distributional divergence $\div F$ is a Radon measure on $\Omega$.
\begin{comment}{;{\bf ....under modification}  that is, for any $\phi \in C^{1}_{c}(\Omega)$ we have
\begin{equation} \label{weak divergence}
 - \int_{\Omega} \ban{F, \nabla_H \phi} \, dx= \int_{\Omega} \phi \, d\, \div F. \end{equation}
\end{comment}
We denote by $\DM^p(H\Omega)$ the space of all $p$-summable divergence-measure horizontal fields, where $1 \le p \le \infty$.
A measurable section $F$ of $H\Omega$ is a {\em locally $p$-summable divergence-measure horizontal field} if, for any open subset $W \Subset \Omega$, we have $F \in \DM^p(HW)$. The space of all such section is denoted by $\DM^{p}_{\rm loc}(H \Omega)$. 
\end{definition}

It is easy to observe that, if $F = \sum_{j = 1}^{\m} F_{j} X_{j}$
and $F_{j} \in L^{p}(\Omega) \cap BV_{H}(\Omega)$ for all $j=1, \dots, \m$, then $F \in \DM^{p}(H \Omega)$.
We also notice that, from \eqref{eq:Ff_j} and \eqref{distributional_div_equiv}, the divergence-measure horizontal fields forms a subspace of the whole space of divergence-measure fields. Hence, if we denote by $T \Omega$ the tangent bundle of $\Omega$, we have $\DM^{p}(H \Omega) \subset \DM^{p}(T \Omega)$, for any $p \in [1, \infty]$, where $\DM^{p}(T \Omega)$ denotes the
classical space of divergence-measure fields with respect
to the Euclidean structure fixed on $\G$.
Actually $\DM^{p}(H \Omega)$ is a closed subspace of 
$\DM^{p}(T \Omega)$, according to the next remark.

\begin{remark} As in the Euclidean case (\cite[Corollary 1.1]{CF1}), $\DM^{p}(H \Omega)$ endowed with the following norm
\begin{equation*} \|F\|_{\DM^{p}(H \Omega)} := \|F\|_{p, \Omega} + |\div F|(\Omega) \end{equation*}
is a Banach space. Any Cauchy sequence $\{F_{k}\}$ is clearly a Cauchy sequence in $L^{p}(H\Omega)$, and so there exists $F \in L^{p}(H \Omega)$ such that $F_{k} \to F$ in $L^{p}(H \Omega)$. Then, the lower semicontinuity of the total variation and the
property of the Cauchy sequence yield
$F \in \DM^{p}(H \Omega)$ and $|\div (F - F_{k})|(\Omega) \to 0$.
\end{remark}

The following example shows that fields of $\DM^p(H\Omega)$
may have components that are not $BV$ functions.
It is a simple modification of an example of Chen-Frid, see
\cite[Example 1.1]{chen2001theory}.
 
\begin{example} \label{examples_fiels} Let $\G = \H^{1}$, be the first Heisenberg group, equipped with graded coordinates $(x, y, z)$ and horizontal left invariant vector fields $X_{1} = \partial_{1} - y \partial_{3}$ and $X_{2} = \partial_{2} + x \partial_{3}$.
We define the divergence-measure horizontal field
\begin{equation*} F(x, y, z) = \sin{\left ( \frac{1}{x - y} \right )} (X_{1} + X_{2}). \end{equation*}
It is plain to see that $F \in L^{\infty}(H \H^1)$, and that
\begin{equation*} \div F = X_{1} \sin{\left ( \frac{1}{x - y} \right )} + X_{2} \sin{\left ( \frac{1}{x - y} \right )} = 0, \end{equation*}
in the sense of Radon measures, but the components of
$F$ are not $BV$.
\end{example}

\begin{remark}
We notice that, for a given $F \in \DM^{p}(T \Omega)$, if we denote by $F_{H}$ its projection on the horizontal subbundle with respect to a fixed left invariant Riemannian metric that makes $X_1, \dots , X_\q$ orthonormal,
we may not get $F_H \in \DM^{p}(H \Omega)$.
Let us consider the Heisenberg group $\H^1$ identified with $\R^3$, as 
in the previous example, along with the vector fields $X_1$, $X_2$,
and define $X_3=\der_3$.

Let us consider the following measurable vector field
\begin{equation*} G(x, y, z) = \sin{\left ( \frac{1}{x - z} \right )} (\partial_{1} + \partial_{2} + \partial_{3}). \end{equation*}
We clearly have $G \in \DM^{\infty}(T\H^1)$, i.e.
$G$ is a divergence-measure field.
However, if we consider its projection onto horizontal fibers
\begin{equation*} G_{H}(x, y, z) = \sin{\left ( \frac{1}{x - z} \right )} (X_{1} + X_{2}),
\end{equation*}
for any $x \neq z$, we have
\begin{equation*} \div\, G_{H}(x, y, z) = - \frac{1 + x + y}{(x - z)^{2}} \cos{\left ( \frac{1}{x - z} \right )}, \end{equation*}
which is not a locally summable function in any neighborhood of $\{ x = z \}$. 
This shows that $\div\, G_{H} \notin \mathcal{M}(\H^{1})$.
\end{remark}

We show now an easy extension result (see also \cite[Remark 2.20]{comi2017locally}).

\begin{remark} \label{div comp supp extension}\rm
If $1\le p\le \infty$ and $F \in \DM^p(H\Omega)$ has compact support in $\Omega$, then its trivial extension 
\begin{equation*}
\hat{F}(x) := \begin{cases} F(x)  & \mbox{if} \ \ x \in \Omega \\ 0 & \mbox{if} \ \ x \in \G \setminus \Omega, \end{cases}
\end{equation*}
belongs to $\DM^p(H \G)$.
Indeed, since $\hat{F} \in L^{p}(H \G)$ and for any 
$\phi \in C^{\infty}_{c}(\G)$ and a fixed $\xi \in C^{\infty}_{c}(\Omega)$ that equals one on a neighborhood of the support of $F$, we have
\begin{equation}\label{eq:extFOmega} 
\begin{split}
\int_{\G} \lan\hat{F}, \nabla_H\phi \ran \, dx  &=\int_\Omega \lan\hat{F}, \nabla_H\phi \ran \, dx \\
&= \int_{\Omega} \lan F, \nabla_H (\xi\phi)\ran \, dx + \int_{\Omega} \lan F, \nabla_H ( (1 - \xi) \phi)\ran \, dx\\
& =-\int_\Omega \phi\, d( \div F\res\xi) =-\int_\G \phi\, d( \div F\res\xi),
\end{split}
\end{equation}
where we denote by $\div F \res \xi$ the signed Radon measure on $\G$ such that
\[
\div F\res \xi (E)=\int_{\Omega\cap E} \xi \, d\div F
\]
for every relatively compact Borel subset $E\subset\G$.
Thus, we have shown that $\hat F\in \DM^p(H\G)$ and $\div \hat F=\div F\res\xi$.
The equalities of \eqref{eq:extFOmega} imply that 
the restriction of $\div\hat F$ to $\Omega$ coincides with $\div F$ 
and in particular $|\div \hat F|(\Omega)=|\div F|(\Omega)$.
The same equalities also imply that $|\div \hat F|(\G\sm\Omega)=0$. 
\end{remark}

As a consequence, we can prove the following result concerning fields with compact support, which can be seen as the easy case of the Gauss--Green formula, since there are no boundary terms. A similar result has been proved in the Euclidean setting in \cite[Lemma 3.1]{comi2017locally}.

\begin{lemma} \label{DMcomptsupp} If $1\le p\le \infty$ and $F \in \DM^{p}(H\Omega)$ has compact support in $\Omega$, then
\begin{equation*} \div F(\Omega) = 0. \end{equation*}
\end{lemma}
\begin{proof} Since $F$ has compact support in $\Omega$, the extension
$\hat F$ defined in Remark~\ref{div comp supp extension} shows that $\hat{F} \in \DM^p(H\G)$, $\div \hat F = \div  F$ as signed Radon measure in $\Omega$ and 
$\div \hat F $ is the null measure when restricted to $\G \setminus \Omega$. 
As a consequence, if $\phi \in C^{\infty}_{c}(\G)$ is chosen such that $\phi = 1$ on a neighborhood of $\Omega$, then 
\begin{equation*}
\int_{\G} \phi\, d \div\hat  F = \int_\Omega  d \div\hat F=\div F(\Omega).
\end{equation*}
By definition of distributional divergence, there holds
\begin{equation*} 
\int_{\G} \phi \, d \div\hat  F = - \int_\Omega \lan F, \nabla_H \phi\ran \, dx  = 0, \end{equation*}
since $F$ has support inside $\Omega$ and $\phi$ is constant on this set. This
concludes the proof.
\end{proof}

We show now a result concerning the absolute continuity properties of $\div F$ with respect to the $\SHaus{\alpha}$-measure, for a suitable $\alpha$ related to the summability exponent $p$. This is a generalization of a known result in the Euclidean case (\cite[Theorem~3.2]{Silhavy}).

\begin{theorem} \label{absolute continuity}
If $F \in \DM^p_{\rm loc}(H \Omega)$ and $\frac{Q}{Q - 1} \le p < + \infty$, then $|\div F|(B) = 0$ for any Borel set $B\subset\Omega$ of $\sigma$-finite $\SHaus{Q - p'}$ measure. If $p = \infty$, then $|\div F| \ll \SHaus{Q - 1}$.
\end{theorem}
\begin{proof}
Let $\frac{Q}{Q - 1} \le p < + \infty$. It suffices to consider a Borel set $B$ such that $\SHaus{Q - p'}(B) < \infty$. We can use the Hahn decomposition in order to split $B$ into $B_{+} \cup B_{-}$, in such a way that $\pm \div F \res B_{\pm} \ge 0$, thus reducing ourselves to show that $\div F(K) = 0$ for any compact set $K \subset B_{\pm}$. Without loss of generality, we consider $K \subset B_{+}$.
Let $\varphi : \G \to [0, 1]$ defined as follows
\begin{equation*} \varphi(x) := \begin{cases} 1 & \mbox{if} \, d(x, 0) < 1 \\ 2 - d(x, 0) & \mbox{if} \, 1 \le d(x, 0) \le 2 \\ 0 & \mbox{if} \, d(x, 0) > 2 \end{cases}. \end{equation*}
It is clear that $\varphi \in \Lip_{c}(\G)$, therefore it is also differentiable $\mu$-a.e. with $|\nabla_H\varphi| \le L$ for some constant $L>0$, by Theorem \ref{MSC-Rademacher}. 
\\
We notice that since $\SHaus{Q - p'}(K) < \infty$, then $\mu(K) = 0$. This implies that for any $\eps >0$ there exists an open set $U \Subset \Omega$ such that $K \subset U$ and $\| F \|_{p, U} < \eps$, because $|F| \in L^{p}_{\rm loc}(\Omega)$. In addition, we can ask that such an $U$ satisfies $|\div F|(U \setminus K) < \eps$, because of the regularity of Radon measures.
\\ 
It is clear that there exists $\delta > 0$ such that for any $0 < 2 r < \delta$ and for any open ball $B(x, r)$ which intersects $K$ we have $B(x, 2 r) \subset U$. Then we can select a covering of $K$ (which can be also taken finite by compactness) of such balls $\{B(x_{j}, r_{j})\}_{j \in J}$ and so, by the definition of spherical measure, we have
\begin{equation*} \sum_{j \in J} r_{j}^{Q - p'} < \SHaus{Q - p'}(K) + 1, \end{equation*}
for $\delta$ small enough. 
\\
We set $\varphi_{j}(x) := \varphi(\delta_{1/r_{j}}(x_{j}^{-1} x))$ and $\psi(x) := \sup \{ \varphi_{j}(x) : j \in J \}$. It is easy to see that $0 \le \psi \le 1$, $\varphi_{j}$ is supported in $B(x_{j}, 2 r_{j})$, $\psi \in \Lip_{c}(\Omega)$, $\mathrm{supp}(\psi) \subset U$ and $\psi \equiv 1$ on $K$. Then, by Remark \ref{Lip test}, we have
\begin{equation*} \div F(K) = \int_{K} \psi \, d \div F = - \int_{U} \ban{F, \nabla_{H} \psi} \, dx - \int_{U \setminus K} \psi \, d \div F, \end{equation*}
which implies
\begin{equation*} \div F(K) < \|F\|_{p, U} \| \nabla_{H} \psi \|_{p', U} + \eps < \eps ( \| \nabla_{H} \psi \|_{p', U} + 1). \end{equation*}
Since $\psi$ is the maximum of a finite family of functions, we have $\nabla_{H} \psi(x) = \nabla_{H} \varphi_{j}(x)$ for some $j \in J$ and $\mu$-a.e. $x \in \Omega$. Indeed, Theorem~\ref{MSC-Rademacher} shows that Lipschitz functions are differentiable $\mu$-a.e., moreover, $\psi(x) = \varphi_{j}(x)$ in the open set $\{ \varphi_{j} > \varphi_{i}, \, \forall i \neq j \}$, 
while $\nabla_H \varphi_{j}(x) = \nabla_H \varphi_{i}(x)$ for $\mu$-a.e. $x$ on the set $\{ \varphi_{j} = \varphi_{i}\}$.
Then
\begin{align*} \int_{U} |\nabla_{H} \psi|^{p'} \, dx & \le \sum_{j \in J} \int_{U} |\nabla_{H} \varphi_{j}|^{p'} \, dx = \sum_{j \in J} \int_{B(x_{j}, 2 r_{j})} |\nabla_{H} \varphi_{j}|^{p'} \, dx \\
& \le 2^{Q} \mu(B(0, 1)) L^{p'} \sum_{j \in J} r_{j}^{Q - p'} \le 2^{Q} L^{p'} \mu(B(0, 1)) (\SHaus{Q - p'}(K) + 1). \end{align*}
This implies $$0 \le \div F(K) \le \eps (1 + 2^{\frac{Q}{p'}} L \mu(B(0, 1))^{\frac{1}{p'}} (\SHaus{Q - p'}(K) + 1)^{\frac{1}{p'}})$$ and, since $\eps$ is arbitrary, we conclude the proof.

In the case $p = \infty$, we proceed similarly by considering a Borel set $B$ such that $\SHaus{Q - 1}(B) = 0$ and a compact subset of $B_{\pm}$. 
For any $\eps > 0$, there exists an open set $U$ satisfying $K \subset U \Subset \Omega$ and $|\div F|(U \setminus K) < \eps$, as before. Now, there exists a $\delta > 0$ small enough such that we can find a finite open covering $\{B(x_{j}, r_{j})\}_{j \in J}$, $2 r_{j} < \delta$, of $K$, which satisfies $\sum_{j \in J} r_{j}^{Q - 1} < \eps$, and $B(x_{j}, 2 r_{j}) \subset U$ whenever $B(x_{j}, r_{j}) \cap K \neq \emptyset$.

It is clear that
\begin{equation*} \left | \int_{U} \ban{F, \nabla_{H} \psi} \, dx \right | \le \|F\|_{\infty, U} \| \nabla_{H} \psi \|_{1, U} \end{equation*}
and that
\begin{align*} \int_{U} |\nabla_{H} \psi| \, dx & \le \sum_{j \in J} \int_{\Omega} |\nabla_{H} \varphi_{j}| \, dx = \sum_{j \in J} \int_{B(x_{j}, 2r_{j})} |\nabla_{H} \varphi_{j}| \, dx \\
& \le 2^{Q} L \mu(B(0, 1)) \sum_{j \in J} r_{j}^{Q - 1} < 2^{Q} L \mu(B(0, 1)) \eps. \end{align*}
Thus, we conclude that
\begin{equation*} \div F(K) =  - \int_{U} \ban{F, \nabla_{H} \psi} \, dx - \int_{U \setminus K} \psi \, d \div F < \eps (1 + 2^{Q} L \mu(B(0, 1)) \|F\|_{\infty, U}), \end{equation*}
which implies $\div F(K) = 0$, since $\eps$ is arbitrary. \end{proof}

We notice that there is a precise way to compare $\SHaus{Q - p'}$ and the Euclidean Hausdorff measure $\Haus{\q - p'}_{|\cdot|}$ on a stratified group $\G$ of topological dimension $\q$, as shown in \cite{balogh2009sub}. In particular, 
\cite[Proposition~3.1]{balogh2009sub} implies that 
\begin{equation*} \SHaus{Q - p'} \ll \Haus{\q - p'}_{|\cdot|}. \end{equation*}
This shows that Theorem~\ref{absolute continuity} is coherent with the Euclidean case (\cite[Theorem~3.2]{Silhavy}), and that the divergence-measure horizontal fields have finer absolute continuity properties than the general ones.

Now we prove a first case of Leibniz rule between an essentially bounded divergence-measure horizontal field and a scalar Lipschitz function, whose gradient is in $L^{1}(H \Omega)$.

\begin{proposition}
If $F \in \DM^{\infty}(H \Omega)$ and $g \in L^{\infty}(\Omega)\cap \Lip_{H, {\rm loc}}(\Omega)$,
with $\nabla_{H} g \in L^{1}(H \Omega)$, then $g F \in \DM^{\infty}(H \Omega)$ and the following formula holds
\begin{equation} \label{productruleLip}
\mathrm{div}(g F) = g \div F + \ban{F, \nabla_{H} g} \mu.   \end{equation}
\end{proposition}
\begin{proof}
It is clear that $gF \in L^{\infty}(H \Omega)$. For any $\phi \in C^{1}_{c}(\Omega)$ with $\| \phi \|_{\infty} \le 1$, by Remarks \ref{product rule Lip H} and \ref{Lip test}, we have
\begin{align*} \int_{\Omega} \ban{gF, \nabla_{H}\phi} \, dx & = \int_{\Omega} \ban{F, \nabla_{H} (g \phi)} \, dx - \int_{\Omega} \phi \ban{F, \nabla_{H} g} \, dx \\
& = - \int_{\Omega} \phi g \, d \div F - \int_{\Omega} \phi \ban{F, \nabla_{H} g} \, dx, \end{align*}
which clearly implies that $gF \in \DM^\infty(H \Omega)$ and \eqref{productruleLip} holds.
\end{proof}

We have now all the tools to establish a general product rule for essentially bounded divergence-measure horizontal fields and $BV_{H}$ functions, see Theorem~\ref{productrule}. This is one of the main ingredients in the proof of the Gauss--Green formulas.

\begin{proof}[Proof of Theorem~\ref{productrule}] 
We notice that \eqref{productruleLip} holds for every $g_{\eps}$ with $\ep>0$ and we also have $|g_{\eps}(x)| \le \|g\|_{\infty, \Omega}$ for any $x \in \Omega$. The family $\{g_{\tilde \eps_k}\}$ is then equibounded in $L^{\infty}(\Omega; |\div F|)$ and there exists $\tilde{g} \in L^{\infty}(\Omega; |\div F|)$ and a subsequence $\ep_k\to0$ such that $g_{\eps_{k}} \weakstarto \tilde{g}$. It follows that 
\begin{equation*} \int_{\Omega} \phi g_{\eps_{k}} \, d \div F \to \int_{\Omega} \phi \tilde{g} \, d \div F \end{equation*}
for any $\phi \in L^{1}(\Omega; |\div F|)$.
In particular, the previous convergence holds for any $\phi \in C_{c}(\Omega)$, and so $g_{\eps_{k}} \div F \weakto \tilde{g} \div F$ in $\mathcal{M}(\Omega)$.

Now we show that $\{ \div (g_{\eps} F) \}$ is uniformly bounded in $\mathcal{M}(\Omega)$: by \eqref{productruleLip}, we obtain that
\begin{align*} \left | \int_{\Omega} \phi \, d \div(g_{\eps} F) \right | & \le \left | \int_{\Omega} \phi g_{\eps} \, d \div F \right | + \left | \int_{\Omega} \phi \ban{F, \nabla_{H}(g_{\eps})} \, dx \right | \\
& \le \|\phi\|_{\infty, \Omega} \|g\|_{\infty, \Omega} |\div F|(\Omega) + \|F\|_{\infty, \Omega} \int_{\Omega} |\phi| |\nabla_{H} g_{\eps}| \, dx. 
\end{align*}
As a result, considering $\mathrm{supp}(\phi) \subset \Omega_{2 \eps}^{\cR}$, by \eqref{total_var_convol_convergence} we conclude that
\begin{equation} \label{uniformbounddiv} |\div (g_{\eps} F)|(\Omega_{2 \eps}^{\cR}) \le \|g\|_{\infty, \Omega} |\div F|(\Omega) + \|F\|_{\infty, \Omega} |D_{H} g|(\Omega). \end{equation}
We have shown that $\{\div (g_{\eps} F)\}$ is uniformly bounded in $\mathcal{M}(\Omega')$ for any open set $\Omega' \Subset \Omega$, and so up to extracting a further subsequence that we relabel as $\ep_k$, 
the sequence  $\{\div (g_{\eps_{k}} F)\}$ is a locally weakly$^{*}$ converging subsequence. However, it is clear that $\div(g_{\eps_k} F)$ weakly$^{*}$ converges to $\div (g F)$ in the sense of distributions, and that $C^{\infty}_{c}(\Omega)$ is dense in $C_{c}(\Omega)$. Therefore, by uniqueness of weak$^{*}$ limits, we conclude that $\div(g_{\eps_k} F) \weakto \div (g F)$ in $\mathcal{M}(\Omega)$.

Thus, $(F, \nabla_{H} g_{\eps_{k}})$ is weakly$^*$ convergent, being the difference of two weakly$^*$ converging sequences, and taking into account \eqref{productruleLip} we get
\begin{equation} \label{nabla term conv} 
(F, \nabla_{H} g_{\eps_{k}}) \weakto (F, D_{H} g) := \div (g F) - \tilde{g} \div F. 
\end{equation}
In relation to $(F, D_{H} g)$, we first argue as in Lemma~\ref{conv_meas_bounded_function}. For any $\phi \in C_{c}(\Omega)$, we have
\begin{align*} \left | \int_{\Omega} \phi \, d (F, D_{H} g) \right | &= \lim_{\eps_{k} \to 0} \left | \int_{\Omega} \phi \ban{F, \nabla_{H} g_{\eps_{k}}} \, dx \right |  \le \|F\|_{\infty, \Omega}\, \limsup_{\eps_{k} \to 0} \int_{\Omega} |\phi| |\nabla_{H} g_{\eps_{k}}| \, dx \\
& \le \|F\|_{\infty, \Omega}\, \limsup_{\eps_{k} \to 0} \int_{\Omega} |\phi| (\rho_{\eps_{k}} \ast |D_{H} g|) \, dx\\
& =\|F\|_{\infty, \Omega}\, \lim_{\eps_{k} \to 0} \int_{\Omega} (\rho_{\eps_{k}} \ast |\phi|) \, d |D_{H} g|  = \|F\|_{\infty, \Omega} \int_{\Omega} |\phi| \, d |D_{H} g|, \end{align*}
where the second inequality follows by \eqref{pointwise_upper_control}, since $\mathrm{\supp}(\phi) \subset \Omega_{2 \eps_{k}}^{\cR}$ for $\eps_{k}$ small enough. The subsequent equality 
is a consequence of \eqref{exchanging_convolution},
therefore proving \eqref{abs_cont_pairing_infty}.
The decomposition \eqref{eq:BVdec-as} in our case yields
\[
D_{H} g = \nb{g}\, \mu + D_H^{\rm s} g,
\]
where $\nabla_Hg$ is also characterized as the approximate differential
of $g$, \cite[Theorem~2.2]{AmbMag2003}.
We aim to show that 
\[
(F, D_{H} g)^{\rm a} \mu = \ban{F, \nb{g}} \mu \qandq 
(F, D_{H} g)^{\rm s} = (F,D_H^{\rm s} g),
\]
for some measure $(F,D_H^{\rm s} g)\in\cM(\Omega)$ that
is absolutely continuos with respect to $|D_H^{\rm s}g|$.
Indeed, by \eqref{commutation conv derivative meas} we get
$\nabla_{H} g_{\eps} = \rho_{\eps} \ast D_{H} g$ on $\Omega'$ for every fixed 
open set $\Omega'\Subset\Omega$ and $\ep>0$ sufficiently small.
On this open set we have
\beq\label{eq:conv_dec_ep} 
\ban{F, \nabla_{H} g_{\eps}} = \ban{F, \rho_{\eps} \ast \nb{g}} + \ban{F, \rho_{\eps} \ast D_H^{\rm s} g}. 
\eeq
By Lemma~\ref{conv_meas_bounded_function} the 
measures $\ban{F,\rho_\ep\ast D_H^{\rm s} g}\mu$ are uniformly
bounded, so that possibly selecting a subsequence of $\ep_k$,
denoted by the same symbol, there exists 
$(F,D_H^{\rm s} g)\in\cM(\Omega)$ such that
\begin{equation*} 
\ban{F, \rho_{\eps_{k}} \ast D_H^{\rm s} g} \mu \weakto (F, D_{H}^{\rm s} g)
\end{equation*}
and applying again Lemma~\ref{conv_meas_bounded_function} we get
\beq\label{eq:FDs}
|(F,D_H^{\rm s} g)|\le \|F\|_{\infty,\Omega}\, |D_H^{\rm s} g|.
\eeq
Since $\nb{g}\in L^1(H\Omega)$, we clearly have
$\rho_{\eps} \ast \nb{g} \to  \nb{g}$ in $L^1(H \Omega)$, which yields the following weak$^*$ convergence
\begin{equation*} \ban{F, \rho_{\eps} \ast \nb{g}} \mu \weakto \ban{F, \nb{g}} \mu
\end{equation*}
in $\cM(\Omega)$. Since \eqref{eq:conv_dec_ep} holds
on every relatively compact open subset of $\Omega$, 
we get
\begin{equation} \label{first decomposition} 
(F, D_{H} g) = \ban{F, \nb{g}}\mu + (F, D_{H}^{\rm s} g). 
\end{equation}
Due to \eqref{eq:FDs} the previous sum is made by
mutually singular measures, then showing that
\[
(F, D_H g)^{\rm a} \mu= \ban{F, \nb{g}} \mu \qandq
(F, D_H g)^{\rm s}= (F,D_H^{\rm s} g),
\]
hence \eqref{pairing_decomposition_parts} holds.
\end{proof}

\begin{proposition} \label{tilde_g_characterization} Let $F \in \DM^{\infty}(H \Omega)$ and let $g \in L^{\infty}(\Omega)$
with $|D_Hg|(\Omega) < +\infty$. If we define $g_{\eps} := \rho_{\eps} \ast g$ using the mollifier $\rho$ of Proposition~\ref{pointwise limit moll}, then any weak$^{*}$ limit point $\tildef{g} \in L^{\infty}(\Omega; |\div F|)$ of some subsequence $g_{\eps_{k}}$ satisfies the property 
\begin{equation*} \tilde{g}(x) = g^{*, \cR}(x) \ \ \text{for} \ \ |\div F|\text{-a.e.} \ x \in C_{g}^{\cR}, \end{equation*} 
where $C_g^{\cR}$ is introduced in Definition~\ref{precisedef}.
In addition, if $g \in L^{\infty}(\Omega) \cap C(\Omega)$ and $\nabla_H g \in L^{1}(H \Omega)$, then $\tildef{g}(x) = g(x)$ for $|\div F|$-a.e. $x \in \Omega$ and
\begin{equation} \label{product rule eq p} 
\div (g F) = g\, \div F + \ban{F, \nabla_H g} \mu. 
\end{equation}
\end{proposition}
\begin{proof}
Let $g_{\eps_{k}} \weakstarto \tildef{g}$ in $L^{\infty}(\Omega; |\div F|)$.
By Proposition \ref{pointwise limit moll}, we know that $$g_{\eps_{k}}(x) \to g^{*, \cR}(x)$$ for any $x \in C_{g}^{\cR}$. If we choose as test function $\phi = \chi_{C_{g}^{\cR}} \psi$, for some $\psi \in L^{1}(\Omega; |\div F|)$, we have
\begin{equation*} \int_{\Omega} \phi g_{\eps_{k}} \, d \div F = \int_{C_{g}^{\cR}} \psi g_{\eps_{k}} \, d \div F \to \int_{C_{g}^{\cR}} \psi g^{*, \cR} \, d \div F \end{equation*}
by Lebesgue's theorem with respect to the measure $|\div F|$. Since $\psi$ is arbitrary, this implies $\tilde{g}(x) = g^{*, \cR}(x)$ for $|\div F|$-a.e. $x \in C_{g}^{\cR}$.
Let now $g \in L^{\infty}(\Omega) \cap C(\Omega)$ with $\nabla_H g \in L^{1}(H \Omega)$. It is then clear that $\tildef{g}(x) = g(x)$ for $|\div F|$-a.e. $x \in \Omega$, since $g^{*, \cR}(x) = g(x)$ for any $x \in \Omega$, being $g$ continuous. In addition, since $D_{H} g$ has no singular part, \eqref{pairing_decomposition_parts} implies immediately \eqref{product rule eq p}.
\end{proof}

\begin{remark}\rm  We stress that in Theorem~\ref{productrule} the pairing term $(F, D_{H} g)$ depends on the particular sequence $g_{\eps_{k}}$, and therefore on $\tilde{g}$. In order to obtain uniqueness, one should be able to show that there exists only one accumulation point $\tilde{g}$ of $g_{\eps}$. For instance, this happens in the Euclidean case $\G = \R^{n}$, in which $\tilde{g} = g^{*} ( = g^{*, \cR})$ $\Haus{n - 1}$-a.e. However, it is possible to impose some more conditions on the measures $\div F$ and $|D_{H} g|$ under which $\tilde{g}$ and $(F, D_{H}g)$ are uniquely determined.
\end{remark}

\begin{corollary} \label{productrulecor} Let $F \in \DM^{\infty}(H \Omega)$ and let $g \in L^{\infty}(\Omega)$ with $|D_Hg|(\Omega)<+\infty$. Let $\divFs $ and $D_{H}^{\rm s}g$ be the singular parts of the measures $\div F$ and $D_{H} g$. If we also assume that $|\divFs|$ and $|D_{H}^{\rm s} g|$ are mutually singular measures, then we have
\begin{equation} \label{product rule eq cor} 
\div (g F) = \lls g\, \divFa + \ban{F, \nb{g}}\rrs \mu +
\tilde{g}\, \divFs + (F, D_{H}^{\rm s} g), \end{equation}
where $\tilde{g} \in L^{\infty}(\Omega; |\div F|)$ and $(F, D_H^{\rm s} g) \in \mathcal{M}(\Omega)$ are defined as in Theorem \ref{productrule}
and the singular measures $\tilde{g}\,\divFs$ and $(F,D_H^{\rm s} g)$ are uniquely determined by $g$ and $F$.
In particular, if $|\div F| \ll \mu$, we have
\begin{equation} \label{product rule eq Leb} 
\div (g F) = \lls g\,\divFa + \ban{F,\nb{g}}\rrs \mu + (F, D_H^{\rm s} g). \end{equation}
\end{corollary}
\begin{proof}
It is well known that we can decompose the measures $\div F$ and $(F, D_{H} g)$ in their absolutely continuous and singular parts.
By Theorem \ref{productrule}, we know that 
\begin{equation*} 
\div(g F) = \lls \tilde{g}\,\divFa + \ban{F, \nb{g}}\rrs \mu + 
\tilde{g}\, \divFs + (F, D_H^{\rm s}g). 
\end{equation*}
We also have $\tilde{g}(x) = g(x)$ at $\mu$-a.e.\ $x \in \Omega$ in that  $g_{\eps}$ converges to $g$ in $L^{1}_{\rm loc}(\Omega)$ for any mollification of $g$. This implies that $\tilde{g}\,\divFa \mu = g\,\divFa \mu$ in the sense of Radon measures. It follows that \eqref{product rule eq cor} holds and clearly
\begin{equation*} 
\tilde{g}\, \divFs + (F, D_H^{\rm s}g) = 
\div (g F) - \lls g \, \divFa + \ban{F, \nb{g}}\rrs \mu.
\end{equation*}
We have shown that the singular measures on the left hand side are uniquely determined by $g$ and $F$, since the right hand side is uniquely determined and the two measures are also mutually singular.
Indeed we have $|(F, D^{\rm s}_Hg)| \le \|F\|_{\infty, \Omega} |D^{\rm s}_H g|$ by Theorem \ref{productrule}.
To conclude the proof, we observe that the condition $|\div F| \ll \mu$ clearly gives $\divFs = 0$, so \eqref{product rule eq Leb} immediately follows.
\end{proof}

\begin{remark} It is clear that one can obtain \eqref{product rule eq Leb} if we have $F \in L^{\infty}(H \Omega)$ with $\div F \in L^{1}(\Omega)$ and $g \in L^{\infty}(\Omega) \cap BV_{H}(\Omega)$. \end{remark}

\begin{remark} Under no additional assumption on $F \in \DM^{\infty}(H \Omega)$ and $g \in L^{\infty}(\Omega)$ with $|D_Hg|(\Omega)<+\infty$, we can always decompose the term $\tilde{g} \div F$. Indeed, we have
\begin{equation} \label{g_divF_decomp} 
\tilde{g}\, \div F = g \, \divFa \mu + g^{*, \cR}\,
\divFs \res C_{g}^{\cR} + \tilde{g}\, \divFs \res (\Omega \setminus C_{g}^{\cR}). \end{equation}
It follows that $\tilde{g} \div F$ is uniquely determined by $\div F$ and $g$ if $|\divFs|(\Omega \setminus C_{g}^{\cR}) = 0$.
\end{remark}

\section{Interior and exterior normal traces}\label{sect:normaltraces}

In this section we introduce interior and exterior normal traces for a divergence-measure field. The absence of sufficient regularity for the
reduced boundary (Definition~\ref{d:RedBdry}) does not guarantee their uniqueness a priori. However, the next section
will present different conditions that lead to a unique normal
trace and a corresponding Gauss--Green theorem.

Let $F \in \DM^{\infty}(H \Omega)$ and $E \subset \Omega$ be a set of finite h-perimeter. Let $\rho \in C_{c}(B(0, 1))$ be a nonnegative mollifier satisfying $\rho(x) = \rho(x^{-1})$ and $\int \rho \, dx = 1$, and $\eps_{k}$ be a suitable vanishing sequence such that
\beq\label{eq:chiEFconvol}
\begin{array}{c} 
\ban{\chi_{E} F,\rho_{\eps_{k}} \ast D_H\chi_{E}} \mu \weakto (\chi_{E} F, D_{H} \chi_{E}) \\
\ban{\chi_{\Omega \setminus E} F,\rho_{\eps_{k}} \ast D_H\chi_{E}} \mu \weakto (\chi_{\Omega \setminus E} F, D_{H} \chi_{E})
\end{array}\quad \text{in $\;\mathcal{M}(\Omega)$.}
\eeq
The existence of such converging subsequences follows from Lemma~\ref{conv_meas_bounded_function}, which implies also the estimates
\begin{eqnarray}\label{eq:chiEF_ie}
&|(\chi_{E} F, D_{H} \chi_{E})| \le \|F\|_{\infty, E}\, |D_{H} \chi_{E}|\quad\text{and} \nonumber\\
&|(\chi_{\Omega\sm E} F, D_{H} \chi_{E})| \le \|F\|_{\infty, \Omega \setminus E}\, |D_{H} \chi_{E}|.
\end{eqnarray}

\begin{comment}
The existence of this sequence follows by a standard diagonalization argument joined with the weak$^*$ compactness of bounded sequences of measures,
taking into account that the converging measures in \eqref{eq:chiEF} 
are bounded on any open set $\Omega'\Subset\Omega$ in view of \eqref{pointwise_upper_control}. 
In fact, we may consider an increasing sequence of open subsets $\Omega_k\nearrow\Omega$,
hence we find a subsequence $\ban{\chi_{E} F, \nabla_{H}(\rho_{\eps_{\alpha_1(k)}} \ast \chi_{E})} \mu\res\Omega_1$ weakly$^*$ that converges to $\nu_1\in\cM(\Omega_1)$. 
We continue this process considering 
$\ban{\chi_{E} F, \nabla_{H}(\rho_{\eps_{\alpha_1(\alpha_2(k))}} \ast \chi_{E})}$ converging 
on $\Omega_2$ and finally take the diagonal subsequence
$\ban{\chi_{E} F, \nabla_{H}(\rho_{\eps_{\alpha_1\circ\cdots\alpha_k(k))}} \ast \chi_{E})}$ 
that weakly$^*$ converges on every relatively compact open subset of $\Omega$.
\end{comment}
It is worth to mention that another definition equivalent to \eqref{eq:chiEFconvol} is possible. Employing formula \eqref{commutation conv derivative meas}, we obtain 
\beq\label{eq:chiEF_equivalence}
\begin{array}{c} 
\ban{\chi_{E} F,\rho_{\eps_{k}} \ast D_H\chi_{E}} = \ban{\chi_{E} F, \nabla_{H}(\rho_{\eps_{k}} \ast \chi_{E})}  \\ 
\ban{\chi_{\Omega \setminus E} F,\rho_{\eps_{k}} \ast D_H\chi_{E}} = \ban{\chi_{\Omega \setminus E} F, \nabla_{H}(\rho_{\eps_{k}} \ast \chi_{E})}
\end{array} \quad \text{in $\;\Omega_{2 \eps_{k}}^{\cR}$.}
\eeq
We point out that the measures at the right hand side in \eqref{eq:chiEF_equivalence} are not defined on the whole $\Omega$, while this is true for those at the left hand side. 
However, arguing as in Remark \ref{remark:weakconvOmegaep}, we can see that the weak$^*$ convergence \eqref{eq:chiEFconvol} is equivalent to the weak$^*$ convergence
\beq\label{eq:chiEF}
\begin{array}{c} 
\ban{\chi_{E} F, \nabla_{H}(\rho_{\eps_{k}} \ast \chi_{E})} \mu \weakto (\chi_{E} F, D_{H} \chi_{E}), \\
\ban{\chi_{\Omega \setminus E} F, \nabla_{H}(\rho_{\eps_{k}} \ast \chi_{E})} \mu \weakto (\chi_{\Omega \setminus E} F, D_{H} \chi_{E}).
\end{array}
\eeq

It is important to stress that at the moment the ``pairing measures'' 
$(\chi_{E} F, D_{H} \chi_{E})$ and $(\chi_{\Omega\sm E} F, D_{H} \chi_{E})$
may depend on the choice of the sequence $\eps_{k}$ and also on the mollifier $\rho$.

We are now in the position to define the {\em interior and exterior normal traces} of $F$ on the boundary of $E$ as the functions $\Tr{F}{i}{E}, \Tr{F}{e}{E} \in L^{\infty}(\Omega; |D_{H} \chi_{E}|)$ satisfying
\begin{align} \label{interior_normal_trace_def} 
- 2 (\chi_{E} F, D_{H} \chi_{E}) & = \Tr{F}{i}{E} |D_{H} \chi_{E}|, \\ \label{exterior_normal_trace_def} 
- 2 (\chi_{\Omega \setminus E} F, D_{H} \chi_{E}) & = \Tr{F}{e}{E} |D_{H} \chi_{E}|. 
\end{align}
Since $\rho_{\eps_{k}} \ast \chi_{E}$ is uniformly bounded,
up to extracting a subsequence, we may also assume that 
\beq\label{eq:limxF}
\rho_{\eps_{k}} \ast \chi_{E} \weakstarto \tildef{\chi_{E}} \quad\text{in} \quad
L^{\infty}(\Omega; |\div F|).
\eeq
This allows us to define the sets
\beq\label{int-ext_F_tilde_chi_E} 
\intE{E} := \{x\in\Omega: \tildef{\chi_{E}}(x) = 1 \} \qandq
\extE{E}:= \{x\in\Omega: \tildef{\chi_{E}}(x) = 0\}
\eeq 
to be the {\em measure theoretic interior} and the {\em measure theoretic exterior
of $E$}, respectively, {\em with respect to $F$ and $\tildef{\chi_{E}}$}.
We may also define an associated reduced boundary
\beq\label{eq:redBB}
\redB{E}=\redb E\sm\pa{\intE{E}\cup\extE{E}}.
\eeq
We wish to underline again the fact that these notions heavily depend on 
$\tildef{\chi_{E}}$, which is not unique, a priori, since it 
depends on the choice of the sequence $\rho_{\eps_{k}} \ast \chi_{E}$.

In the sequel, we will refer to the above sequence $\ep_k$, or possible subsequences, such that
\eqref{eq:chiEFconvol} and \eqref{eq:limxF} hold.
Notice that despite this dependence we will provide conditions under which the
limit measures of \eqref{eq:chiEFconvol} and the sets of \eqref{int-ext_F_tilde_chi_E} and \eqref{eq:redBB} prove to have an intrinsic geometric meaning.

\begin{remark} \label{normal_traces_opposite_orientation} 
By \eqref{eq:chiEFconvol}, observing that $\rho_{\eps_{k}} \ast D_{H} \chi_{E} = - \rho_{\eps_{k}} \ast D_{H} \chi_{\Omega \setminus E}$, we also get
the following equalities 
\begin{align*} (\chi_{E} F, D_{H} \chi_{E}) & = - (\chi_{E} F, D_{H} \chi_{\Omega \setminus E}), \\
(\chi_{\Omega \setminus E} F, D_{H} \chi_{E}) & = - (\chi_{\Omega \setminus E} F, D_{H} \chi_{\Omega \setminus E}).
\end{align*}
We conclude that the normal traces of $F$ on $E$ and $\Omega \setminus E$ satisfy the following relations
\[
\Tr{F}{i}{E} = - \Tr{F}{e}{\Omega \setminus E}\quad\text{and}\quad
\Tr{F}{e}{E} = - \Tr{F}{i}{\Omega \setminus E}.
\]
\end{remark}

We employ now the Leibniz rule (Theorem \ref{productrule}) and \eqref{eq:chiEFconvol} to achieve the following result, which is a key step in order to prove the Gauss--Green formulas. 

\begin{proposition} \label{boundary_term_divergence}
Let $F \in \DM^{\infty}(H \Omega)$ and $E \subset \Omega$ be a set of finite h-perimeter, then the following formulas hold
\begin{align}
\label{Leibniz_infty_E_1} \div(\chi_{E} F) & = \tildef{\chi_{E}} \div F + (F, D_{H} \chi_{E}), \\
\label{Leibniz_infty_E_2} \div(\chi_{E} F) & = (\tildef{\chi_{E}})^{2} \div F + \frac{1}{2} (F, D_{H} \chi_{E}) + (\chi_{E} F, D_{H} \chi_{E}), \\
\label{boundary term div} 
\tildef{\chi_{E}} ( 1 - \tildef{\chi_{E}}) \div F & = \frac{1}{2} (\chi_{E} F, D_{H} \chi_{E}) - \frac{1}{2} (\chi_{\Omega \setminus E} F, D_{H} \chi_{E})
\end{align}
in the sense of Radon measures on $\Omega$, where $\tildef{\chi_{E}} \in L^{\infty}(\Omega; |\div F|)$ is defined in \eqref{eq:limxF}.
\end{proposition}
\begin{proof}
By Theorem~\ref{productrule} applied to $F$ and $g=\chi_{E}$, up to extracting a subsequence, we can assume that the choice of $\ep_k$ leads to \eqref{Leibniz_infty_E_1} in analogy with Theorem~\ref{productrule}.
We observe that   
$$
\div ( (\rho_{\eps_k} \ast \chi_{E}) \chi_{E} F) \weakto \div( \chi_{E}^{2} F) = \div ( \chi_E F),
$$
as measures, since $\chi_{E}^{2} = \chi_{E}$. By \eqref{productruleLip}, \eqref{Leibniz_infty_E_1} and \eqref{commutation conv derivative meas}, we get the following identities of measures on $\Omega_{2 \eps_{k}}^{\cR}$:
\begin{align} \label{chiE2_approx} 
\div(F \chi_{E} (\rho_{\eps_k} \ast \chi_{E}))  = &\, (\rho_{\eps_k} \ast \chi_{E}) \div (\chi_{E} F) + \ban{\chi_{E} F, \nabla_{H} (\rho_{\eps_k} \ast \chi_{E})} \mu \\
 =&\, (\rho_{\eps_k} \ast \chi_{E}) \tildef{\chi_{E}} \div F + (\rho_{\eps_k} \ast \chi_{E}) (F, D_{H} \chi_{E}) \\
&+ \ban{ \chi_{E} F, \rho_{\eps_k} \ast D_{H} \chi_{E}} \mu. \nonumber \end{align}
Recall that our subsequence $\ep_k$ is chosen such that both 
\eqref{eq:limxF} and \eqref{eq:chiEFconvol} hold. 
In view of \eqref{limit1/2}, we have $\displaystyle (\rho_{\eps_{k}} \ast \chi_{E}) \weakstarto \frac{1}{2} \in L^{\infty}(\Omega; |D_{H} \chi_{E}|)$. By \eqref{abs_cont_pairing_infty} we get
\[
|(F, D_{H} \chi_{E})| \le \|F\|_{\infty, \Omega} |D_{H} \chi_{E}|
\]
and we observe that the definition of $(F, D_{H} \chi_{E})$ from Theorem~\ref{productrule} fits
with the definitions \eqref{eq:chiEFconvol}, thanks to \eqref{eq:chiEF}, getting the obvious identity
\begin{equation} \label{decomposition E and compl} 
( F, D_{H} \chi_{E}) = (\chi_{E} F, D_{H} \chi_{E}) + (\chi_{\Omega \setminus E} F, D_{H} \chi_{E}). 
\end{equation}
Remark~\ref{weak_conv_absolutely_continuous_perimeter_measures} shows that $$
(\rho_{\eps_k} \ast \chi_{E}) (F, D_{H} \chi_{E}) \weakto \frac{1}{2} (F, D_{H} \chi_{E}).
$$ 
All in all, by passing to the weak$^{*}$ limits in \eqref{chiE2_approx}, we get \eqref{Leibniz_infty_E_2}.
Subtracting \eqref{Leibniz_infty_E_2} from \eqref{Leibniz_infty_E_1} we have
\begin{equation} \label{boundary step}
\tildef{\chi_{E}} ( 1 - \tildef{\chi_{E}}) \div F = (\chi_{E} F, D_{H} \chi_{E}) - \frac{1}{2}( F, D_{H} \chi_{E}). 
\end{equation}
From \eqref{boundary step} and \eqref{decomposition E and compl} we get \eqref{boundary term div}.
\end{proof}
\begin{remark}
In the assumptions of Proposition~\ref{boundary_term_divergence},
joining \eqref{boundary term div}, \eqref{interior_normal_trace_def} and \eqref{exterior_normal_trace_def}, we get the following equality
\begin{equation} \label{boundary_term_div_trace} 
\tildef{\chi_{E}} ( 1 - \tildef{\chi_{E}}) \div F = \frac{ \Tr{F}{e}{E} - \Tr{F}{i}{E}}{4} |D_{H} \chi_{E}|.
\end{equation} 
\end{remark}
We now prove sharp estimates on the $L^{\infty}$-norm of the normal traces. Let us point out that such estimates could not be obtained directly from \eqref{eq:chiEF_ie}, employing \eqref{interior_normal_trace_def} and \eqref{exterior_normal_trace_def}.
A more refined argument is necessary, involving the differentiation with
respect to the h-perimeter measure.

\begin{proposition}\label{boundary_term_divergence1}
If $F \in \DM^{\infty}(H \Omega)$ and $E \subset \Omega$ is a set of finite h-perimeter, then
\begin{align} \label{interior_normal_trace_norm}
 \|\Tr{F}{i}{E}\|_{L^{\infty}(\redb E; |D_{H} \chi_{E}|)} & \le \| F\|_{\infty, E}, \\
 \label{exterior_normal_trace_norm} \|\Tr{F}{e}{E}\|_{L^{\infty}(\redb E; |D_{H} \chi_{E}|)} & \le \| F\|_{\infty, \Omega \setminus E},
\end{align}
where the interior and exterior normal traces of $F$ are defined in \eqref{interior_normal_trace_def} and \eqref{exterior_normal_trace_def}.
\end{proposition}
\begin{proof}
By Theorem \ref{perimeter repr} the perimeter measure $|D_H\chi_{E}|(\cdot)$ is a.e.\ asymptotically doubling. Therefore the following differentiation property holds
(see \cite[Sections 2.8.17 and 2.9.6]{Fe}): for $D_{H} \chi_{E}$-a.e. $x \in \redb E$ one has
$$ \Tr{F}{i}{E}(x) = \lim_{r \to 0} - \frac{2 (\chi_{E} F, D_{H} \chi_{E})(B(x, r))}{|D_{H} \chi_{E}|(B(x, r)}. $$
\begin{comment}
It is clearly not restrictive to establish our estimates on an arbitrarily fixed
open subset $\Omega'\Subset\Omega$, hence considering $x\in\redb E\cap\Omega'$. 
\end{comment}
Let $\ep_k$ be the sequence defining \eqref{eq:chiEFconvol} and \eqref{eq:limxF}.
By \eqref{eq:conv_unif_bd}, we obtain that the sequence $|\ban{\chi_{E} F, \rho_{\eps_{k}} \ast D_{H} \chi_{E}}| \mu$ is uniformly bounded in $\mathcal{M}(\Omega)$.
%
\begin{comment}
Indeed, we have
\[
\int_{\Omega'} |\ban{\chi_{E} F, \nabla_{H} (\rho_{\eps_{k}} \ast \chi_{E})}| \, dx 
\le \|F\|_{\infty, E}\; \| \nabla_{H} (\rho_{\eps_{k}} \ast \chi_{E})\|_{1, \Omega'} \le \|F\|_{\infty, E}\; |D_{H} \chi_{E}|(\Omega), 
\]
where the last inequality follows from  \eqref{total_var_convol_convergence}.
\end{comment}
Thus, there exists a weak$^{*}$ converging subsequence, which we do not relabel.
Let the positive measure $\lambda_{i} \in \mathcal{M}(\Omega)$ be its limit. 

In an analogous way, one can prove that the measures $|\ban{\chi_{\Omega \setminus E} F, \rho_{\eps_{k}} \ast D_{H} \chi_{E}}| \mu$ are uniformly bounded in $\cM(\Omega)$.
So there exists a weak$^{*}$ converging subsequence, which we do not relabel again, 
and whose limit is the positive Radon measure $\lambda_{e}\in\cM(\Omega)$. 
We also observe that the sequences $\chi_{E} |\rho_{\eps_{k}} \ast D_{H} \chi_{E}| \mu$ and $\chi_{\Omega \setminus E} |\rho_{\eps_{k}} \ast D_{H} \chi_{E}| \mu$ are bounded 
in $\cM(\Omega)$ and that, if $\gamma\in\cM(\Omega)$ is any of their weak$^{*}$ limit points, 
then $\gamma \le |D_{H} \chi_{E}|$, due to \eqref{weak_conv_D_f}.

We can choose a sequence of balls $B(x, r_{j})$ with $r_{j} \to 0$ in such a way that 
\[
|D_{H} \chi_{E}|(\partial B(x, r_{j})) = \lambda_{i}(\partial B(x, r_{j})) 
=\lambda_e(\partial B(x, r_{j}))=0
\]
for all $j$. As a result, taking into account \cite[Proposition~1.62]{AFP}
and \eqref{eq:chiEFconvol}, we have
\[
\begin{split}
\left | \frac{2 ( \chi_{E} F, D_{H} \chi_{E})(B(x, r_{j}))}{|D_{H} \chi_{E}|(B(x, r_{j}))} \right | &= \left | \frac{  \displaystyle \lim_{\eps_{k} \to 0} 2 \int_{B(x, r_{j})} \ban{\chi_{E} F, \rho_{\eps_{k}} \ast D_{H} \chi_{E}} \, dx}{ \displaystyle \lim_{\eps_{k} \to 0} \int_{B(x, r_{j})} |\rho_{\eps_{k}} \ast D_{H} \chi_{E}| \, dx} \right | \\
& \le 2 \|F\|_{\infty, E}  \frac{ \displaystyle \lim_{\eps_{k} \to 0} \int_{B(x, r_{j})} \chi_{E} | \rho_{\eps_{k}} \ast D_{H} \chi_{E} | \, dx}{ \displaystyle \lim_{\eps_{k} \to 0} \int_{B(x, r_{j})} |\rho_{\eps_{k}} \ast D_{H} \chi_{E} | \, dx}.
\end{split} 
\]
The last term can be also written as
\[ 
2 \|F\|_{\infty, E}  \left ( 1 - \frac{ \displaystyle \lim_{\eps_{k} \to 0} \int_{B(x, r_{j})} \chi_{\Omega \setminus E} | \rho_{\eps_{k}} \ast D_{H} \chi_{E} | \, dx }{ \displaystyle \lim_{\eps_{k} \to 0} \int_{B(x, r_{j})} |\rho_{\eps_{k}} \ast D_{H} \chi_{E}| \, dx} \right ). 
\]
It follows that  
\begin{align*}
\left | \frac{2 ( \chi_{E} F, D_{H} \chi_{E})(B(x, r_{j}))}{|D_{H} \chi_{E}|(B(x, r_{j}))} \right | 
&\le 2 \|F\|_{\infty, E}  \left ( 1 - \frac{ \displaystyle \lim_{\eps_{k} \to 0} \pal{ \int_{B(x, r_{j})} \chi_{\Omega \setminus E} (\rho_{\eps_{k}} \ast D_{H}\chi_{E}) \, dx } }{ \displaystyle \lim_{\eps_{k} \to 0}  \int_{B(x, r_{j})} |\rho_{\eps_{k}} \ast D_{H} \chi_{E})| \, dx } \right ) \\
& = 2 \|F\|_{\infty, E} \left ( 1 - \frac{1}{2} \frac{ | D_{H} \chi_{E}(B(x, r_{j}))| }{|D_{H} \chi_{E}|(B(x, r_{j}))} \right ). 
\end{align*}
by \eqref{overline_D_chi_E_2} and the second limit of \eqref{weak_conv_D_f}.
Taking the limit as $j\to\infty$, the definition of reduced boundary immediately yields
\[
\pal{\Tr{F}{i}{E}(x)}=\lim_{k\to\infty}\left | \frac{2 ( \chi_{E} F, D_{H} \chi_{E})(B(x, r_{j}))}{|D_{H} \chi_{E}|(B(x, r_{j}))} \right | 
\le \|F\|_{\infty, E}.
\]
The estimate for the exterior normal trace $\Tr{F}{e}{E}$ can be obtained in a similar way,
hence the proof is complete.
\end{proof}

\subsection{Locality of normal traces}

In this section we show the locality of normal traces, along with their relation with the orientation of the reduced boundary.
First, we need to recall some known facts on the locality properties of perimeter in stratified groups.
By Theorem \ref{perimeter repr} (see also \cite[Theorem 4.2]{Ambrosio2001}) and Lemma \ref{difference redb mtb}, for any set $E$ of finite h-perimeter in $\Omega$, there exists a Borel function $\theta_{E}$, such that $\theta_{E} \ge \alpha > 0$ and 
\begin{equation} \label{repr_D_H_chi_E} |D_{H} \chi_{E}|(B) = \int_{B \cap \redb E} \theta_{E} \, d \SHaus{Q - 1}, \end{equation}
which implies $\theta \in L^{1}(\Omega; \SHaus{Q - 1} \res \redb E)$. By this representation, a property holds $|D_{H} \chi_{E}|$-a.e. if and only if it holds $\SHaus{Q - 1}$-a.e. on $\redb E$, see also Remark \ref{negl_abs_cont_red_boundary}.

Given two sets $E_{1}, E_{2}$ of finite h-perimeter such that $\SHaus{Q - 1}(\redb E_{1} \cap \redb E_{2}) > 0$, by \cite[Theorem 2.9]{ambrosio2010locality}, for any Borel set $B \subset \redb E_{1} \cap \redb E_{2}$ we have
\begin{equation*} |D_{H} \chi_{E_{1}}|(B) = |D_{H} \chi_{E_{2}}|(B). \end{equation*}
Hence, \eqref{repr_D_H_chi_E} implies that 
\begin{equation} \label{loc_property_intersection} \theta_{E_{1}}(x) = \theta_{E_{2}}(x) \ \ \text{for} \ \ \SHaus{Q - 1}\text{-a.e.} \ \ x \in \redb E_{1} \cap \redb E_{2}. \end{equation}
Moreover, \cite[Corollary 2.6]{ambrosio2010locality} implies that, for $\SHaus{Q - 1}$-a.e. $x \in \redb E_{1} \cap \redb E_{2}$, we have $\nu_{E_{1}}(x) = \pm \nu_{E_{2}}(x)$.

\begin{lemma} \label{locality_per_orientation} 
If $E_{1}$ and $E_{2}$ have finite h-perimeter in $\Omega$ with $\SHaus{Q - 1}(\redb E_{1} \cap \redb E_{2}) > 0$, then we have
\begin{equation} \label{local_prop_h_per_4} |D_{H} ( \chi_{E_{1}} - \chi_{E_{2}})|(B(x, r)) = o(|D_{H} \chi_{E_{j}}|(B(x, r))) \end{equation} 
for $\SHaus{Q - 1}$-a.e.\ $x \in \redb E_{1} \cap \redb E_{2}$ such that $\nu_{E_{1}}(x) = \nu_{E_{2}}(x)$, and for $j = 1, 2$. 
Analogously, we have 
\begin{equation} \label{local_prop_h_per_5} |D_{H} ( \chi_{E_{1}} + \chi_{E_{2}})|(B(x, r)) = o(|D_{H} \chi_{E_{j}}|(B(x, r))) \end{equation} 
for $\SHaus{Q - 1}$-a.e.\ $x \in \redb E_{1} \cap \redb E_{2}$ such that $\nu_{E_{1}}(x) = - \nu_{E_{2}}(x)$, and for $j = 1, 2$. 
In addition, we have 
\begin{equation} \label{local_asimptotic_prop_h_per} |D_{H} \chi_{E_{1}}|(B(x, r)) \sim |D_{H} \chi_{E_{2}}|(B(x, r)), \end{equation}
for $\SHaus{Q - 1}$-a.e.\ $x \in \redb E_{1} \cap \redb E_{2}$ and $j = 1, 2$.
\end{lemma}
\begin{proof}
We first define the following sets
\[
L := \redb E_{1} \cap \redb E_{2} \qandq G := \redb E_{1} \Delta \redb E_{2}.
\] 
Then, by \eqref{repr_D_H_chi_E} and \eqref{loc_property_intersection}, we obtain
\begin{equation} \label{local_prop_h_per_1} |D_{H} \chi_{E_1}| \res L = |D_{H} \chi_{E_{2}}| \res L, \end{equation}
\begin{equation} \label{local_prop_h_per_2} \left | |D_{H} \chi_{E_{1}}| - |D_{H} \chi_{E_{2}}| \right | = \theta \SHaus{Q - 1} \res G, \end{equation}
where $\theta = \theta_{E_{j}}(x) $ for $\SHaus{Q - 1}$-a.e.\ $x \in \redb E_{j}$ and for $j = 1, 2$.
Hence, since $L \cap G = \emptyset$, for $\SHaus{Q - 1}$-a.e.\ $x \in L$, we have
\begin{equation} \label{local_prop_h_per_3} \left | |D_{H} \chi_{E_{1}}| - |D_{H} \chi_{E_{2}}| \right | (B(x, r)) = \int_{G \cap B(x, r)} \theta \, d \SHaus{Q - 1} = o(r^{Q - 1}), \end{equation}
by \eqref{local_prop_h_per_2} and standard differentiation of Borel measures.
In addition, 
\begin{equation} \label{decay_D_H_chi_E_r_Q_1} |D_{H} \chi_{E_{j}}|(B(x, r)) \ge c r^{Q - 1} \end{equation} 
for $\SHaus{Q - 1}$-a.e. $x \in \redb E_{j}$, $r > 0$ sufficiently small and $j = 1, 2$, by \cite[Theorem 4.3]{Ambrosio2001}.
Then, \eqref{local_prop_h_per_3} and the triangle inequality imply that, for $j = 1, 2$,
\begin{equation*} | |D_{H} \chi_{E_{1}}|(B(x, r)) - |D_{H} \chi_{E_{2}}|(B(x, r)) | = o(|D_{H} \chi_{E_{j}}|(B(x, r))), \end{equation*}
from which we get \eqref{local_asimptotic_prop_h_per}.
Then, we notice that, for $\SHaus{Q - 1}$-a.e.\ $x \in L$ such that $\nu_{E_{1}}(x) = \nu_{E_{2}}(x)$, and $j = 1, 2$, we have 
\begin{align*} 
 |D_{H} ( \chi_{E_{1}} - \chi_{E_{2}})|(B(x, r)) & = \Big| (\nu_{E_{1}} - \nu_{E_{2}}) |D_{H} \chi_{E_{1}}| \res L + \nu_{E_{1}} |D_{H} \chi_{E_{1}}| \res G + \\
 & - \nu_{E_{2}} |D_{H} \chi_{E_{2}}| \res G \Big|(B(x, r)) \\
& \le \int_{B(x, r)} |\nu_{E_{1}} - \nu_{E_{2}}| \, d |D_{H} \chi_{E_{1}}| \res L \ + \\
& + |D_{H} \chi_{E_{1}}|(G \cap B(x, r)) + |D_{H} \chi_{E_{2}}|(G \cap B(x, r)) \\
& \le \int_{B(x, r)} |\nu_{E_{1}} - \nu_{E_{1}}(x)| \, d |D_{H} \chi_{E_{1}}| + \\
& + \int_{B(x, r)} |\nu_{E_{2}} - \nu_{E_{2}}(x)| \, d |D_{H} \chi_{E_{2}}| + o(r^{Q - 1}) \\
& = o(|D_{H} \chi_{E_{j}}|(B(x, r))),
\end{align*} 
by \eqref{local_prop_h_per_1}, \eqref{local_asimptotic_prop_h_per}, \eqref{decay_D_H_chi_E_r_Q_1}, the triangle inequality and standard differentiation of Borel measures.
Thus, we can conclude that \eqref{local_prop_h_per_4} holds. Analogously, \eqref{local_prop_h_per_5} follows for $\SHaus{Q - 1}$-a.e. $x\in\redb E_1\cap\redb E_2$  such that $\nu_{E_1}(x)=-\nu_{E_2}(x)$.
\end{proof}
\begin{theorem} \label{normal_trace_locality_theorem} Let $F \in \DM^{\infty}(H \Omega)$, and $E_{1}, E_{2} \subset \Omega$ be sets of finite h-perimeter such that $\SHaus{Q - 1}(\redb E_{1} \cap \redb E_{2}) > 0$. Then, we have
\begin{equation} \label{same_or_normal_traces} \Tr{F}{i}{E_{1}}(x) = \Tr{F}{i}{E_{2}}(x) \ \ \text{and} \ \ \Tr{F}{e}{E_{1}}(x) = \Tr{F}{e}{E_{2}}(x), \end{equation}
for $\SHaus{Q - 1}$-a.e. $x \in \{ y \in \redb E_{1} \cap \redb E_{2} : \nu_{E_{1}}(y) = \nu_{E_{2}}(y) \}$, and
\begin{equation} \label{opposite_or_normal_traces} \Tr{F}{i}{E_{1}}(x) = - \Tr{F}{e}{E_{2}}(x) \ \ \text{and} \ \ \Tr{F}{e}{E_{1}}(x) = - \Tr{F}{i}{E_{2}}(x), \end{equation}
for $\SHaus{Q - 1}$-a.e. $x \in \{ y \in \redb E_{1} \cap \redb E_{2} : \nu_{E_{1}}(y) = - \nu_{E_{2}}(y) \}$.
\end{theorem}
\begin{proof}
We recall that, by Theorem \ref{perimeter repr}, the perimeter measure $|D_H\chi_{E_{j}}|(\cdot)$ is a.e.\ asymptotically doubling, for $j = 1, 2$. Therefore, by the definitions \eqref{interior_normal_trace_def} and \eqref{exterior_normal_trace_def}, and the differentiation of perimeters (see \cite[Sections 2.8.17 and 2.9.6]{Fe}), we have
\begin{align*} 
\Tr{F}{i}{E_{j}}(x) & = - 2 \lim_{r \to 0} \frac{(\chi_{E_{j}} F, D_{H} \chi_{E_{j}})(B(x, r))}{|D_{H} \chi_{E_{j}}|(B(x, r))}, \\
\Tr{F}{e}{E_{j}}(x) & = - 2 \lim_{r \to 0} \frac{(\chi_{\Omega \setminus E_{j}} F, D_{H} \chi_{E_{j}})(B(x, r))}{|D_{H} \chi_{E_{j}}|(B(x, r))}, \end{align*}
for $j = 1, 2$, and for $\SHaus{Q - 1}$-a.e.\ $x \in \redb E_{j}$.
Let $x \in \redb E_{1} \cap \redb E_{2}$ be such that $\nu_{E_{1}}(x) = \nu_{E_{2}}(x)$ and \eqref{local_prop_h_per_4} and \eqref{local_asimptotic_prop_h_per} hold true. 
Taking into account that
$ | \Tr{F}{i}{E_{1}}(x) - \Tr{F}{i}{E_{2}}(x)|$
can be written as the limit of the difference
\[
2 \lim_{r \to 0} \left  | \frac{(\chi_{E_{1}} F, D_{H} \chi_{E_{1}})(B(x, r))}{|D_{H} \chi_{E_{1}}|(B(x, r))} - \frac{(\chi_{E_{2}} F, D_{H} \chi_{E_{2}})(B(x, r))}{|D_{H} \chi_{E_{2}}|(B(x, r))} \right | 
\]
using the linearity and the triangle inequality, we get
\begin{align*} | \Tr{F}{i}{E_{1}}(x) -& \Tr{F}{i}{E_{2}}(x)|  \le 2 \limsup_{r \to 0} \left | \frac{(\chi_{E_{1}} F, D_{H} (\chi_{E_{1}} - \chi_{E_{2}}))(B(x, r))}{|D_{H} \chi_{E_{1}}|(B(x, r))} \right |  \\
& + \limsup_{r\to0}\left | \frac{(\chi_{E_{2}} F, D_{H} (\chi_{E_{1}} - \chi_{E_{2}}))(B(x, r))}{|D_{H} \chi_{E_{2}}|(B(x, r))} \right |   \\
& + \limsup_{r\to0}\left | \frac{(\chi_{E_{1}} F, D_{H}\chi_{E_{2}})(B(x, r))}{|D_{H} \chi_{E_{1}}|(B(x, r))} - \frac{(\chi_{E_{2}} F, D_{H}\chi_{E_{1}})(B(x, r))}{|D_{H} \chi_{E_{2}}|(B(x, r))} \right |. \end{align*} 
By \eqref{abs_cont_pairing_infty}, we have $|(\chi_{E_{j}} F, D_{H} (\chi_{E_{1}} - \chi_{E_{2}}))| \le \|F\|_{\infty, E_{j}} |D_{H} (\chi_{E_{1}} - \chi_{E_{2}})|$, for $j = 1, 2$, and so, by \eqref{local_prop_h_per_4}, we conclude that 
\begin{equation*}  \left | \frac{(\chi_{E_{j}} F, D_{H} (\chi_{E_{1}} - \chi_{E_{2}}))(B(x, r))}{|D_{H} \chi_{E_{j}}|(B(x, r))} \right | \to 0, \end{equation*}
for $j = 1, 2$. Now we have to deal with the last term, which, by \eqref{local_asimptotic_prop_h_per}, is infinitesimal as $r\to0$
if and only if so is 
\begin{equation}\label{eq:quotEstE_1E_2} 
\left | \frac{(\chi_{E_{1}} F, D_{H}\chi_{E_{2}})(B(x, r)) - (\chi_{E_{2}} F, D_{H}\chi_{E_{1}})(B(x, r))}{|D_{H} \chi_{E_{2}}|(B(x, r))} \right |. 
\end{equation}
By Theorem \ref{productrule}, we know that $\chi_{E_{j}} F \in \DM^{\infty}(H \Omega)$, for $j = 1, 2$. 
Let $\ep_k$ be the defining sequence for 
\[
\Tr{F}{e}{E_{j}},\quad \Tr{F}{i}{E_{j}},\quad (\chi_{E_{j}} F, D_{H} \chi_{E_{j}})
\quad\text{and}\quad  (\chi_{\Omega\setminus E_{j}} F, D_{H} \chi_{E_{j}})
\]
through limits analogous to those of \eqref{eq:chiEFconvol} and \eqref{eq:limxF} for $E_{j}$, $j=1,2$, in place of $E$.
It follows that 
\[
\rho_{\eps_{k}} \ast \chi_{E_{1}} \weakstarto \tildef{\chi_{E_{1}}}\quad \text{ in $L^{\infty}(\Omega; |\div F|)$}
\]
and, by \eqref{productruleLip}, \eqref{Leibniz_infty_E_1} and \eqref{commutation conv derivative meas}, in $\Omega_{2 \eps}^{\cR}$ we have
\beq\label{eq:divchiE_1E_2}
\begin{split}
\div( (\rho_{\eps_{k}} \ast \chi_{E_{1}}) \chi_{E_{2}} F) & = \rho_{\eps_{k}} \ast \chi_{E_{1}} \div(\chi_{E_{2}} F) + \ban{\chi_{E_{2}} F, \nabla_{H} (\rho_{\eps_{k}} \ast \chi_{E_{1}}) } \mu \\
& = (\rho_{\eps_{k}} \ast \chi_{E_{1}}) \tildef{\chi_{E_{2}}} \div(F) + (\rho_{\eps_{k}} \ast \chi_{E_{1}}) (F, D_{H} \chi_{E_{2}})   \\
& + \ban{\chi_{E_{2}} F, \rho_{\eps_{k}} \ast D_{H} \chi_{E_{1}} } \mu. 
\end{split}
\eeq 
Since $|\rho_{\eps_{k}} \ast \chi_{E_{1}}|(x) \le 1$ for any $x \in \Omega$, up to extracting a further subsequence, we may find $\overline{\chi_{E_{1}}} \in L^{\infty}(\Omega; |(F, D_{H} \chi_{E_{2}})|)$ such that 
\[
\rho_{\eps_{k}} \ast \chi_{E_{1}} \weakstarto \overline{\chi_{E_{1}}} \quad\text{ in $L^{\infty}(\Omega; |(F, D_{H} \chi_{E_{2}})|)$}.
\]
Thus, by Lemma~\ref{lemma:weak_conv_absolutely_continuous_perimeter_measures}, we get
the weak$^{*}$ convergence
\[
(\rho_{\eps_{k}} \ast \chi_{E_{1}}) (F, D_{H} \chi_{E_{2}})\weakto \overline{\chi_{E_{1}}} (F, D_{H} \chi_{E_{2}}).
\]
Moreover, by Lemma~\ref{conv_meas_bounded_function}, the sequence $\ban{\chi_{E_{2}} F, \rho_{\eps_{k}} \ast D_{H} \chi_{E_{1}}) } \mu$ is uniformly bounded on $\Omega$, hence, up to extracting further subsequences, there exists the weak$^{*}$ limit 
\[
\ban{\chi_{E_{2}} F, \rho_{\eps_{k}} \ast D_{H} \chi_{E_{1}} } \mu
\rightharpoonup (\chi_{E_{2}} F, D_{H} \chi_{E_{1}}).
\]
By Remark~\ref{r:sign_chi_E} we know that $|\div F|$-a.e.\
there holds $0 \le \tildef{\chi_{E_{2}}} \le 1$ and by Lemma~\ref{lemma:weak_conv_absolutely_continuous_perimeter_measures}, we
conclude that 
\[
(\rho_{\eps_{k}} \ast \chi_{E_{1}}) \tildef{\chi_{E_{2}}} \div(F) \weakto \tildef{\chi_{E_{1}}} \tildef{\chi_{E_{2}}} \div F.
\]
Passing now to the limit in \eqref{eq:divchiE_1E_2} as $\eps_{k} \to 0$, there holds
\begin{equation} \label{double_set_F_1} \div(\chi_{E_{1}} \chi_{E_{2}} F) = \tildef{\chi_{E_{1}}} \tildef{\chi_{E_{2}}} \div F + \overline{\chi_{E_{1}}} (F, D_{H} \chi_{E_{2}}) + (\chi_{E_{2}} F, D_{H} \chi_{E_{1}}). \end{equation}
Arguing in an analogous way, exchanging the role of $\chi_{E_{1}}$ and $\chi_{E_{2}}$, we get
\begin{equation} \label{double_set_F_2} \div(\chi_{E_{1}} \chi_{E_{2}} F) = \tildef{\chi_{E_{1}}} \tildef{\chi_{E_{2}}} \div F + \overline{\chi_{E_{2}}} (F, D_{H} \chi_{E_{1}}) + (\chi_{E_{1}} F, D_{H} \chi_{E_{2}}). \end{equation}
Then \eqref{double_set_F_1} and \eqref{double_set_F_2} yield
\begin{equation} \label{third_term_trace_ineq}  (\chi_{E_{2}} F, D_{H} \chi_{E_{1}}) - (\chi_{E_{1}} F, D_{H} \chi_{E_{2}}) = \overline{\chi_{E_{2}}} (F, D_{H} \chi_{E_{1}}) - \overline{\chi_{E_{1}}} (F, D_{H} \chi_{E_{2}}).\end{equation}
Joining \eqref{limit1/2}, Lemma~\ref{lemma:weak_conv_absolutely_continuous_perimeter_measures} 
and \eqref{local_prop_h_per_1}, we can conclude that 
\[
\overline{\chi_{E_{1}}}(x) = \overline{\chi_{E_{2}}}(x) = 1/2 \quad\text{ for $\SHaus{Q - 1}-a.e.\ x \in \redb E_{1} \cap \redb E_{2} =: L$}.
\]
By \eqref{abs_cont_pairing_infty}, we notice that
\beq\label{eq:FD_HchiE_1}
|(F, D_{H} \chi_{E_{j}})| \res L \le \|F\|_{\infty, \Omega} |D_{H} \chi_{E_{j}}| \res L, \ \text{for} \ \ j = 1, 2; \eeq
and so, by Remark \ref{negl_abs_cont_red_boundary}, we obtain
\begin{equation*} \overline{\chi_{E_{2}}} (F, D_{H} \chi_{E_{1}}) \res L = \frac{1}{2} (F, D_{H} \chi_{E_{1}}) \res L, \ \  \overline{\chi_{E_{1}}} (F, D_{H} \chi_{E_{2}}) \res L = \frac{1}{2} (F, D_{H} \chi_{E_{2}}) \res L. \end{equation*}
Now, if we set $G := \redb E_{1} \Delta \redb E_{2}$, we observe that we can rewrite \eqref{third_term_trace_ineq} as
\begin{align} \label{third_term_trace_ineq_1} 
(\chi_{E_{2}} F, D_{H} \chi_{E_{1}}) - (\chi_{E_{1}} F, D_{H} \chi_{E_{2}}) & = \overline{\chi_{E_{2}}} (F, D_{H} \chi_{E_{1}}) \res G - \overline{\chi_{E_{1}}} (F, D_{H} \chi_{E_{2}}) \res G \\
& + \frac{1}{2} (F, D_{H} (\chi_{E_{1}} - \chi_{E_{2}})) \res L. \nonumber \end{align}
By \eqref{abs_cont_pairing_infty} and by standard differentiation of Borel measures, we have
\[
|(F, D_{H} \chi_{E_{j}}) \res G| (B(x,r)) \le \|F\|_{\infty, \Omega} |D_{H} \chi_{E_{j}}|(G \cap B(x, r)) = o(r^{Q-1})
\]
for $\SHaus{Q - 1}$-a.e. $x \in L$, since $G \cap L = \emptyset$, and $j = 1, 2$. In addition, $|D_{H} \chi_{E_{j}}|(B(x, r)) \ge c r^{Q - 1}$ for $\SHaus{Q - 1}$-a.e. $x \in \redb E_{j}$, $r > 0$ sufficiently small and $j = 1, 2$, by \cite[Theorem 4.3]{Ambrosio2001}; and so we obtain
\begin{equation*} |(F, D_{H} \chi_{E_{j}}) \res G|(B(x,r)) = o(|D_{H} \chi_{E_{2}}|(B(x, r))). \end{equation*}
As for the second term, by \eqref{abs_cont_pairing_infty} and \eqref{local_prop_h_per_4} we get
\begin{align*} 
|(F, D_{H} (\chi_{E_{1}} - \chi_{E_{2}}))|(L \cap B(x, r)) &\le
\|F\|_{\infty, \Omega} |D_{H} (\chi_{E_{1}} - \chi_{E_{2}})|(B(x, r))\\
& = o(|D_{H} \chi_{E_{2}}|(B(x, r))) 
\end{align*}
for $\SHaus{Q - 1}$-a.e.\ $x \in \redb E_{1} \cap \redb E_{2}$ such that $\nu_{E_{1}}(x) = \nu_{E_{2}}(x)$.
This implies that the expression in \eqref{eq:quotEstE_1E_2} goes to zero as $r \to 0$, and so it proves the first part of \eqref{same_or_normal_traces}.
Concerning the exterior normal traces, one can argue in a similar way for the sets $\Omega \setminus E_{1}$ and $\Omega \setminus E_{2}$.
Finally, taking into account \eqref{interior_normal_trace_def}, \eqref{exterior_normal_trace_def}, Remark~\ref{normal_traces_opposite_orientation} and \eqref{same_or_normal_traces} applied to $E_{1}$ and $\Omega \setminus E_{2}$, and conversely,
we arrive at \eqref{opposite_or_normal_traces}.
\end{proof}

\subsection{Tripartition by weak$^*$ limit of mollified functions}

In this section we study the properties of the limit $\tildef{\chi_E}$ 
defined in \eqref{eq:limxF} and the related Leibniz rule.
From Remark~\ref{r:sign_chi_E} we have that 
$0 \le \tildef{\chi_{E}}(x) \le 1$ for $|\div F|$-a.e.\ $x\in\Omega$.

\begin{proposition} \label{tripartition_div_1} 
Let $F \in \DM^{\infty}(H \Omega)$ and $E \subset \Omega$ be a set of finite h-perimeter.
Let $\tildef{\chi_E}$, $\intE{E}$, $\extE{E}$ and $\redB{E}$ 
be as in \eqref{eq:limxF}, \eqref{int-ext_F_tilde_chi_E} and
\eqref{eq:redBB}, respectively. It holds
\begin{equation} \label{div_F_partition_Omega} 
|\div F|(\Omega \setminus (\redb E \cup \intE{E} \cup \extE{E}))=
|\div F|(\Omega \setminus (\redB{E} \cup \intE{E} \cup \extE{E})) = 0.
\end{equation}
In particular, $\tildef{\chi_{E}}$ is uniquely determined on $\Omega \setminus \redb E$, up to $|\div F|$-negligible sets, and we have 
\begin{equation} \label{tilde_chi_E_div_no_redb_E}
\tildef{\chi_{E}} = \chi_{\intE{E}} \quad |\div F|\text{-a.e. in} \ \Omega \setminus \redb E, \end{equation}
so that
\begin{equation}  \label{eq:restr_no_redb_E_inteE}
\div(\chi_{E} F) \res (\Omega \setminus \redb E) = \div F \res (\intE{E} \setminus \redb E).
\end{equation}
\end{proposition}
\begin{proof}
From \eqref{boundary_term_div_trace} we immediately conclude that
\begin{equation*} \int_{\Omega \setminus \redb E} \tildef{\chi_{E}}(1 - \tildef{\chi_{E}}) \, d |\div F| = 0. \end{equation*}
Therefore definitions \eqref{int-ext_F_tilde_chi_E} and \eqref{eq:redBB} give
\begin{equation*} 
\Omega = \intE{E}\cup \extE{E}\cup \redB{E} \cup Z_{E}^{F},
\end{equation*}
where $Z_E^F=\Omega\sm(\redb E\cup \intE{E}\cup \extE{E})$
is $|\div F|$-negligible. This proves \eqref{div_F_partition_Omega}. Then, if we restrict \eqref{Leibniz_infty_E_1} to $\Omega \setminus \redb E$, we have
\begin{equation} \label{restr_no_redb_E}
\div(\chi_{E} F) \res (\Omega \setminus \redb E) = \tildef{\chi_{E}} \div F \res (\Omega \setminus \redb E),
\end{equation}
which immediately shows that $\tildef{\chi_{E}}$ is uniquely determined on $\Omega \setminus \redb E$, as a function in $L^{\infty}(\Omega; |\div F|)$.
Moreover, by \eqref{div_F_partition_Omega}, we have that $\tildef{\chi_{E}}(x) \in \{0, 1\}$ for $|\div F|$-a.e. $x \in \Omega \setminus \redb E$, and this implies \eqref{tilde_chi_E_div_no_redb_E}. Finally, this immediately shows that \eqref{restr_no_redb_E} is equivalent to \eqref{eq:restr_no_redb_E_inteE}.
\end{proof}

Formula \eqref{tilde_chi_E_div_no_redb_E} will be important to show that indeed the set $\intE{E}$ is uniquely defined up to $|\div F|$-negligible sets.

\begin{remark}\label{r:tripartition_div_2} 
We notice that, by Proposition \ref{tilde_g_characterization}, $\tildef{\chi_{E}}(x) = \chi_{E}^{*, \cR}(x)$ for $|\div F|$-a.e. $x \in C_{E}^{\cR}$. In particular, we obtain
\begin{equation*}
|\div F|\left ( (\intE{E} \Delta E^{1, \cR}) \cap C_{E}^{\cR} \right ) = 0 \ \text{and} \ |\div F|\left ( (\extE{E} \Delta E^{0, \cR} ) \cap C_{E}^{\cR} \right ) = 0.
\end{equation*}
If we now assume that 
\begin{equation} \label{eq:assumption_precise_right_repr_div}
|\div F|(\Omega \setminus C_{E}^{\cR}) = 0,
\end{equation}
then it follows that $\tildef{\chi_{E}}(x) = \chi_{E}^{*, \cR}(x)$ for $|\div F|$-a.e. $x \in \Omega$. In particular, this yields 
\begin{equation*}
|\div F|(\intE{E} \Delta E^{1, \cR}) = 0 \quad \text{and} \quad |\div F|(\extE{E} \Delta E^{0, \cR}) = 0.
\end{equation*}
Thus, we have shown that
\[ \mtbR E \setminus \redb E = \Omega \setminus (E^{1, \cR} \cup E^{0, \cR} \cup \redb E) = (\Omega \setminus (\intE{E}\cup \extE{E}\cup \redb{E})) \cup \tildef{Z}_{E}^{F}, \]
for some $|\div F|$-negligible set $\tildef{Z}_{E}^{F}$. By \eqref{div_F_partition_Omega}, we obtain $|\div F|(\mtbR E \setminus \redb E) = 0$. 
Hence, under the assumption \eqref{eq:assumption_precise_right_repr_div}, we can identify $\intE{E}$ and $\extE{E}$ with $E^{1, \cR}$ and $E^{0, \cR}$, up to $|\div F|$-negligible sets, thus obtaining their uniqueness in this special case. In fact, as we shall see below, the uniqueness holds in general, even if \eqref{eq:assumption_precise_right_repr_div} fails to be true.
\end{remark}

The following proposition proves that any given set of finite h-perimeter $E$ in $\Omega$
yields a tripartition of $\Omega$. More precisely, for $F \in \DM^{\infty}(H \Omega)$ there exists a representative $\tildef{\chi_{E}}$ of $\chi_{E}$ such that $\tildef{\chi_{E}}(x) \in \{1, 0, 1/2 \}$ for $|\div F|$-a.e. $x \in \Omega$.

\begin{proposition} \label{tilde_chi_E}
Let $F \in \DM^{\infty}(H \Omega)$ and let $E \subset \Omega$ be a set of finite h-perimeter. Let $\rho\in C_{c}(B(0, 1))$ be a mollifier satisfying $\rho(x) = \rho(x^{-1})$ and $\int_{B(0, 1)} \rho(y) \, dy = 1$.
If $\tildef{\chi_{E}} \in L^{\infty}(\Omega; |\div F|)$ is defined by
\eqref{eq:limxF}, then  
\begin{equation} \label{tilde_chi_E_div_F}
\displaystyle \tildef{\chi_{E}} = \frac{1}{2} \ \ |\div F|\text{-a.e.\ on} \ \ 
\redB{E}.
\end{equation}
In addition, the normal traces of $F$ on the boundary of $E$ satisfy 
\begin{equation} \label{trace_equality_boundary} 
\Tr{F}{i}{E} = \Tr{F}{e}{E} \ \ |D_{H} \chi_{E}|\text{-a.e.\ on} \ \ \intE{E} \cup \extE{E}
\end{equation} 
and we have
\begin{equation} \label{boundary_term_div_1_2} \chi_{\redB{E}} \div F = (\Tr{F}{e}{E} - \Tr{F}{i}{E}) |D_{H} \chi_{E}|.
\end{equation}
\end{proposition}
\begin{proof}
From \eqref{boundary_term_div_trace} it follows immediately that 
\[
(\Tr{F}{i}{E} - \Tr{F}{e}{E}) |D_{H} \chi_{E}| = 0  \quad\text{on} \ \ \redb E \cap ( \intE{E} \cup \extE{E}),
\]
proving \eqref{trace_equality_boundary}.
Let $\eps_{k}$ be the defining sequence such that \eqref{eq:chiEFconvol} and \eqref{eq:limxF} hold. We have
\begin{equation*} (\rho_{\eps_{k}} \ast \chi_{E})\tildef{\chi_{E}}(1 - \tildef{\chi_{E}}) \div F \weakto (\tildef{\chi_{E}})^{2}(1 - \tildef{\chi_{E}}) \div F,
\end{equation*}
by Lemma~\ref{lemma:weak_conv_absolutely_continuous_perimeter_measures}.
Since the traces $\Tr{F}{i}{E}, \Tr{F}{e}{E}$ defined in
\eqref{interior_normal_trace_def} and \eqref{exterior_normal_trace_def}  
belong to $L^{\infty}(\Omega; |D_{H} \chi_{E}|)$, Remark~\ref{weak_conv_absolutely_continuous_perimeter_measures}
and \eqref{boundary_term_div_trace} imply that
\begin{align*} (\rho_{\eps_{k}} \ast \chi_{E}) \frac{ \Tr{F}{e}{E} - \Tr{F}{i}{E}}{4} |D_{H} \chi_{E}| & \weakto \frac{1}{2} \frac{ \Tr{F}{e}{E} - \Tr{F}{i}{E}}{4} |D_{H} \chi_{E}| \\
& = \frac{1}{2} \tildef{\chi_{E}}(1 - \tildef{\chi_{E}}) \div F.  
\end{align*}
Again \eqref{boundary_term_div_trace} shows that the previous
sequences of measures are equal, hence so are their limits.
Taking their difference, we get
\begin{equation*} \left ( \tildef{\chi_{E}} - \frac{1}{2} \right ) \tildef{\chi_{E}}(1 - \tildef{\chi_{E}}) \div F \res \redb E = 0. \end{equation*}
This implies $\displaystyle \tildef{\chi_{E}} = \frac{1}{2}$ $|\div F|$-a.e.\ on 
$\redB{E}$.
From \eqref{tilde_chi_E_div_F} and \eqref{boundary_term_div_trace}, we obtain \eqref{boundary_term_div_1_2}.
\end{proof}

\begin{remark} \label{tripartition_div_2} 
Proposition~\ref{tilde_chi_E} and \eqref{tilde_chi_E_div_F} imply that 
\[
|\div F|\Big(\redB{E} \setminus  \{x\in\Omega: \tildef{\chi_{E}} = 1/2 \}\Big) = 0.
\]
Since Proposition~\ref{tripartition_div_1} states that 
$\Omega = \intE{E} \cup \extE{E}\cup \redB{E} \cup Z_{E}^{F}$, for some $|\div F|$-negligible set $Z_{E}^{F}$, then we get the tripartition 
\begin{equation*} 
\tildef{\chi_{E}}(x) \in \{0, 1, 1/2 \} \ \ \text{for} \ \ |\div F|\text{-a.e.} \ x \in \Omega. 
\end{equation*} 
As a result, we have shown that for every $F \in \DM^{\infty}(H \Omega)$,
taking any weak$^{*}$ limit of $\rho_{\eps} \ast \chi_{E}$ in $L^{\infty}(\Omega; |\div F|)$, this limit attains only the three possible values $1, 0, 1/2$ for $|\div F|$-a.e.\ $x \in \Omega$. This motivates our definitions \eqref{int-ext_F_tilde_chi_E} and \eqref{eq:redBB}.
\end{remark}
\begin{remark} \label{F_D_chi_E_pairing} If $F \in \DM^{\infty}(H \Omega)$ and $E \subset \Omega$ is a set of finite h-perimeter, then
\begin{equation} \label{F_D_chi_E_pairing_eq} 
(F, D_{H} \chi_{E}) =
-\frac{\Tr{F}{i}{E} + \Tr{F}{e}{E}}{2} |D_{H} \chi_{E}|.
\end{equation}
This follows from \eqref{decomposition E and compl} and from the definitions of the normal traces, \eqref{interior_normal_trace_def}, \eqref{exterior_normal_trace_def}.
\end{remark}

We are now arrived at our first general result on the Leibniz rule for divergence-measure horizontal fields and characteristic functions of sets of finite h-perimeter in stratified groups.

\begin{theorem} \label{Leibniz_infty_E_trace}
If $F \in \DM^{\infty}(H \Omega)$ and $E \subset \Omega$ is a set of finite h-perimeter, then we have
\begin{align} \label{Leibniz_infty_E_1_trace} 
\div(\chi_{E} F) & = \chi_{ \intE{E}} \div F - \Tr{F}{i}{E} |D_{H} \chi_{E}|,  \\
\label{Leibniz_infty_E_2_trace} 
\div(\chi_{E} F) & = 
\chi_{ \intE{E}\cup\redB{E}} \div F - \Tr{F}{e}{E} |D_{H} \chi_{E}|, 
\end{align}
where $\tildef{\chi_{E}} \in L^{\infty}(\Omega; |\div F|)$ is 
the weak$^{*}$ limit defined in \eqref{eq:limxF}.
\end{theorem}
\begin{proof}
By Remark \ref{tripartition_div_2}, we have 
\[
\tildef{\chi_{E}}(x) = \chi_{\intE{E}}(x) + \frac{1}{2} \chi_{\redB{E}}(x)\quad \text{ for} \ \ |\div F|-a.e.\ x \in \Omega.
\]
Due to \eqref{F_D_chi_E_pairing_eq}, we can rewrite \eqref{Leibniz_infty_E_1} as follows
\begin{equation*}
\div(\chi_{E} F) = \chi_{\intE{E}} \div F 
+ \frac{1}{2} \chi_{\redB{E}} \div F  - \frac{\Tr{F}{i}{E} + \Tr{F}{e}{E}}{2} |D_{H} \chi_{E}|. \end{equation*}
We can now employ \eqref{boundary_term_div_1_2} to substitute the term
$\chi_{\redB{E}} \div F$, 
obtaining
\begin{align*}
\div(\chi_{E} F) & = \chi_{\intE{E}} \div F 
+ \frac{\Tr{F}{e}{E} - \Tr{F}{i}{E}}{2} |D_{H} \chi_{E}| \res\redB{E} + \\ 
 & - \frac{\Tr{F}{i}{E} + \Tr{F}{e}{E}}{2} |D_{H} \chi_{E}|. \end{align*}
The previous equality immediately gives \eqref{Leibniz_infty_E_1_trace}.
To derive \eqref{Leibniz_infty_E_2_trace}, we simply join \eqref{Leibniz_infty_E_1_trace} with \eqref{boundary_term_div_1_2} and \eqref{trace_equality_boundary}.
\end{proof}

\subsection{Uniqueness results}
The previous results, together with the auxiliary definitions of $\intE{E}$, $\extE{E}$ and $\redB{E}$, allow us to obtain the following uniqueness theorem,
along with a number of relevant consequences.

\begin{theorem}[Uniqueness] \label{Uniqueness_traces}
If $F \in \DM^{\infty}(H \Omega)$ and $E \subset \Omega$ is a set of finite h-perimeter, then there exists a unique $|\div F|$-measurable subset 
$$\intee{E}{F} \subset \Omega \setminus \redb E,$$
up to $|\div F|$-negligible sets, such that 
\begin{equation}\label{chiEtildeUniq}
\tildef{\chi_{E}}(x) = \chi_{\intee{E}{F}}(x) + \frac{1}{2} \chi_{\redb E}(x)\quad \text{ for} \ \ |\div F|\text{-a.e.} \ x \in \Omega.
\end{equation}
In addition, we have 
\beq\label{uniqRedB}
|\div F|\pa{\redb E\sm\redB E}=0
\eeq
and there exist unique normal traces 
$$\Tr{F}{i}{E}, \Tr{F}{e}{E}\in L^{\infty}(\redb E; |D_{H} \chi_{E}|)$$
satisfying
\begin{align} \label{Leibniz_infty_E_1_trace_ref} 
\div(\chi_{E} F) & = \chi_{\intee{E}{F}} \div F - \Tr{F}{i}{E} |D_{H} \chi_{E}|,  \\
\label{Leibniz_infty_E_2_trace_ref} 
\div(\chi_{E} F) & = 
\chi_{\intee{E}{F} \cup \redb E} \div F - \Tr{F}{e}{E} |D_{H} \chi_{E}|.
\end{align}
\end{theorem}

\begin{proof}
Thanks to \eqref{tilde_chi_E_div_no_redb_E} the set 
$$
\intee{E}{F}=\intE{E}\sm\redb E \subset \Omega \setminus \redb E
$$
satisfies the equality $\widetilde{\chi_{E}}=\chi_{\intee{E}{F}}$ $|\div F|$-a.e. in $\Omega\sm\redb E$
and it is uniquely determined up to $|\div F|$-negligible sets. 
By the tripartition stated in Proposition~\ref{tripartition_div_1}, we get 
\begin{equation} \label{eq:contr_argument_tildef_chi_E}
\tildef{\chi_{E}}(x) = \chi_{\intee{E}{F}}(x) + \chi_{\intE{E} \cap \redb E}(x) + \frac{1}{2} \chi_{\redB{E}}(x)\quad \text{ for} \ \ |\div F|\text{-a.e.} \ x \in \Omega,
\end{equation}
from which we get
\begin{equation*}
\intE{E} = \intee{E}{F} \cup (\intE{E} \cap \redb E) \ \text{and} \ \extE{E} = \extee{E}{F} \cup (\extE{E} \cap \redb E),
\end{equation*}
where $\extee{E}{F} = \Omega \setminus ( \intee{E}{F} \cup \redb E)$.
Let $E$ be a set of finite h-perimeter in $\Omega$. We notice that $\redb E$ is a Borel set, by definition, so that, if $F \in \DM^{\infty}(H \Omega)$, the measure $|\div F| \res \redb E$ is well defined.
Let $\eps_{k}$ be the fixed nonnegative vanishing sequence such that $\rho_{\eps_{k}} \ast \chi_{E} \weakstarto \tildef{\chi_{E}}$ in $L^{\infty}(\Omega; |\div F|)$. Then, we also have
\begin{equation} \label{eq:weak_star_conv_redb}
\rho_{\eps_{k}} \ast \chi_{E} \weakstarto \tildef{\chi_{E}} \ \ \text{in} \  L^{\infty}(\Omega; |\div F| \res \redb E).
\end{equation}
To see this, it is enough to multiply any test functions $\psi \in L^{1}(\Omega; |\div F|\res\redb E)$ with $\chi_{\redb E}$, getting a function in $L^{\infty}(\Omega; |\div F|)$.

Now, Theorem~\ref{absolute continuity} shows that $|\div F| \ll \SHaus{Q - 1}$, so that $$|\div F| \res \redb E \ll \SHaus{Q - 1} \res \redb E,$$ which, by Theorem~\ref{perimeter repr}, gives
\begin{equation} \label{eq:abs_cont_div_ref}
|\div F| \res \redb E \ll |D_{H} \chi_{E}|.
\end{equation}
Then, it is easy to see that \eqref{eq:abs_cont_div_ref} implies $|\div F| \res \redb E = \theta_{F, E} |D_{H} \chi_{E}|$, for some $\theta_{F, E} \in L^{1}(\Omega; |D_{H} \chi_{E}|)$, by Radon-Nikod\'ym theorem.
We recall that, by \eqref{limit1/2}, we have $$\rho_{\eps} \ast \chi_{E} \weakstarto 1/2$$ in $L^{\infty}(\Omega; |D_{H} \chi_{E}|)$.
Due to Lemma~\ref{lemma:weak_conv_absolutely_continuous_perimeter_measures}, it follows that
\begin{equation*}
(\rho_{\eps} \ast \chi_{E}) |\div F| \res \redb E \weakto \frac{1}{2} |\div F| \res \redb E.
\end{equation*}
On the other hand, \eqref{eq:weak_star_conv_redb} implies
\begin{equation*}
(\rho_{\eps_{k}} \ast \chi_{E}) |\div F| \res \redb E \weakto \tildef{\chi_{E}} |\div F| \res \redb E,
\end{equation*}
and so we conclude that any weak* limit point of $\{\rho_{\eps} \ast \chi_{E}\}_{\eps > 0}$ in $L^{\infty}(\Omega; |\div F|)$ must satisfy
$\tildef{\chi_{E}}(x) = \frac{1}{2}$ for $|\div F| \res \redb E$-a.e. $x \in \Omega$. 
Clearly, this means that  
\begin{equation} \label{eq:tildef_chi_E_redb}
\tildef{\chi_{E}}(x) = \frac{1}{2} \ \ \text{for} \ |\div F|\text{-a.e.} \ x \in \redb E.
\end{equation}
As an immediate consequence, we obtain
\begin{equation*}
|\div F|(\tildef{E^{1}} \cap \redb E) = 0 \ \ \text{and} \ \ |\div F|(\tildef{E^{0}} \cap \redb E) = 0,
\end{equation*}
which implies \eqref{uniqRedB}.
Thus, combining these results with \eqref{eq:contr_argument_tildef_chi_E}, we deduce that there exists a unique $|\div F|$-measurable set $\intee{E}{F} \subset \Omega \setminus \redb E$ such that \eqref{chiEtildeUniq} holds.
Hence, there exists a unique weak* limit $\tildef{\chi_{E}}$ of $\{ \rho_{\eps} \ast \chi_{E}\}_{\eps > 0}$ in $L^{\infty}(\Omega; |\div F|)$. Thanks to \eqref{Leibniz_infty_E_1_trace} and \eqref{Leibniz_infty_E_2_trace}, we obtain the uniqueness of the normal traces. Indeed, we have
\begin{align*} 
\div(\chi_{E} F) - \chi_{\intee{E}{F}} \div F & = - \Tr{F}{i}{E} |D_{H} \chi_{E}|,  \\
\div(\chi_{E} F) - \chi_{\intee{E}{F} \cup \redb E} \div F & = - \Tr{F}{e}{E} |D_{H} \chi_{E}|, 
\end{align*}
and the uniqueness of the terms on the left hand sides implies the uniqueness of those on the right hand sides.
\end{proof}

In view of the previous uniqueness result, we are in the position to
introduce the following definition.

\begin{definition}\label{def:measthinterior}
Let $E\subset\Omega$ be a set of finite h-perimeter and let 
$F\in \DM^{\infty}(H \Omega)$. We define 
the {\em measure theoretic interior of $E$ with respect to $F$} as $\intee{E}{F}\subset\Omega\sm\redb E$, such that 
\begin{equation} \label{defE1F}
\div(\chi_{E} F) \res (\Omega \setminus \redb E) = \chi_{\intee{E}{F}} \div F.
\end{equation}
Analogously, we define the {\em measure theoretic exterior of $E$ with respect to $F$} as a set $\extee{E}{F}\subset\Omega\sm\redb E$ such that  
\begin{equation} \label{defE0F}
\div(\chi_{\Omega\sm E} F) \res (\Omega \setminus \redb E) = \chi_{\extee{E}{F}} \div F.
\end{equation}
\begin{comment}
\begin{equation} \label{defE0F}
\div(\chi_{\Omega\sm E} F) \res (\Omega \setminus \redb E) =\widetilde{\chi_{\Omega\sm E}}\div F=(1-\widetilde{\chi_E})\div F\res (\Omega \setminus \redb E)= \chi_{\extee{E}{F}} \div F.
\end{equation}
\end{comment}
\end{definition}

\begin{remark}\label{r:E1F}
The existence of $\intee{E}{F}$, along with its uniqueness up to 
$|\div F|$-negligible sets, is a direct consequence of restricting
\eqref{Leibniz_infty_E_1_trace_ref} to $\Omega\sm \redb E$. Analogously, the existence and uniqueness up to $|\div F|$-negligible sets of $\extee{E}{F}$ follows from applying \eqref{Leibniz_infty_E_1_trace_ref} to $\Omega \setminus E$ and restricting it to $\Omega\sm \redb E$. In addition, we have $$|\div F| \left (\Omega \setminus ( \intee{E}{F} \cup \extee{E}{F} \cup \redb E) \right ) = 0,$$
since
\begin{align*}
\chi_{\extee{E}{F}} \div F & = \div ( \chi_{\Omega \setminus E} F) \res (\Omega \setminus \redb E) = \chi_{\Omega \setminus \redb E} \div F - \div (\chi_{E} F) \res (\Omega \setminus \redb E) \\
& = \left (\chi_{\Omega \setminus \redb E} - \chi_{\intee{E}{F}} \right ) \div F = \chi_{\Omega \setminus ( \intee{E}{F} \cup \redb E)} \div F,
\end{align*}
thanks to \eqref{defE1F}.
\end{remark}

\begin{remark} \label{rem:unique_pairing}
Theorem~\ref{Uniqueness_traces} shows that the
interior and exterior normal traces $\Tr{F}{i}{E}$ and $\Tr{F}{e}{E}$ are unique up to $|D_H\chi_E|$-negligible sets. As an immediate consequence of this fact, joined with \eqref{interior_normal_trace_def} and \eqref{exterior_normal_trace_def}, we see that also the pairings $(\chi_{E} F, D_{H} \chi_{E})$ and $(\chi_{\Omega \setminus E} F, D_{H} \chi_{E})$ are uniquely determined and do not depend on the approximating sequences $\ban{\chi_{E} F, \nabla_{H}(\rho_{\eps_{k}} \ast \chi_{E})} \mu$ and $\ban{\chi_{\Omega \setminus E} F, \nabla_{H}(\rho_{\eps_{k}} \ast \chi_{E})} \mu$.
In addition, \eqref{F_D_chi_E_pairing_eq} also shows that the pairing 
$(F,D_H\chi_E)$ is unique and independent from the choice of the approximating sequence.
\end{remark}

We conclude this section with the following easy refinement of \eqref{boundary_term_div_1_2}.

\begin{corollary} 
Let $F \in \DM^{\infty}(H \Omega)$ and $E$ be a set of finite h-perimeter. Then, we have
\begin{equation} \label{boundary_term_div_1_2_ref} \chi_{\redb{E}} \div F = (\Tr{F}{e}{E} - \Tr{F}{i}{E}) |D_{H} \chi_{E}|.
\end{equation}
\end{corollary}
\begin{proof}
It is enough to subtract \eqref{Leibniz_infty_E_1_trace_ref} from \eqref{Leibniz_infty_E_2_trace_ref}.
\end{proof}

\section{Gauss--Green and integration by parts formulas}

This section is devoted to establish different Gauss--Green formulas
and integration by parts formula in stratified groups. 
Throughout we shall use the measure theoretic interior 
$\intee{E}{F}\subset\G$ introduced in Definition~\ref{def:measthinterior}. We start with a general version of the Gauss--Green formulas, which is a direct consequence of Theorem \ref{Uniqueness_traces}.

\begin{theorem}[Gauss--Green formula I]
Let $F \in \DM^{\infty}(H \Omega)$ and $E \Subset \Omega$ be a set of finite h-perimeter. Then, we have
\begin{align}
\label{G-G groups 1} \div F(\intee{E}{F}) & = \int_{\redb E} \Tr{F}{i}{E} \, d |D_{H} \chi_{E}|, \\
\label{G-G groups 2} \div F(\intee{E}{F}\cup \redb E) & = \int_{\redb E} \Tr{F}{e}{E} \, d |D_{H} \chi_{E}|.
\end{align}
\end{theorem}
\begin{proof}
If we evaluate \eqref{Leibniz_infty_E_1_trace_ref} and \eqref{Leibniz_infty_E_2_trace_ref} on $\Omega$, we obtain 
\begin{align*}
\div(\chi_{E} F)(\Omega) & = \div F(\intee{E}{F}) - \int_{\redb E} \Tr{F}{i}{E} \, d |D_{H} \chi_{E}|, \\
\div(\chi_{E} F)(\Omega) & = \div F(\intee{E}{F}\cup \redb E) - \int_{\redb E} \Tr{F}{e}{E} \, d |D_{H} \chi_{E}|.
\end{align*}
Then, we exploit the fact that $\chi_{E} F \in \DM^{\infty}(H \Omega)$, thanks to Theorem \ref{productrule}, and Lemma~\ref{DMcomptsupp} in order to conclude that $\div(\chi_{E} F)(\Omega) = 0$.
\end{proof}

We consider now some special cases in which the normal traces coincides; that is, in which there are no jumps along the reduced boundary of the integration domain. To this purpose, we give the following definition.

\begin{definition} \label{def:average_normal_trace}
Let $F \in \DM^{\infty}(H \Omega)$ and let $E \subset \Omega$ be a set of finite h-perimeter. We define the {\em average normal trace} $\Tr{F}{}{E}$ as the function in $L^{\infty}(\Omega; |D_{H} \chi_{E}|)$ satisfying
\begin{equation} \label{density_pairing_unique_trace} 
(F, D_{H} \chi_{E}) = - \Tr{F}{}{E} |D_{H} \chi_{E}|.
\end{equation}
\end{definition}

\begin{remark}
Thanks to \eqref{abs_cont_pairing_infty}, we have
\[
|(F, D_{H} \chi_{E})| \le \|F\|_{\infty, \Omega} |D_{H} \chi_{E}|,
\]
which implies the existence of $\Tr{F}{}{E} \in L^{\infty}(\Omega; |D_{H} \chi_{E}|)$ satisfying \eqref{density_pairing_unique_trace}, by Radon-Nikod\'ym theorem. In addition, \eqref{F_D_chi_E_pairing_eq} shows that
\begin{equation} \label{eq:average_normal_trace}
\Tr{F}{}{E} = \frac{\Tr{F}{i}{E} + \Tr{F}{e}{E}}{2}  \ \ |D_{H} \chi_{E}|\text{-a.e. in} \  \Omega. 
\end{equation}
\end{remark}

\begin{proposition} \label{G-G_groups_no_red_boundary}
Let $F \in \DM^{\infty}(H \Omega)$ and let $E \subset \Omega$ be a set of finite h-perimeter such that $|\div F|(\redb E) = 0$.
Then we have
\begin{equation} \label{trace_equality_abs_cont} 
\Tr{F}{i}{E} = \Tr{F}{e}{E} = \Tr{F}{}{E} \ \ |D_{H} \chi_{E}|\text{-a.e. in} \ \Omega. \end{equation}
As a consequence, we obtain 
\begin{equation} \label{productrulechi_no_red_boundary} 
\div(\chi_{E} F) = \chi_{\intee{E}{F}} \div F + (F, D_{H} \chi_{E}) = \chi_{\intee{E}{F}} \div F -  \Tr{F}{}{E} |D_{H} \chi_{E}|.
\end{equation}
\end{proposition}
\begin{proof} 
Equality \eqref{trace_equality_abs_cont} is an immediate consequence of $|\div F|(\redb E) = 0$, \eqref{boundary_term_div_1_2_ref} and \eqref{eq:average_normal_trace}.
Then, by combining \eqref{Leibniz_infty_E_1_trace_ref}, \eqref{density_pairing_unique_trace} and \eqref{trace_equality_abs_cont} we obtain \eqref{productrulechi_no_red_boundary}.
\end{proof}

The previous result immediately gives a new version of the Gauss--Green formula without jumps on the reduced boundary of the domain.

\begin{theorem}[Gauss--Green formula II] \label{G-G groups general II} 
Let $F \in \DM^{\infty}(H \Omega)$ and let $E \Subset \Omega$
be a set of finite h-perimeter with $|\div F|(\redb E)=0$.
Then there exists a unique normal trace $\Tr{F}{}{E}  \in L^{\infty}(\Omega; |D_{H} \chi_{E}|)$ such that
\begin{equation} \label{G-G_no_red_boundary}
\div F(\intee{E}{F}) = \int_{\redb E} \Tr{F}{}{E} \, d |D_{H} \chi_{E}|.
\end{equation}
\end{theorem}
\begin{proof}
The existence of a unique normal trace $\Tr{F}{}{E}$ follows from Proposition \ref{G-G_groups_no_red_boundary}. Then, we evaluate \eqref{productrulechi_no_red_boundary} on $\Omega$ and apply Lemma~\ref{DMcomptsupp}, and thus we obtain \eqref{G-G_no_red_boundary}.
\end{proof}

We now prove that, in the case $F$ is continuous, the normal traces are equal and coincide with the scalar product in the horizontal section.

\begin{proposition} \label{F_continuous}
Let $F \in \DM^{\infty}(H \Omega) \cap C(H\Omega)$ and $E \subset \Omega$ be a set of finite h-perimeter. Then we have
\begin{equation} \label{F_continuous_trace} 
\Tr{F}{i}{E}(x) = \Tr{F}{e}{E}(x) = \ban{F(x), \nu_{E}(x)} \ \ \text{for} \ \ |D_{H} \chi_{E}|\text{-a.e.} \  x \in \redb E
\end{equation}
and in particular $|\div F|(\redb{E}) = 0$.
\end{proposition}
\begin{proof} 
Let $\phi \in C_{c}(\Omega)$ and let $(\chi_{E} F, D_{H} \chi_{E})$
be as defined in \eqref{eq:chiEFconvol}, hence
\begin{equation*} 
\int_{\Omega} \phi \, d (\chi_{E} F, D_{H} \chi_{E}) = \lim_{\eps \to 0} \int_{\Omega} \ban{\phi F, \chi_{E} (\rho_{\eps} \ast D_{H} \chi_{E})} \, dx. \end{equation*}
We observe that $\phi F \in C_{c}(\Omega, H \Omega)$ and 
taking into account that $\nu_E$ is the measure theoretic exterior h-normal,
by \eqref{overline_D_chi_E_1} we obtain
\begin{equation*} \int_{\Omega} \phi \, d (\chi_{E} F, D_{H} \chi_{E}) = - \int_{\Omega} \frac{1}{2} \phi \ban{F, \nu_{E}} \, d |D_{H} \chi_{E}|. \end{equation*}
By definition of interior normal trace \eqref{interior_normal_trace_def}, we obtain that
\[
\Tr{F}{i}{E}(x) = \ban{F(x), \nu_{E}(x)} \quad \text{for $|D_{H} \chi_{E}|$-a.e.\ $x \in \Omega$,}
\]
which implies \eqref{F_continuous_trace} for the interior normal trace.
The identity for the exterior normal trace in \eqref{F_continuous_trace} can be proved in an analogous way, employing \eqref{overline_D_chi_E_2} and definition \eqref{exterior_normal_trace_def}.
Finally, in view of \eqref{boundary_term_div_1_2_ref}, we get $|\div F|(\redb{E}) = 0$.
\end{proof}

\begin{theorem}[Gauss--Green formula III] \label{G-G groups general III}
Let $F \in \DM^{\infty}(H \Omega) \cap C(H\Omega)$ and let $E \Subset \Omega$ be a set of finite h-perimeter. Then the following formula holds
\begin{equation}
\label{G-G groups continuous} 
\div F(\intee{E}{F}) = \int_{\redb E}  \ban{F,\nu_E} \, d |D_{H} \chi_{E}|. \end{equation}
\end{theorem}
\begin{proof}
By Proposition~\ref{F_continuous}, we have a unique normal trace, that $|D_H\chi_E|$-a.e.\ in $\redb E$ equals the scalar product $\ban{F,\nu_E}$. In addition,  $|\div F|(\redb{E}) = 0$. Since $E \Subset \Omega$, we may apply \eqref{G-G_no_red_boundary} to conclude the proof.
\end{proof}

Next, we apply the Leibniz rule (Theorem \ref{productrule}) to derive integration by parts formulas.

\begin{theorem}[Integration by parts I] \label{IBP_general} Let $F \in \DM^{\infty}_{\rm loc}(H \Omega)$, $E$ be a set of locally finite h-perimeter in $\Omega$ and $\varphi \in C(\Omega)$ with $\nabla_{H} \varphi \in L^{1}_{\rm loc}(H \Omega)$ such that $$\mathrm{supp}(\varphi \chi_{E}) \Subset \Omega.$$ 
Then, there exist interior and exterior normal traces $$\Tr{F}{i}{E}, \Tr{F}{e}{E} \in L^{\infty}_{\rm loc}(\Omega; |D_{H} \chi_{E}|)$$ such that the following formulas hold
\begin{align} \label{IBP_1} \int_{\intee{E}{F}} \varphi \, d \div F + \int_{E} \ban{F, \nabla_{H} \varphi} \, dx & = \int_{\redb E} \varphi \Tr{F}{i}{E} \, d |D_{H} \chi_{E}|, \\
\label{IBP_2} \int_{\intee{E}{F} \cup \redb E} \varphi \, d \div F + \int_{E} \ban{F, \nabla_{H} \varphi} \, dx & = \int_{\redb E} \varphi \Tr{F}{e}{E} \, d |D_{H} \chi_{E}|. \end{align}
In addition, for any open set $U \Subset \Omega$, we have the following estimates
\begin{align} \label{interior_normal_trace_norm_loc} \|\Tr{F}{i}{E}\|_{L^{\infty}(\redb E \cap U; |D_{H} \chi_{E}|)} & \le \| F\|_{\infty, E \cap U}, \\
 \label{exterior_normal_trace_norm_loc} \|\Tr{F}{e}{E}\|_{L^{\infty}(\redb E \cap U; |D_{H} \chi_{E}|)} & \le \| F\|_{\infty, U \setminus E}. \end{align}
\end{theorem}
\begin{proof} Let $U \Subset \Omega$ be an open set such that $\mathrm{supp}(\varphi \chi_{E}) \subset U$. Then, we clearly have $F \in \DM^{\infty}(H U)$, $\chi_{E} \in BV(U)$ and $\varphi \in C(U) \cap L^{\infty}(U)$ with $\nabla_{H} \varphi \in L^{1}(H U)$. 
Hence, Theorem~\ref{productrule} implies that $\chi_{E} F \in \DM^{\infty}(H U)$ and so we can apply \eqref{product rule eq p} to $\varphi$ and $\chi_{E} F$, thus obtaining
\begin{equation} \label{productrule_phi} \div(\varphi \chi_{E} F) = \varphi \div(\chi_{E} F) + \ban{\chi_{E} F, \nabla_{H} \varphi} \mu \end{equation}
in the sense of Radon measures on $U$. Now, Theorem \ref{Uniqueness_traces} implies the existence of interior and exterior normal traces $\Tr{F}{i}{E}, \Tr{F}{e}{E}$ in $L^{\infty}(U; |D_{H} \chi_{E}|)$, and, by \eqref{Leibniz_infty_E_1_trace_ref} and \eqref{Leibniz_infty_E_2_trace_ref}, we get
\begin{align} \label{productrule_phi_1} \div(\varphi \chi_{E} F) & = \chi_{\intee{E}{F}} \varphi \div F - \varphi \Tr{F}{i}{E} |D_{H} \chi_{E}| + \chi_{E} \ban{F, \nabla_{H}\varphi} \mu, \\
\label{productrule_phi_2} \div(\varphi \chi_{E} F) & = \chi_{\intee{E}{F} \cup \redb E} \varphi \div F - \varphi \Tr{F}{e}{E} |D_{H} \chi_{E}| + \chi_{E} \ban{F, \nabla_{H}\varphi} \mu. \end{align}
Finally, we evaluate \eqref{productrule_phi_1} and \eqref{productrule_phi_2} on $U$, and we employ Lemma \ref{DMcomptsupp} and the assumption that $\mathrm{supp}(\varphi \chi_{E}) \subset U$, so that \eqref{IBP_1} and \eqref{IBP_2} immediately follow.
The estimates \eqref{interior_normal_trace_norm_loc} and \eqref{exterior_normal_trace_norm_loc} follow from the restriction of $F$ and $\chi_{E}$ to $U$ and from \eqref{interior_normal_trace_norm} and \eqref{exterior_normal_trace_norm}.
\end{proof}

\begin{remark} We notice that the local statement of Theorem \ref{IBP_general} shows that the field $F$ needs not be essentially bounded on the whole set $\Omega$, but only on an arbitrarily small neighborhood of $\redb E$. In particular, let $\eps > 0$ and $E \Subset \Omega$. We define
\begin{align*} E_{\eps} & := \{ x \in E : \mathrm{dist}(x, \redb E) < \eps \}, \\
E^{\eps} & := \{ x \in \Omega \setminus E : \mathrm{dist}(x, \redb E) < \eps \}. \end{align*} 
Then, from \eqref{interior_normal_trace_norm_loc} and \eqref{exterior_normal_trace_norm_loc} one can deduce that we have
\begin{align*} \|\Tr{F}{i}{E}\|_{L^{\infty}(\redb E; |D_{H} \chi_{E}|)} & \le \inf_{\eps > 0} \| F\|_{\infty, E_{\eps}}, \\
\|\Tr{F}{e}{E}\|_{L^{\infty}(\redb E; |D_{H} \chi_{E}|)} & \le \inf_{\eps > 0} \| F\|_{\infty, E^{\eps}}. \end{align*}
Indeed, it is enough to take the open set $U = E_{\eps} \cup E^{\eps} = \{ x \in \Omega : \mathrm{dist}(x, \redb E) < \eps \}$ for some $\eps > 0$ such that $U \Subset \Omega$, and then to pass to the infimum in $\eps$.
\end{remark}

As an application of the integration by parts formulas, one can generalize the classical Euclidean Green's identities to $C^1_{H}(\Omega)$ functions whose horizontal gradients are in $\DM^{\infty}_{\rm loc}(H \Omega)$.

In the spirit of Definition \ref{d:divergence_distrib}, we can define the distributional sub-Laplacian of a locally summable function $u : \Omega \to \R$ with horizontal gradient satisfying $\nabla_{H} u \in L^{1}_{\rm loc}(H \Omega)$ as the distribution
\begin{equation} \label{distributional_sub_laplacian}
C_c^\infty(\Omega)\ni \phi\mapsto  - \int_{\Omega} \ban{\nabla_{H} u, \nabla_H \phi} \, dx.
\end{equation}
We shall denote the distributional sub-Laplacian of $u$ by $\Delta_{H} u$ and, with a little abuse of notation, we shall use the same symbol to denote also the measurable function defining the distribution, whenever it exists. Arguing as in the paragraph after Remark \ref{Lip test}, one can show that, if $u \in C^{2}_{H}(\Omega)$, then its distributional sub-Laplacian coincides with the pointwise sub-Laplacian, and so we can write
\begin{equation*} \Delta_{H} u = \sum_{j = 1}^{\m} X_{j}^{2} u. \end{equation*}

\begin{theorem}[Green's identities I] \label{Green_identity}  Let $u \in C^1_{H}(\Omega)$ satisfy $\Delta_{H} u \in \mathcal{M}_{\rm loc}(\Omega)$ and let $E \subset \Omega$ be a set of locally finite h-perimeter in $\Omega$. Then for each $v \in C_{c}(\Omega)$ with $\nabla_{H} v \in L^{1}(H \Omega)$, one has
\begin{equation} \label{GFI}
\int_{E^{1}_{u}} v \, d\Delta_{H} u  = \int_{\redb E} v \, \ban{\nabla_{H} u, \nu_{E}} \, d |D_{H} \chi_{E}| - \int_{E} \ban{\nabla_{H} v, \nabla_{H} u} \, dx, 
\end{equation}
where $E^{1}_{u} := \intee{E}{\nabla_{H} u}$ is the measure theoretic interior of $E$ with respect to $\nabla_Hu$.
If $u, v \in C^1_{H, c}(\Omega)$ also satisfy $\Delta_{H} u, \Delta_{H} v \in \mathcal{M}(\Omega)$, then we have
\begin{equation} \label{GSI}
\int_{E^{1}_{u}} v \, d\Delta_{H} u  - \int_{E^{1}_{v}} u \, d\Delta_{H} v  = \int_{\redb E} \ban{v \nabla_{H} u - u \nabla_{H} v, \nu_{E}} \, d |D_{H} \chi_{E}|, \end{equation}
where $E^{1}_{v} := \intee{E}{\nabla_{H} v}$, analogously as with $u$. If $E \Subset \Omega$, one can drop the assumption that $u$ and $v$ have compact support in $\Omega$.
\end{theorem}
\begin{proof} Arguing as in the proof of Theorem \ref{IBP_general}, we can localize to an open set $U \Subset \Omega$ such that $\mathrm{supp}(v \chi_{E}) \subset U$. Then, we notice that, since $u \in C^{1}_{H}(U)$ and $\Delta_{H} u \in \mathcal{M}(U)$, then ${\nabla_{H} u \in \DM^{\infty}(H U) \cap C(HU)}$. Thus, since $E$ is a set of finite h-perimeter in $U$, the normal traces of $\nabla_{H} u$ on $\redb E \cap U$ coincide with $\ban{\nabla_{H} u(x), \nu_{E}(x)}$ for $|D_{H} \chi_{E}|$-a.e.\ $x \in U$, by Proposition~\ref{F_continuous}. 
In addition, Proposition~\ref{F_continuous} gives $|\Delta_{H} u|(\redb E) = 0$, and so \eqref{IBP_1} implies \eqref{GFI}, if we set $\intee{E}{\nabla_{H} u} =: E^{1}_{u}$.

If now $u, v \in C^1_{H, c}(\Omega)$ and satisfy $\Delta_{H} u, \Delta_{H} v \in \mathcal{M}(\Omega)$, one also has \eqref{GFI} with the roles of $u$ and $v$ interchanged, and thus with a set $E^{1}_{v}$ uniquely determined by $\nabla_{H} v$, instead. Subtracting these two expressions leads to \eqref{GSI}.
If $E \Subset \Omega$, then the assumption on the compact support of $u$ and $v$ are not anymore needed.
\end{proof}

\section{Absolutely continuous divergence-measure horizontal fields}

This section is devoted to some first applications of our previous results, which cover the case of $F \in L^{\infty}(H \Omega)$ with $\div F \in L^{1}(\Omega)$, and consequently the cases $F \in W^{1, 1}(H \Omega)$ and $F \in \Lip_{H, c}(H \Omega)$. 

\begin{theorem} \label{divF_abs_continuous} If $F \in \DM^{\infty}(H \Omega)$ and $|\div F| \ll \mu$, then for any set of finite h-perimeter $E \subset \Omega$, we have 
\begin{equation} \label{productrulechi_abs_cont} \div(\chi_{E} F) = \chi_{E} \div F + (F, D_{H} \chi_{E}) \end{equation}
in the sense of Radon measures on $\Omega$. Therefore, we also obtain \eqref{trace_equality_abs_cont},
\begin{equation} \label{eq:E_1_F_repr_abs_cont}
|\div F|(\intee{E}{F} \Delta E) = 0,
\end{equation} 
and
\begin{equation} \label{productrulechi_abs_cont_trace} \div(\chi_{E} F) = \chi_{E} \div F - \Tr{F}{}{E} |D_{H} \chi_{E}|. \end{equation} 
\end{theorem}
\begin{proof} 
Formula \eqref{productrulechi_abs_cont} is a simple application of \eqref{product rule eq Leb} to $g = \chi_{E}$, taking into account \eqref{eq:BVdec-as}. Since the absolute continuity assumption and \eqref{LebesgueNegBdry} give $|\div F|(\redb E) = 0$, we can apply Proposition~\ref{G-G_groups_no_red_boundary} and obtain \eqref{trace_equality_abs_cont}. In addition, by comparing \eqref{productrulechi_no_red_boundary} and \eqref{productrulechi_abs_cont}, we get \eqref{eq:E_1_F_repr_abs_cont}. Thus, \eqref{productrulechi_abs_cont_trace} is an immediate consequence of \eqref{productrulechi_no_red_boundary} and \eqref{eq:E_1_F_repr_abs_cont}.
\end{proof}
 
Thanks to Theorem \ref{divF_abs_continuous}, the proofs of Theorem \ref{divF_abs_cont_G_G} and Theorem~\ref{IBP_divF_abs_cont} can be immediately achieved.

\begin{proof}[Proof of Theorem \ref{divF_abs_cont_G_G}]
We evaluate \eqref{productrulechi_abs_cont_trace} on $\Omega$ and apply Lemma \ref{DMcomptsupp}, thanks to the fact that $E \Subset \Omega$.
\end{proof}

\begin{proof}[Proof of Theorem~\ref{IBP_divF_abs_cont}]
It suffices to combine \eqref{trace_equality_abs_cont}, \eqref{eq:E_1_F_repr_abs_cont} and $|\div F|(\redb E) = 0$ with Theorem~\ref{IBP_general}.
\end{proof}

We notice now that Theorem~\ref{IBP_divF_abs_cont} may be applied to a set of locally finite h-perimeter whose reduced boundary is not rectifiable in the Euclidean sense.

\begin{example}\label{ex:notrectif} We recall that a set $S \subset \G$ is called a $C^{1}_{H}$-regular surface if, for any $p \in S$, there exists an open set $U$ and a map $f \in C^{1}_{H}(U)$ such that
\begin{equation*} S \cap U = \{ q \in U: f(q) = 0 \ \text{and} \ \nabla_{H} f(p) \neq 0\}. \end{equation*} 
In \cite[Theorem 3.1]{cassano2004rectifiability}, the authors proved the existence of a $C^{1}_{H}$-regular surface $S$ in the Heisenberg group $\H^1$ such that $\Haus{\frac{5 - \eps}{2}}_{|\cdot|}(S) > 0$ for any $\eps \in (0, 1)$; which means that $S$ is not $2$-Euclidean rectifiable. In particular, they showed that there exists a function $f \in C^{1}_{H}(\H^{1})$ related to $S$ as above, with $U = \H^1$. From \cite[Theorem 2.1]{franchi2003regular}, it is known that the open set $E = \{ p \in \H^{1} : f(p) < 0 \}$ is of locally finite h-perimeter and $\redb E = S$.
Thus, given any $F \in \DM^{\infty}_{\rm loc}(H \H^1)$ such that $|\div F| \ll \mu = \Leb{3}$, we can apply 
Theorem~\ref{IBP_divF_abs_cont} to $F$ and $E$ to show that there exists a unique normal trace $\Tr{F}{}{E} \in L^{\infty}_{\rm loc}(\H^1; |D_{H} \chi_{E}|)$. In addition, for any $\varphi \in C_c(\H^1)$ with $\nabla_{H} \varphi \in L^{1}(H \H^1)$ we obtain
\begin{equation*} \int_{E} \varphi \, d \div F + \int_{E} \ban{F, \nabla_{H} \varphi} \, dx = \int_{S} \varphi \Tr{F}{}{E} \, d |D_{H} \chi_{E}|. \end{equation*}
We stress the fact that, on the right hand side, we are integrating on a fractal object, which is an Euclidean unrectifiable set.
\end{example}

\begin{theorem}[Green's identities II]\label{t:GreenII}
Let $u \in C^1_{H}(\Omega)$ be such that $\Delta_{H} u \in \mathcal{M}_{\rm loc}(\Omega)$ with $|\Delta_{H} u| \ll \mu$ and let $E \subset \Omega$ be a set of locally finite h-perimeter in $\Omega$. Then for each $v \in C_{c}(\Omega)$ with $\nabla_{H} v \in L^{1}(H \Omega)$ one has
\begin{equation}\label{eq:Green1}
\int_{E} v \, d\Delta_{H} u  = \int_{\redb E} v \, \ban{\nabla_{H} u, \nu_{E}} \, d |D_{H} \chi_{E}| - \int_{E} \ban{\nabla_{H} v, \nabla_{H} u} \, dx. 
\end{equation}
If $u, v \in C^1_{H, c}(\Omega)$ also satisfy $\Delta_{H} u, \Delta_{H} v \in \mathcal{M}(\Omega)$, $|\Delta_{H} u| \ll \mu, |\Delta_{H} v| \ll \mu$, one has
\begin{equation} \label{eq:Green2}
\int_{E} v \, d\Delta_{H} u  - u \, d\Delta_{H} v  = \int_{\redb E} \ban{v \nabla_{H} u - u \nabla_{H} v, \nu_{E}} \, d |D_{H} \chi_{E}|. \end{equation}
If $E \Subset \Omega$, one can drop the assumption that $u$ and $v$ have compact support in $\Omega$.
\end{theorem}
\begin{proof}
It suffices to combine the results of Theorem~\ref{Green_identity} with the fact that, in this case, $E^{1}_{u} = E^{1}_{v} = E$ up to $\mu$-negligible sets, which follows from 
Theorem~\ref{divF_abs_continuous}.
\end{proof}

It is worth noticing that one could weaken the absolute continuity assumption on $\div F$, by asking only $|\div F|(\mtbR E) = 0$. We also notice that, in the Euclidean context, this resembles part of the hypotheses assumed by Degiovanni, Marzocchi and Musesti in \cite[Theorem 5.2]{degiovanni1999cauchy} and Schuricht in \cite[Proposition 5.11]{Schuricht}. However, we do not require the existence of any suitable smooth approximation of $F$, as they do: thus, our results are more general, even though we cannot represent the normal traces as the classical scalar product. 

Before stating our results, we need a preliminary lemma.

\begin{lemma}\label{lem:R_repr_equality} 
Let $F \in \DM^{\infty}(H \Omega)$ and let $E \subset \Omega$ be a set of finite h-perimeter such that \eqref{eq:assumption_precise_right_repr_div} holds. Then, we have 
\begin{equation} \label{eq:E_1_0_F_R}
|\div F|(\intee{E}{F} \Delta E^{1, \cR}) = 0 \quad \text{and} \quad |\div F|(\extee{E}{F} \Delta E^{0, \cR}) = 0.
\end{equation}
\end{lemma}
\begin{proof}
Thanks to Remark \ref{r:tripartition_div_2}, we see that 
\begin{equation*}
|\div F|\left ( (\intee{E}{F} \Delta E^{1, \cR}) \cap C_{E}^{\cR} \right ) = 0 \ \text{and} \ |\div F|\left ( (\extee{E}{F} \Delta E^{0, \cR} ) \cap C_{E}^{\cR} \right ) = 0.
\end{equation*}
If \eqref{eq:assumption_precise_right_repr_div} holds, then it is clear that we obtain \eqref{eq:E_1_0_F_R}.
\end{proof}

\begin{proposition} \label{G-G_groups_no_boundary} Let $F \in \DM^{\infty}(H \Omega)$ and let $E \subset \Omega$ be a set of finite h-perimeter such that $|\div F|(\mtbR E) = 0$. Then we have \eqref{trace_equality_abs_cont}, \eqref{eq:E_1_0_F_R} and
\begin{equation} \label{productrulechi_no_boundary_trace} \div(\chi_{E} F) = \chi_{E^{1, \cR}} \div F - \Tr{F}{}{E} |D_{H} \chi_{E}|. \end{equation}
\end{proposition}
\begin{proof} 
We notice that $\Omega \setminus C_{E}^{\cR} \subset \mtbR E$, therefore we can apply Lemma \ref{lem:R_repr_equality} to get \eqref{eq:E_1_0_F_R}. Thanks to \eqref{Leibniz_infty_E_1_trace_ref} and \eqref{eq:E_1_0_F_R}, we easily get
\begin{equation} \label{eq:productrulechi_no_boundary_trace_1}
\div(\chi_{E} F)  = \chi_{E^{1, \cR}} \div F - \Tr{F}{i}{E} |D_{H} \chi_{E}|.
\end{equation}
In addition, \eqref{eq:E_1_0_F_R} and $|\div F|(\mtbR E) = 0$ imply that 
\begin{align*}
|\div F|(\redb E) & \le |\div F|(E^{1, \cR} \cap \redb E) + |\div F|(E^{0, \cR} \cap \redb E) \\
& = |\div F|(\intee{E}{F} \cap \redb E) + |\div F|(\extee{E}{F} \cap \redb E) = 0,
\end{align*}
since $\intee{E}{F}, \extee{E}{F} \subset \Omega \setminus \redb E$ by Definition \ref{def:measthinterior}. As an immediate consequence of this fact and of \eqref{eq:E_1_0_F_R}, we may rewrite \eqref{Leibniz_infty_E_2_trace_ref} as
\begin{equation} \label{eq:productrulechi_no_boundary_trace_2} 
\div(\chi_{E} F) = \chi_{E^{1, \cR}} \div F - \Tr{F}{e}{E} |D_{H} \chi_{E}|. 
\end{equation}
Finally, by combining \eqref{eq:productrulechi_no_boundary_trace_1} and \eqref{eq:productrulechi_no_boundary_trace_2}, the equalities \eqref{trace_equality_abs_cont} and \eqref{productrulechi_no_boundary_trace} follow.
\end{proof}

\begin{theorem} [Gauss--Green formula IV] \label{thm:GGIV} Let $F \in \DM^{\infty}(H \Omega)$ and let $E \Subset \Omega$ be a set of finite h-perimeter such that $|\div F|(\mtbR E) = 0$. Then, we have
\begin{equation} \label{G-G_no_boundary} \div F(E^{1, \cR}) = \int_{\redb E} \Tr{F}{}{E} \, d |D_{H} \chi_{E}|. \end{equation}
\end{theorem}
\begin{proof} We just need to evaluate \eqref{productrulechi_no_boundary_trace} on $\Omega$ and we apply Lemma \ref{DMcomptsupp} to get \eqref{G-G_no_boundary}. 
\end{proof}

In analogy to the case $|\div F| \ll \mu$, we can obtain similar integration by parts formula and Green's identities in the case $|\div F|(\mtbR E) = 0$, with $E^{1, \cR}$ instead of $E$ in the integration with respect to the divergence and the Laplacian measure, respectively.

\begin{theorem}[Integration by parts II] Let $F \in \DM^{\infty}_{\rm loc}(H \Omega)$, $E$ be a set of locally finite h-perimeter in $\Omega$ such that $|\div F|(\mtbR E) = 0$, and let $\varphi \in C(\Omega)$ with $\nabla_{H} \varphi \in L^{1}_{\rm loc}(H \Omega)$ such that $\mathrm{supp}(\varphi \chi_{E}) \Subset \Omega$. Then there exists a unique normal trace $\Tr{F}{}{E} \in L^{\infty}_{\rm loc}(\Omega; |D_{H} \chi_{E}|)$ of $F$, such that the following formula holds
\begin{equation*} \int_{E^{1, \cR}} \varphi \, d \div F + \int_{E} \ban{F, \nabla_{H} \varphi} \, dx = \int_{\redb E} \varphi \Tr{F}{}{E} \, d |D_{H} \chi_{E}|. \end{equation*}
\end{theorem}
\begin{proof} Proposition \ref{G-G_groups_no_boundary} implies that, if $|\div F|(\mtbR E) = 0$, then we have \eqref{trace_equality_abs_cont}, \eqref{density_pairing_unique_trace} and \eqref{eq:E_1_0_F_R}. One needs just to combine these results with Theorem \ref{IBP_general}.
\end{proof}

\begin{theorem}[Green's identities III]
Let $u \in C^1_{H}(\Omega)$ satisfy $\Delta_{H} u \in \mathcal{M}_{\rm loc}(\Omega)$ and let $E \subset \Omega$ be a set of locally finite h-perimeter in $\Omega$ such that $|\Delta_{H} u|(\mtbR E) = 0$. Then for each $v \in C_{c}(\Omega)$ with $\nabla_{H} v \in L^{1}(H \Omega)$ one has
\begin{equation*}
\int_{E^{1, \cR}} v \, d\Delta_{H} u  = \int_{\redb E} v \, \ban{\nabla_{H} u, \nu_{E}} \, d |D_{H} \chi_{E}| - \int_{E} \ban{\nabla_{H} v, \nabla_{H} u} \, dx. \end{equation*}
If $u, v \in C^1_{H, c}(\Omega)$ also satisfy $\Delta_{H} u, \Delta_{H} v \in \mathcal{M}(\Omega)$, $|\Delta_{H} u|(\mtbR E) = |\Delta_{H} v|(\mtbR E) = 0$, one has
\begin{equation*} 
\int_{E^{1, \cR}} v \, d\Delta_{H} u  - u \, d\Delta_{H} v  = \int_{\redb E} \ban{v \nabla_{H} u - u \nabla_{H} v, \nu_{E}} \, d |D_{H} \chi_{E}|. \end{equation*}
If $E \Subset \Omega$, one can drop the assumption that $u$ and $v$ have compact support in $\Omega$.
\end{theorem}
\begin{proof} 
It suffices to combine the results of Theorem~\ref{Green_identity} with the fact that, by Proposition~\ref{G-G_groups_no_boundary}, $E^{1}_{u} = E^{1}_{v} = E^{1, \cR}$ up to $|\Delta_{H} u|, |\Delta_{H} v|$-negligible sets.
\end{proof}

As an easy consequence of 
Theorem~\ref{normal_trace_locality_theorem}, we obtain the same locality property for the normal trace in the case $|\div F| \ll \mu$ and $|\div F|(\mtbR E) = 0$.

\begin{proposition} Let $F \in \DM^{\infty}(H \Omega)$, and $E_{1}, E_{2} \subset \Omega$ be sets of finite h-perimeter such that $\SHaus{Q - 1}(\redb E_{1} \cap \redb E_{2}) > 0$ and $|\div F|(\mtbR E_{j}) = 0$, for $j = 1, 2$. Then, we have
\begin{equation} \label{same_or_unique_normal_trace} \Tr{F}{}{E_{1}}(x) = \Tr{F}{}{E_{2}}(x), \end{equation}
for $\SHaus{Q - 1}$-a.e.\ $x \in \{ y \in \redb E_{1} \cap \redb E_{2} : \nu_{E_{1}}(y) = \nu_{E_{2}}(y) \}$, and
\begin{equation} \label{opposite_or_unique_normal_trace} \Tr{F}{}{E_{1}}(x) = - \Tr{F}{}{E_{2}}(x), \end{equation}
for $\SHaus{Q - 1}$-a.e.\ $x \in \{ y \in \redb E_{1} \cap \redb E_{2} : \nu_{E_{1}}(y) = - \nu_{E_{2}}(y) \}$.
\end{proposition}
\begin{proof} The result follows immediately from Theorem \ref{normal_trace_locality_theorem} and \eqref{trace_equality_abs_cont}, which holds by Proposition \ref{G-G_groups_no_boundary}.
\end{proof}

\section{Applications to sets of Euclidean finite perimeter}

The underlying linear structure of $\G$ allows for introducing 
an Euclidean scalar product, for instance using a fixed system of
graded coordinates, see Section~\ref{sect:metric}.
With respect to this metric structure the classical
sets of finite perimeter can be considered. 
We will call them sets of {\em Euclidean finite perimeter}
to make a precise distinction with sets of finite h-perimeter.

If $E\subset\G$ is a set of locally finite Euclidean perimeter in
$\Omega\subset\G$ and $F\in \DM^{\infty}_{\rm loc}(H \Omega)$, then we can refine \eqref{G-G groups 1} and \eqref{G-G groups 2} employing the theory of divergence-measure fields in Euclidean space. From the Euclidean Leibniz rule for essentially bounded divergence-measure fields (\cite[Theorem~2.1]{Frid} of Frid), the uniqueness of the representative $\tildef{g}$ in Theorem \ref{productrule} and of the pairing measure follows.

Recall that we can identify $\G$ with $\R^{\q}$, where $\q$ is the topological dimension of $\G$. In this section, we shall denote the Euclidean norm by $|\cdot|$, and the Riemannian one by $|\cdot|_{g}$. The $L^{\infty}$-norm $\|\cdot\|_{\infty, \Omega}$ for horizontal fields is the same defined in \eqref{Linfty_norm} using $|\cdot|_{g}$. 

We denote the Euclidean Hausdorff measure by 
$\Haus{\alpha}_{|\cdot|}$ and the Euclidean ball by
\begin{equation*} 
B_{|\cdot|}(x, r) := \{ y \in \R^{\q} : |x - y| < r \}. 
\end{equation*}
Consequently, given $u \in L^{1}_{\rm loc}(\G)$, we denote by 
\begin{equation} \label{Euclidean_representative}
u^{*}_{|\cdot|}(x) := \begin{cases} \displaystyle \lim_{r \to 0} \mean{B_{|\cdot|}(x,r)} u(y) \, dy & \mbox{ if the limit exists, } \\ 0 &  \mbox{ otherwise,} \end{cases} \end{equation}
the {\em Euclidean precise representative} of $u$.

The following useful lemma is a consequence of
the rectifiability of the reduced boundary and of the negligibility
of characteristic points \cite{magnani2006characteristic}. 
Its proof can be found in \cite{cassano2016some}.
For the ease of the reader, we add a short proof.

\begin{lemma} \label{Euclidean_set_fin_per} If $E$ is a set of Euclidean locally finite perimeter in $\Omega$ and if we denote by $\redbeu E$ the Euclidean reduced boundary, we have $\SHaus{Q - 1}(\redb E \Delta \redbeu E) = 0$.
\end{lemma}
\begin{proof}
By Theorem \ref{perimeter repr}, Lemma \ref{difference redb mtb} and \cite[Proposition 5.11]{cassano2016some}, we know that 
\begin{equation} \label{perimeter_Euclidean_identity} 
|D_{H} \chi_{E}| = |\pi_{H} N_{E}|\, \Haus{\q - 1}_{|\cdot|} \res \redbeu E = 
\theta_{E}\, \SHaus{Q - 1} \res \redb E, \end{equation}
where $\pi_{H} N_{E}$ is the projection of the Euclidean measure theoretic exterior normal $N_{E}$ on the horizontal bundle of $\G$. Hence, since $\theta_{E} \ge \alpha > 0$ by Theorem \ref{perimeter repr}, we get
$$\SHaus{Q - 1}(\redb E \setminus \redbeu E) = 0, $$
In addition, \cite[Proposition 5.11]{cassano2016some} implies also that the set $${\rm Char}(E):= \{x \in \redbeu E : \pi_{H} N_{E}(x) = 0 \}$$ is 
$\SHaus{Q - 1}$-negligible.
Therefore, by \eqref{perimeter_Euclidean_identity} we have
\begin{equation*} 
\Haus{\q - 1}_{|\cdot|}(\redbeu E \setminus (\redb E \cup {\rm Char}(E))) = 0, \end{equation*}
which implies $\SHaus{Q - 1}(\redbeu E \setminus (\redb E \cup {\rm Char}(E))) = 0$ by \cite[Proposition 4.4]{FSSC5}. Since
${\rm Char}(E)$ is $\SHaus{Q - 1}$-negligible, 
the proof is complete.
\end{proof}
We now recall the Euclidean Leibniz rule for essentially bounded divergence-measure fields in a stratified group.
\begin{theorem} \label{product_rule_Euclidean} 
Let $F\in \DM^{\infty}(H \Omega)$ and $g \in L^{\infty}(\Omega)$ be such that for every $j = 1, \dots, \q$ we have $\der_{x_{j}} g \in \mathcal{M}(\Omega)$.  It follows that $g F \in \DM^{\infty}(H \Omega)$ and 
\begin{equation} \label{product_rule_Euclidean_eq} \div(g F) = g^{*}_{|\cdot|} \div F + (F, D g), \end{equation}
where $(F, D g)$ is the weak$^*$ limit of $\ban{F,D(\rho_{\eps} \tilde{\ast} g)}_{\R^\q} \mu$ as $\eps \to 0$, denoting by $\tilde{\ast}$ the Euclidean convolution product, by $\rho \in C^{\infty}_{c}(B_{|\cdot|}(0, 1))$ a radially symmetric mollifier with $\int \rho \, dx = 1$
and $\rho_{\eps}(x) = \eps^{-\q} \rho(x/\eps)$.
\end{theorem}
\begin{proof}
Since $F \in \DM^{\infty}(H \Omega) \subset \DM^{\infty}(\Omega)$ and $g$ is an essentially bounded function of Euclidean bounded variation, \cite[Theorem~2.1]{Frid} shows that $g F \in \DM^{\infty}(\Omega)$ and that we have \eqref{product_rule_Euclidean_eq}. Then we clearly have $g F \in \DM^{\infty}(H\Omega)$, since $F$ is a measurable horizontal section.
\end{proof}
We stress the fact that $g \in BV(\Omega)$ in general does not imply $g \in BV_{H}(\Omega)$, unless the set $\Omega$ is bounded. Since a function of Euclidean bounded variation on $\Omega$ belongs only to $BV_{H, {\rm loc}}(\Omega)$, we shall need to localize all the following statements. 

\begin{theorem} \label{abs_continuity_pairing_infty_Euclidean} Let $F\in \DM^{\infty}(H \Omega)$ and $g \in L^{\infty}(\Omega)$ be such that $\der_{x_{j}} g \in \mathcal{M}(\Omega)$ for $j = 1, \dots, \q$. Then, the measure $(F, D g)$ satisfies 
\begin{equation} \label{abs_cont_pairing_infty_Euclidean}
|(F, D g)| \res U \le \|F\|_{\infty, U} |D_{H} g| \res U, 
\end{equation} 
for any open bounded set $U \subset \Omega$.
\end{theorem}
\begin{proof}
Without loss of generality, we may assume $\Omega$ to be bounded, which means that we have $g \in L^{\infty}(\Omega) \cap BV_{H}(\Omega)$.
By Theorem~\ref{product_rule_Euclidean}, we know that 
\[
\ban{F,\nabla(\rho_{\eps} \tilde{\ast} g)}_{\R^\q} \mu\weakto (F,D g)
\]
as $\eps \to 0$.
By \eqref{eq:Ff_j} one easily observes that
$\ban{F,\nabla (\rho_\ep \tilde{\ast} g)}_{\R^\q}=\ban{F,\nabla_H(\rho_\ep \tilde{\ast} g)}$
and this means that
\begin{equation} \label{Euclidean_pairing_horizontal} \ban{F,\nabla_H(\rho_{\eps} \tilde{\ast} g)} \mu \weakto (F, D g). \end{equation}
We notice that, for any $\phi \in C_{c}(\Omega)$, we have
\begin{align*} \limsup_{\eps \to 0} \left | \int_{\Omega} \phi \ban{F, (\rho_\ep \tilde{\ast} D_{H} g)} \, dx \right | & \le \limsup_{\eps \to 0} \|F\|_{\infty, \Omega} \int_{\Omega} |\phi| \rho_{\eps} \tilde{\ast} |D_{H} g| \, dx \\
& = \|F\|_{\infty, \Omega} \int_{\Omega} |\phi| \, d |D_{H} g|, \end{align*}
by well-known properties of Euclidean convolution of measures (see \cite[Theorem~2.2]{AFP}). Now we show that 
\begin{equation} \label{comm_est_Euclidean_pairing} \ban{F, \nabla_{H} (\rho_\ep \tilde{\ast} g)} \mu - \ban{F, (\rho_\ep \tilde{\ast} D_{H} g)} \mu \weakto 0. \end{equation} 
Indeed, if \eqref{comm_est_Euclidean_pairing} holds, then, for any $\phi \in C_{c}(\Omega)$, by \eqref{Euclidean_pairing_horizontal} we have
\begin{align*} \left | \int_{\Omega} \phi \, d (F, D g) \right | & = \lim_{\eps \to 0} \left | \int \phi \ban{F, \nabla_{H} (\rho_{\eps} \tilde{\ast} g)} \, dx \right | \\
& \le \limsup_{\eps \to 0} \left | \int \phi \ban{F, \nabla_{H} (\rho_{\eps} \tilde{\ast} g) - (\rho_{\eps} \tilde{\ast} D_{H} g)} \, dx \right |\\
& + \limsup_{\eps \to 0} \left | \int \phi \ban{F, \rho_{\eps} \tilde{\ast} D_{H} g} \, dx \right | \\
& \le \|F\|_{\infty, \Omega} \int_{\Omega} |\phi| \, d |D_{H} g|, \end{align*}
which implies \eqref{abs_cont_pairing_infty_Euclidean}. 
Therefore, we need to show a commutation estimate.
We recall the fact that $|a_{j}^{i}(x) - a_{j}^{i}(y)| \le C |x - y|$ on compact sets, for any $j = 1, \dots, \m$ and $i = \m + 1, \dots, \q$. Hence, for any $x \in \Omega$ and $\eps > 0$ such that $B_{|\cdot|}(x, \eps) \subset \Omega$, 
the equality between the modulus of the sum
\[
\left | \sum_{i=\m+1}^\q a_j^i(x) (\rho_\ep \tilde{\ast} \der_{y_{i}} g)(x) - \rho_{\eps} \tilde{\ast} (a_{j}^{i} \der_{y_{i}} g)(x) \right | 
\]
and its more explicit version 
\[
\left | \sum_{i=\m+1}^\q \int_{B_{|\cdot|}(x, \eps)} (a_j^i(x) - a_{j}^{i}(y)) \rho_\ep(x - y) \, d \der_{y_{i}} g(y) \right | 
\]
leads us to the inequality 
\begin{equation} \label{commutation_estimate_Euclidean} \left | \sum_{i=\m+1}^\q a_j^i(x) (\rho_\ep \tilde{\ast} \der_{y_{i}} g)(x) - \rho_{\eps} \tilde{\ast} (a_{j}^{i} \der_{y_{i}} g)(x) \right |  \le C \| \rho \|_{\infty, B_{|\cdot|}(0, 1)} \frac{|D g|(B_{|\cdot|}(x, \eps))}{\eps^{\q - 1}}. 
\end{equation}
We now take $\phi \in C_{c}(\Omega)$ and we employ the fact that $\der_{x_{j}} (\rho_\ep \tilde{\ast} g) = (\rho_\ep \tilde{\ast} \der_{x_j} g)$, for any $j = 1, \dots, \q$, to obtain
\begin{align*} \left | \int_{\Omega} \ban{\phi F, \nabla_{H} (\rho_\ep \tilde{\ast} g) - (\rho_\ep \tilde{\ast} D_{H} g)} \, dx \right | & = \left | \int_{\Omega} \phi \sum_{j=1}^\m F_j\,\pa{\sum_{i=\m+1}^\q a_j^i (\rho_\ep \tilde{\ast} \der_{x_{i}} g) - \rho_{\eps} \tilde{\ast} (a_{j}^{i} \der_{x_{i}} g)  } \, dx \right | \\
& \le C \|F\|_{\infty, \Omega} \| \rho \|_{\infty, B_{|\cdot|}(0, 1)} \int_{\Omega} |\phi(x)| \frac{|D g|(B_{|\cdot|}(x, \eps))}{\eps^{\q - 1}} \, dx
\end{align*}
by \eqref{commutation_estimate_Euclidean}. 
Let now $\eps > 0$ be smal enough so that 
\[
\mathrm{supp}(\phi) \subset  \Omega^{\eps}:= \{ x \in \Omega : \dist_{|\cdot|}(x, \partial \Omega) > \eps \}.
\]
It follows that 
\begin{align*} \int_{\Omega} |\phi(x)| \frac{|D g|(B_{|\cdot|}(x, \eps))}{\eps^{\q - 1}} \, dx & = \int_{\Omega^{\eps}} \int_{B_{|\cdot|}(x, \eps)} |\phi(x)| \eps^{1 - q} \, d |D g|(y) \, dx \\
& = \int_{\Omega} \int_{B_{|\cdot|}(y, \eps)}  |\phi(x)| \eps^{1 - q} \, dx \, d |D g|(y) \\
& \le \mu\big(B_{|\cdot|}(0, 1)\big)\, \eps\, \|\phi\|_{\infty, \Omega} |D g|(\Omega). \end{align*}
We finally conclude that
\begin{equation*} \int_{\Omega} \phi(x) \frac{|D g|(B_{|\cdot|}(x, \eps))}{\eps^{\q - 1}} \, dx \to 0. \end{equation*}
All in all, \eqref{comm_est_Euclidean_pairing} follows, and this ends the proof of \eqref{abs_cont_pairing_infty_Euclidean}.

\end{proof}

Thanks to the Leibniz rule of Theorem \ref{product_rule_Euclidean} and to its refinement given in Theorem \ref{abs_continuity_pairing_infty_Euclidean}, we are able to obtain the Gauss--Green formulas for Euclidean sets of finite perimeter in stratified groups. Even though such results could be proved directly, using \eqref{abs_cont_pairing_infty_Euclidean} and employing techniques very similar to those presented in \cite[Theorems 3.2, 4.1, 4.2]{comi2017locally}, we shall instead first show that, in the case of a set of Euclidean finite perimeter $E$, the pairing $(F, D_{H} \chi_{E})$ defined in Theorem \ref{productrule} actually coincides with $(F, D \chi_{E})$ introduced in Theorem \ref{product_rule_Euclidean}. Then, the Gauss--Green formulas will be just a consequence of Theorem \ref{Leibniz_infty_E_trace}.

Let us denote by $E^{1}_{|\cdot|}$ and $E^{0}_{|\cdot|}$ the Euclidean measure theoretic interior and exterior of a measurable set $E \subset \Omega$; that is,
\begin{align*}
E^{1}_{|\cdot|}&= \left \{ x \in \Omega : \lim_{r \to 0} \frac{\mu(B_{|\cdot|}(x, r) \cap E)}{\mu(B_{|\cdot|}(x, r))} = 1 \right \}, \\
E^{0}_{|\cdot|}&= \left \{ x \in \Omega : \lim_{r \to 0} \frac{\mu(B_{|\cdot|}(x, r) \cap E)}{\mu(B_{|\cdot|}(x, r))} = 0 \right \}. 
\end{align*}
We recall now that, if $E$ is a set of Euclidean finite perimeter and we denote by $\redbeu E$ the Euclidean reduced boundary, then 
\begin{equation} \label{precise_repr_char_Eu} \chi_{E, |\cdot|}^{*} = \chi_{E^{1}_{|\cdot|}} + \frac{1}{2} \chi_{\redbeu E}, \end{equation}
see for instance \cite[Lemma 2.13]{comi2017locally} and the references therein. 

We proceed now to show that, in the case $g = \chi_{E}$ for an Euclidean set of finite perimeter $E$, the Euclidean Leibniz rule is equivalent to the group one.

\begin{theorem} \label{equivalence_Leibniz_rule_Eucl} Let $F \in \DM^{\infty}(H \Omega)$ and $E$ be a set of Euclidean finite perimeter in $\Omega$. Then we have $\chi_{E} F \in \DM^{\infty}(H \Omega)$, 
\begin{equation} \label{id_local_pairings} (F, D \chi_{E}) = (F, D_{H} \chi_{E}), \end{equation} 
and the following Leibniz rule holds:
\begin{equation} \label{Leibniz_set_Eucl_H} \div(\chi_{E} F) = \chi_{E, |\cdot|}^{*} \div F + (F, D_{H} \chi_{E}). \end{equation}
In addition, for any $\rho \in C_{c}(B(0, 1))$ satisfying $\rho \ge 0$, $\rho(x) = \rho(-x)$, $\int_{B(0, 1)} \rho(x) \, dx = 1$, we have $\rho_{\eps} \ast \chi_{E} \weakstarto \chi_{E, |\cdot|}^{*}$ in $L^{\infty}(\Omega; |\div F|)$ and $\ban{F, \nabla_{H} (\rho_{\eps} \ast \chi_{E})} \mu \weakto (F, D \chi_{E})$ in $\mathcal{M}(\Omega)$. In particular, $\intee{E}{F} = E^{1}_{|\cdot|}$ and $\extee{E}{F} = E^{0}_{|\cdot|}$, up to $|\div F|$-negligible sets.
\end{theorem}
\begin{proof} 
It is easy to see that $\chi_{E} F \in L^{\infty}(H \Omega)$, hence \eqref{product_rule_Euclidean_eq} with $g = \chi_{E}$ yields 
\begin{equation} \label{Leibniz_set_Eucl} \div(\chi_{E} F) = \chi_{E, |\cdot|}^{*} \div F + (F, D \chi_{E}), \end{equation}
which gives $\chi_{E} F \in \DM^{\infty}(H \Omega)$. Notice that this fact would not follow directly from Theorem \ref{productrule}, since in our assumptions the h-perimeter of $E$ is only locally finite.
Let us assume now $\Omega$ to be bounded, so that $E$ is a set of finite h-perimeter in $\Omega$.
By \eqref{Leibniz_infty_E_1}, we immediately obtain
\begin{equation} \label{two_rules} \chi_{E, |\cdot|}^{*} \div F + (F, D \chi_{E}) = \tildef{\chi_{E}} \div F + (F, D_{H} \chi_{E}), 
\end{equation}
where $\tildef{\chi_{E}}$ and $(F, D_{H} \chi_{E})$ are unique, thanks to Theorem \ref{Uniqueness_traces} and Remark \ref{rem:unique_pairing}.
This means that 
\begin{equation} \label{pairing_D_H_chi_E_difference} (F, D_{H} \chi_{E}) = \chi_{E, |\cdot|}^{*} \div F + (F, D \chi_{E}) - \tildef{\chi_{E}} \div F. \end{equation}
Hence, if $\Omega$ is unbounded, we get \eqref{two_rules} and \eqref{pairing_D_H_chi_E_difference} in the sense of Radon measures on any bounded open set $U \subset \Omega$. Taking into account Theorem \ref{product_rule_Euclidean}, we know that  
\[
\div F \in \mathcal{M}(\Omega), \quad \tildef{\chi_{E}},\chi_{E, |\cdot|}^{*} \in L^{\infty}(\Omega; |\div F|)\quad
\text{and}\quad (F, D \chi_{E}) \in \mathcal{M}(\Omega),
\]
therefore the right hand side of \eqref{pairing_D_H_chi_E_difference} is a finite Radon measure on $\Omega$. In other words, we have proved that $(F, D_{H} \chi_{E}) \in \mathcal{M}(\Omega)$, even if $E$ is only a set of locally finite h-perimeter in $\Omega$. Hence, \eqref{two_rules} holds on the whole $\Omega$.
We recall now that, by Lemma~\ref{Euclidean_set_fin_per}, $\SHaus{Q - 1}(\redb E \Delta \redbeu E) = 0$, which implies 
\begin{equation} \label{div_redb_E_Delta_redbeu_E} |\div F|(\redb E \Delta \redbeu E) = 0, \end{equation} 
by Theorem \ref{absolute continuity}. 
Now, we employ \eqref{chiEtildeUniq} in order to rewrite \eqref{two_rules} as follows
\begin{equation} \label{eq:three_rules}
\left ( \chi_{E^{1}_{|\cdot|}} - \chi_{\intee{E}{F}} \right ) \div F = \frac{1}{2} \left ( \chi_{\redb E} - \chi_{\redbeu E} \right ) \div F + (F, D_{H} \chi_{E}) - (F, D \chi_{E}).
\end{equation}
Thanks to \eqref{div_redb_E_Delta_redbeu_E}, we have
\begin{equation*}
\left ( \chi_{\redb E} - \chi_{\redbeu E} \right ) \div F = 0,
\end{equation*}
so that \eqref{eq:three_rules} reduces to
\begin{equation} \label{eq:four_rules}
\left ( \chi_{E^{1}_{|\cdot|}} - \chi_{\intee{E}{F}} \right ) \div F = (F, D_{H} \chi_{E}) - (F, D \chi_{E}).
\end{equation}
If we restrict \eqref{eq:four_rules} to $\redb E$, we obtain
\begin{equation} \label{eq:five_rules}
\chi_{E^{1}_{|\cdot|} \cap \redb E} \div F = (F, D_{H} \chi_{E}) - (F, D \chi_{E}),
\end{equation}
since $\intee{E}{F} \subset \Omega \setminus \redb E$, by Definition \ref{def:measthinterior}, and $|(F, D_{H} \chi_{E})|, |(F, D \chi_{E})| \ll |D_{H} \chi_{E}|$, by \eqref{abs_cont_pairing_infty} and \eqref{abs_cont_pairing_infty_Euclidean}. We notice that
\begin{align*}
|\div F|(E^{1}_{|\cdot|} \cap \redb E) & = |\div F|(E^{1}_{|\cdot|} \cap (\redb E \cap \redbeu E)) + |\div F|(E^{1}_{|\cdot|} \cap (\redb E \setminus \redbeu E)) \\
& \le |\div F|( \redb E \setminus \redbeu E) = 0,
\end{align*}
by \eqref{div_redb_E_Delta_redbeu_E} and the fact that $E^{1}_{|\cdot|} \cap \redbeu E = \emptyset$. As a consequence, we get
$$\chi_{E^{1}_{|\cdot|} \cap \redb E} \div F = 0,$$ 
which, joined with \eqref{eq:five_rules}, yields the equality of measures \eqref{id_local_pairings}. This equality, along with 
\eqref{Leibniz_set_Eucl}, leads us to \eqref{Leibniz_set_Eucl_H}.
Thus, combining \eqref{two_rules} and \eqref{id_local_pairings}, we obtain 
\begin{equation*}
\left ( \chi_{E, |\cdot|}^{*} - \tildef{\chi_{E}} \right ) \div F = 0,
\end{equation*}
which immediately yields $\tildef{\chi_{E}}(x) = \chi_{E, |\cdot|}^{*}(x)$ for $|\div F|$-a.e. $x \in \Omega$. As a consequence, we get $|\div F|(\intee{E}{F} \Delta E^{1}_{|\cdot|}) = 0$ and $|\div F|(\extee{E}{F} \Delta E^{0}_{|\cdot|}) = 0$.
\end{proof}

\begin{remark}
By Theorem \ref{equivalence_Leibniz_rule_Eucl}, $\chi_{E} F, \chi_{\Omega \setminus E} F \in \DM^{\infty}(H \Omega)$ for any $F \in \DM^{\infty}(H \Omega)$ and any set $E$ of Euclidean finite perimeter in $\Omega$. This means that, by \eqref{id_local_pairings}, we have 
\begin{equation*} (\chi_{E} F, D \chi_{E}) = (\chi_{E}F, D_{H} \chi_{E}) \ \text{and} \ (\chi_{\Omega \setminus E} F, D \chi_{E}) = (\chi_{\Omega \setminus E}F, D_{H} \chi_{E}). \end{equation*} 
Thus, we can define the normal traces of $F$ on the reduced boundary of an Euclidean set of finite perimeter as in \eqref{interior_normal_trace_def} and \eqref{exterior_normal_trace_def}. We wish to stress that, as a simple consequence of Theorem~\ref{Uniqueness_traces}, the measures  $(\chi_{E} F, D_{H} \chi_{E})$ and $(\chi_{\Omega \setminus E} F, D_{H} \chi_{E})$ do not depend on the vanishing sequence $\eps_{k}$, see Remark~\ref{rem:unique_pairing} for more details.
\end{remark}

These results enable us to prove Gauss--Green formulas for sets of Euclidean finite perimeter, Theorem~\ref{Gauss_Green_Euclidean}, extending \cite[Theorem 4.2]{comi2017locally} to all geometries of stratified groups.

\begin{proof}[Proof of Theorem~\ref{Gauss_Green_Euclidean}]
If we choose $U \Subset \Omega$, it is clear that $E$ is a set of finite Euclidean perimeter in $U$, and so of finite h-perimeter in $U$. By Theorem \ref{equivalence_Leibniz_rule_Eucl}, we know that, up to $|\div F|$-negligible sets, $\intee{E}{F} = E^{1}_{|\cdot|}$. Then, \eqref{Leibniz_Eu_infty_1}, \eqref{Leibniz_Eu_infty_2} and \eqref{Leibniz_Eu_infty_3} follow immediately from \eqref{Leibniz_infty_E_1_trace_ref}, \eqref{Leibniz_infty_E_2_trace_ref} and \eqref{boundary_term_div_1_2_ref}. 
The estimates on the normal traces
are a consequence of  \eqref{interior_normal_trace_norm_loc} and \eqref{exterior_normal_trace_norm_loc}
in Theorem~\ref{IBP_general},
since assumptions imply that $E$ is also a set of 
finite h-perimeter on any bounded open set of $\Omega$. The same theorem shows that \eqref{IBP_1_Eu} and \eqref{IBP_2_Eu} are a consequence of \eqref{IBP_1} and \eqref{IBP_2}, taking into account that $|\div F|(\intee{E}{F} \Delta E^{1}_{|\cdot|}) = 0$. Thus, we conclude the proof.
\end{proof}

\begin{remark} The normal traces of $F$ on the reduced boundary of an Euclidean set of finite perimeter $E$ satisfy the same locality property stated in Theorem \ref{normal_trace_locality_theorem}. As a byproduct, we have also provided an alternate proof of the locality of normal traces on reduced boundaries of Euclidean sets of finite perimeter. Such proof does not employ De Giorgi's blow-up theorem, which was essential in \cite[Proposition 4.10]{comi2017locally}.
\end{remark}

Arguing as for Theorem \ref{Gauss_Green_Euclidean}, we can provide a generalization of Green's identities to stratified groups for sets of Euclidean locally finite perimeter, Theorem~\ref{t:GreenIV}, which extends the result of \cite[Proposition 4.5]{comi2017locally} to stratified groups.

\begin{proof}[Proof of Theorem~\ref{t:GreenIV}]
It suffices to combine the results of Theorem~\ref{Green_identity} with the case of a set of Euclidean finite perimeter. By Theorem \ref{equivalence_Leibniz_rule_Eucl}, we know that, up to $|\Delta_{H} u|$-negligible sets, $E^{1}_{u} = E^{1}_{|\cdot|}$, and so we get \eqref{first_Green_Euclidean}. The same is clearly true up to $|\Delta_{H} v|$-negligible sets, and this concludes the proof.
\end{proof}

{\bf Acknowledgements.} It is our pleasure to thank Luigi Ambrosio for his generous comments and fruitful discussions.

\end{document}